\documentclass{amsart}
\usepackage[T1]{fontenc}
\usepackage{lmodern}
\usepackage{verbatim}
\usepackage{amsthm}
\usepackage{caption}
\usepackage{amstext}
\usepackage{booktabs}
\usepackage{graphicx}
\usepackage{epstopdf} 
\usepackage{enumitem}
\usepackage{tikz-cd}
\usepackage{comment}

\makeatletter
\theoremstyle{plain}
\newtheorem{thm}{\protect\theoremname}[section]
  \theoremstyle{plain}
  \newtheorem{lem}[thm]{\protect\lemmaname}
  \newtheorem{lemma}[thm]{\protect\lemmaname}
  \theoremstyle{definition}
  \newtheorem{defn}[thm]{\protect\definitionname}
  \theoremstyle{plain}
  \newtheorem{cor}[thm]{\protect\corollaryname}
   \theoremstyle{plain}
  \newtheorem{prop}[thm]{\protect\propositionname}
   
   \theoremstyle{remark}
  \newtheorem{rem}[thm]{\protect\remarkname}

\theoremstyle{remark}
\newtheorem*{exam*}{Example}
\newtheorem*{acknowledgement*}{Acknowledgement}
\newtheorem*{claim}{Claim}
\newtheorem*{rem*}{Remark}

\usepackage{xcolor}

\renewcommand{\epsilon}{\varepsilon}
\renewcommand{\phi}{\varphi}

\DeclareMathOperator{\Irr}{Irr}
\DeclareMathOperator{\irr}{Irr}
\DeclareMathOperator{\Ind}{Ind}

\DeclareMathOperator{\Stab}{Stab}
\DeclareMathOperator{\tr}{tr}

\DeclareMathOperator{\im}{Im}
\DeclareMathOperator{\Ker}{Ker}
\DeclareMathOperator{\Hom}{Hom}

\DeclareMathOperator{\GL}{GL}
\DeclareMathOperator{\SL}{SL}
\DeclareMathOperator{\M}{M}

\DeclareMathOperator{\sym}{Sym}

\newcommand{\df}{\,\mathrm{d}}
\newcommand{\ca}{\mathfrak{h}}
\newcommand{\zetareg}{\zeta^{\mathrm{reg}}}
\newcommand{\zetairr}[1]{\zeta^{(#1)}}
\DeclareMathOperator{\twirr}{\widetilde{\Irr}}
\DeclareMathOperator{\mult}{mult}

\DeclareMathOperator{\chara}{char} 

\newcommand{\iso}{\mathbin{\kern.15em\widetilde{\hphantom{\hspace{.6em}}}
\kern-.98em\rightarrow\kern.05em}}
\newcommand{\longiso}{\mathbin{\kern.3em\widetilde{\hphantom{\hspace{1.1em}}}
\kern-1.55em\longrightarrow\kern.1em}}

\newcommand{\mfp}{\mathfrak{p}}

\newcommand{\C}{\ensuremath{\mathbb{C}}}
\newcommand{\F}{\ensuremath{\mathbb{F}}}

\newcommand{\N}{\ensuremath{\mathbb{N}}}
\newcommand{\Q}{\ensuremath{\mathbb{Q}}}
\newcommand{\R}{\ensuremath{\mathbb{R}}}

\newcommand{\Z}{\ensuremath{\mathbb{Z}}}

\newcommand{\bfa}{\ensuremath{\mathbf{a}}}

\newcommand{\bfG}{\ensuremath{\mathbf{G}}}

\newcommand{\cO}{\ensuremath{\mathcal{O}}}

\newcommand{\GLnO}{\ensuremath{\mathrm{GL}_n(\mathcal{O})}}
\newcommand{\SLnO}{\ensuremath{\mathrm{SL}_n(\mathcal{O})}}

\newcommand{\twistGLn}{\tilde{\zeta}_{\mathrm{GL}_n(\mathcal{O})}}
\newcommand{\twistGLtwo}{\tilde{\zeta}_{\mathrm{GL}_2(\mathcal{O})}}
\newcommand{\twistGLR}{\tilde{\zeta}_{\mathrm{GL}_2(\mathcal{O}_R)}}
\newcommand{\twistGLFq}{\tilde{\zeta}_{\mathrm{GL}_2(\F_q)}}
\newcommand{\chigroup}{\hat\cO_r^\times}

\newcommand{\centra}[1]{C_{#1}}
\newcommand{\centb}[1]{\centra{#1}}
\newcommand{\scentb}[1]{S\centb{#1}}

\newcommand{\centbr}{C}
\newcommand{\scentbr}{S\centbr} 
\newcommand{\centbl}{\centb{l'}}
\newcommand{\scentbl}{S\centbl} 
\newcommand{\centbi}{\centb{i}}
\newcommand{\scentbi}{\scentb{i}}

\newcommand{\SK}{K_{\SL}}

\newcommand{\ceil}[1]{\lceil{#1}\rceil}
\newcommand{\floor}[1]{\lfloor{#1}\rfloor}
\newcommand{\inner}[1]{\langle{#1}\rangle}

  \providecommand{\corollaryname}{Corollary}
  \providecommand{\definitionname}{Definition}
  \providecommand{\lemmaname}{Lemma}
  \providecommand{\propositionname}{Proposition}
  \providecommand{\remarkname}{Remark}
\providecommand{\theoremname}{Theorem}

\makeatother

  \providecommand{\corollaryname}{Corollary}
  \providecommand{\lemmaname}{Lemma}
\providecommand{\theoremname}{Theorem}
\begin{document}

\title{Representation growth of compact linear groups}
\author{Jokke Häsä and Alexander Stasinski}
\address{Department of Mathematical Sciences, Durham University, South Rd,
Durham, DH1 3LE, UK}
\email{jokke.hasa@helsinki.fi\\alexander.stasinski@durham.ac.uk}

\date{}

\begin{abstract}
We study the representation growth of simple compact Lie groups and of $\SL_n(\cO)$, where $\cO$ is a compact discrete valuation ring, as well as the twist representation growth of $\GL_n(\cO)$. This amounts to a study of the abscissae of convergence of the corresponding (twist) representation zeta functions. 

We determine the abscissae for a class of Mellin zeta functions which include the Witten zeta functions. As a special case, we obtain a new proof of the theorem of Larsen and Lubotzky that the abscissa of Witten zeta functions is $r/\kappa$, where $r$ is the rank and $\kappa$ the number of positive roots.

We then show that the twist zeta function of $\GL_n(\cO)$ exists and has the same abscissa of convergence as the zeta function of $\SL_n(\cO)$, provided $n$ does not divide $\chara{\cO}$. We compute the twist zeta function of $\GL_2(\cO)$ when the residue characteristic $p$ of $\cO$ is odd, and approximate the zeta function when $p=2$ to deduce that the abscissa is $1$. Finally, we construct a large part of the representations of $\SL_2(\F_q[[t]])$, $q$ even, and deduce that its abscissa lies in the interval $[1,\,5/2]$. 
\end{abstract}

\maketitle

\tableofcontents

\section{Introduction}

For a group $G$, let $r_i(G)$ denote the number of isomorphism classes of irreducible complex representations of $G$ of dimension $i$. When $G$ is a topological group, we only consider continuous representations. If $G$ is such that $r_i(G)$ is finite for every $i\in\N$ and the sequence $r_i(G)$ grows at most polynomially, then the Dirichlet series
\[
\zeta_G(s)=\sum_{i=1}^{\infty}\frac{r_i(G)}{i^s}
\]
converges for all $s$ in a complex right half plane, and we then call $\zeta_G(s)$ the \emph{representation zeta function} of $G$. Groups satisfying the above conditions include arithmetic groups with the congruence subgroup property, as well as compact Lie groups and rational points of semisimple group schemes over compact discrete valuation rings, such as $\SL_n(\Z_p)$ and $\SL_n(\F_q[[t]])$ (see \cite{Lubotzky-Martin} and  Remark~\ref{rem:SL2-polygrowth}).

Representation growth pertains to the asymptotic properties of the sequence $R_N(G)=\sum_{i=1}^{N}r_i(G)$. If $R_N(G)$ grows at most polynomially, the series defining $\zeta_G(s)$ converges for some real $s$, and if $R_N(G)\to\infty$, the series defining $\zeta_G(s)$ diverges for some negative real $s$. It follows that in this case
\[
\limsup_{N\to\infty}\frac{\log R_{N}(G)}{\log N}=\inf\{s\in\R\mid\zeta_{G}(s)\text{ converges}\}
\]
is the \emph{abscissa of convergence} of $\zeta_G(s)$, which we denote by  $\alpha$ or $\alpha(G)$ and sometimes refer to simply as the abscissa of $G$. Thus the series defining
$\zeta_G(s)$ converges in the right half-plane $\mathrm{Re}(s)>\alpha$ and diverges in the left half-plane $\mathrm{Re}(s)<\alpha$ (see e.g., \cite[Theorem~8.2]{Apostol}). 
Moreover, $R_N(G)=O(N^{\alpha+\epsilon})$ for every real $\epsilon>0$, and $\alpha$ is minimal with this property, so the abscissa of convergence controls the rate of representation growth.

A systematic study of the representation growth of arithmetic and compact groups was initiated by Michael Larsen and Alex Lubotzky in \cite{Larsen-Lubotzky}. Among other things, they proved that if $\bfG$ is a sufficiently nice simple group scheme over the ring of integers $\cO_K$ in a number field $K$ such that $\bfG(\cO_K)$ has the Congruence Subgroup Property (CSP), then $\zeta_{\bfG(\cO_K)}(s)$ has an Euler product:
\[
\zeta_{\bfG(\cO_K)}(s)=\zeta_{\bfG(\C)}(s)^{[K:\Q]}\,\prod_{\mfp}\zeta_{\bfG(\cO_{K,\mfp})}(s),
\]
where $\mfp$ runs through the non-zero prime ideals of $\cO_K$ and $\cO_{K,\mfp}$ denotes the completion of $\cO_K$ at $\mfp$  (see \cite[Proposition~1.3]{Larsen-Lubotzky}). Note that each archimedean factor is equal to $\zeta_{\bfG(\C)}(s)$, which counts rational (equivalently, smooth) representations of the Lie group $\bfG(\C)$. It is well known that the representations of $\bfG(\C)$ are in one to one correspondence with those of the compact real form.

In the present paper, we study the abscissae of convergence of local factors of the above Euler product. In addition, we consider some zeta functions $\zeta_{\bfG(\C)}(s)$ and $\zeta_{\bfG(\cO_{K,\mfp})}(s)$ which do not necessarily arise as local factors of an Euler product (this happens for groups which do not satisfy the CSP, for example, $\SL_2(\Z)$).

\subsection*{Witten and Mellin zeta functions}
Zeta functions of the form $\zeta_{\bfG(\C)}(s)$, or more precisely, their meromorphic continuations, are called Witten zeta functions. In \cite{Witten-zeta} Witten related some of their special values to geometric invariants. The Witten zeta functions are defined by series which are special cases of series of the form 
\[
\zeta(P;s)=\sum_{x_1,\dots,x_r=1}^\infty P(x_1,\dots,x_r)^{-s},
\]
where $P\in\C[x_1,\dots,x_r]$ is a polynomial such that the real parts of its coefficients are positive. Such series were first studied by Hjalmar Mellin \cite{Mellin}, who proved that they converge in a half-plane and admit meromorphic continuation to the whole complex plane. We will refer to these meromorphic continuations as \emph{Mellin zeta functions}. Subsequently, several authors considered variants of these functions with different restrictions on the coefficients of $P$ (see, for example, \cite{MahlerMellin}, \cite{Sargos} and \cite{Essouabri}). 
Witten zeta functions are obtained when $P$ is a polynomial originating from the Weyl dimension formula for representation degrees of compact Lie groups. 

One of the main theorems proved by Larsen and Lubotzky \cite[Theorem~5.1]{Larsen-Lubotzky} is that the Witten zeta function $\zeta_{\bfG(\C)}(s)=\zeta(P;s)$ has abscissa of convergence $r/\kappa$, where $r$ is the rank and $\kappa$ is the number of positive roots in the root system associated to $\bfG(\C)$. In fact, here $r$ is also the number of variables of $P$ and $\kappa$ is the degree of $P$. In Section~\ref{sec:Lie groups} we give a new proof of this, and in fact we determine the abscissa for a more general class of Mellin zeta functions. Namely, let $P=P_1\cdots P_\kappa$, where each $P_i$ is a linear polynomial in $r$ variables over $\R$ of the form
\[
P_i=a_{i1}x_1+a_{i2}x_2+\cdots+a_{ir}x_r.
\]
We assume all the coefficients $a_{ij}$ are non-negative and at least one of them is positive for each $i$. One can show that Witten zeta functions are of the form $\zeta(P;s)$ (after the change of variables $x_i\mapsto x_i-1$), but not every $\zeta(P;s)$ comes from a Witten zeta function.

To describe the abscissa of $\zeta(P;s)$, we need information on the structure of the linear factors $P_i$. For any subset of variables $\{x_{j_1},\dots,x_{j_l}\}$, we consider subpolynomials $P[j_1,\dots,j_l]$, which we define as the product of those linear factors of $P$ which contain only the variables $x_{j_1},\dots,x_{j_l}$ with a non-zero 
coefficient. For any subpolynomial $Q$ of $P$, we write $r(Q)$ for $l$, the 
number of variables appearing in $Q$,
and $\kappa(Q)$ for the number of linear factors in $Q$, which is also the degree of $Q$.
Let $S_P$ be the set of all proper subpolynomials of $P$, including the `empty' subpolynomial $1$ with no variables and degree 0. In Theorems~\ref{theorem:abscissa_divergence} and \ref{theorem:abscissa_convergence}, we show that the abscissa of convergence of $\zeta(P;s)$ equals
\[
\max_{Q\in S_P}\frac{r-r(Q)}{\kappa-\kappa(Q)}.
\]

The main idea of the proof (see the proof of Theorem~\ref{theorem:abscissa_convergence}) is the following. First, we partition the domain of summation for the series (i.e.\ the positive orthant) into components. Then, we use the structure of the linear factors $P_i$ to find a certain monomial term $x_1^{e_i}\cdots x_r^{e_r}$ from $P$ for each component separately, to obtain an upper bound for $P(x_1,\dots,x_r)^{-s}$. The idea is that in each component, the variables $x_i$ dominate each other in a certain order, and the monomial is chosen so that the contribution of each variable is maximised in that component.

In the case where $\zeta(P;s)$ is a Witten zeta function, the structure of the linear factors in the polynomial $P$ is translated via the Weyl dimension formula into the structure of the root system. We obtain the result of Larsen and Lubotzky as a special case of our general result by applying the following property of irreducible root systems: taking the ratio $r/\kappa$ for all Levi subsystems of the root system, the minimal value is obtained for the full system (Lemma~\ref{lemma:2.5}). This result is not original, as it is used in the proof in \cite{Larsen-Lubotzky}. However, having considered the more general zeta functions $\zeta(P;s)$, we are able to show that this property is in fact necessary for the result: if the polynomial $P$ does not have the corresponding property with respect to its substructure, the abscissa will be strictly greater than $r/\kappa$ (cf.\ Corollary~\ref{corollary:2.5}).

\subsection*{Zeta functions of $\SL_n$ and $\GL_n$ over compact local rings}

In the remainder of the paper we study the abscissae of convergence of certain zeta functions $\zeta_{\SL_{n}(\cO_{K,\mfp})}(s)$, including the case where $\cO_K$ is the ring of integers in a global function field $K$. Prior to the present work, most results have been in the number field case, as this case allows the use of the Kirillov orbit method. Let $\cO$ denote one of the rings $\cO_{K,\mfp}$, that is, $\cO$ is any compact discrete valuation ring with residue field of characteristic $p$. Andrei Jaikin-Zapirain \cite[Theorem~7.5]{Jaikin-zeta} showed that  $\alpha(\SL_2(\cO))=1$ when $p\neq 2$. Nir Avni, Benjamin Klopsch, Uri Onn and Christopher Voll proved that when $\chara\cO=0$, we have $\alpha(\SL_2(\cO))=1$ for any $p$ (see \cite[Corollary~2.3]{AKOV-2012}) and $\alpha(\SL_3(\cO))=2/3$ for $p$ large enough (see  \cite[Theorem~1.4]{AKOV-2012}, \cite[Theorem~E]{AKOV-Duke} and \cite[Corollary~D]{AKOV-ProcLMS}).
  Furthermore, Michele Zordan \cite{Zordan-SL4} has proved that the abscissa of $\alpha(\SL_4(\cO))=1/2$ for $\chara\cO=0$ and $p\neq 2$. While these known values equal $r/\kappa$, it is also known, by a result of Larsen and Lubotzky \cite[Theorem~8.1]{Larsen-Lubotzky}, that $\alpha(\SL_n(\cO))\geq 1/15$, so the abscissa is greater than $r/\kappa$ for $n>30$. 
In \cite{Aizenbud-Avni-2016}, Avraham Aizenbud and Nir Avni showed that if $\chara\cO=0$, we have $\alpha(\SL_n(\cO))\leq 22$, for any $n$ (they also proved similar results for other semisimple groups). 


\subsubsection*{Twist representation growth of $\GL_n(\cO)$}
In Section~\ref{sec:Twist reps}, we initiate the study of twist representation zeta functions of $\GLnO$, where the representations are counted up to one-dimensional twists. Such twist zeta functions have already been studied for nilpotent groups (cf.~\cite{Stasinski-Voll-Tgrps} and \cite{Hrushovski-Martin-Rideau-Cluckers}). The idea is that the abscissa of $\zeta_{\SL_n(\cO)}(s)$ is closely related to that of the twist zeta function of $\GLnO$.  This is useful because from some points of view, the representation theory of $\GL_n(\cO)$ (even up to twisting) is easier than that of $\SL_n(\cO)$. 
We show that the number of twist isoclasses of irreducible representations of $\GL_{n}(\cO)$ of a given dimension is finite, and hence that the twist representation zeta function $\twistGLn(s)$ can be defined by the Dirichlet series counting twist isoclasses of a given dimension. We also prove that if  $\chara\cO$ does not divide $n$, then the abscissa of $\twistGLn(s)$ is equal to that of $\zeta_{\SL_n(\cO)}(s)$ (see Proposition~\ref{prop:abscSLn-twistGLn}).

\subsubsection*{Twist zeta functions of $\GL_2(\cO)$}
In Section~\ref{sec:twist-GL2}, we study the twist zeta functions $\twistGLtwo(s)$ for any $\cO$, using Clifford theory. In Theorem~\ref{thm:twist_zeta_function_p_odd}, we give an exact formula for $\twistGLtwo(s)$ when the residue characteristic $p$ of $\cO$ is odd. This formula does not follow from the known formula for $\zeta_{\SL_2(\cO)}(s)$ (with $p\neq 2$) in any straightforward way. As part of the formula for $\twistGLtwo(s)$, we compute the twist zeta function of $\GL_2(\F_q)$. Even this was, as far as we are aware, not known previously, and the computation hinges on certain properties of Deligne--Lusztig induction.
In Theorem~\ref{thm:twist_zeta_function_char_zero}, we give an asymptotic upper bound for the number of twist isoclasses of given dimension, and deduce that the abscissa of $\twistGLtwo(s)$, for $\cO$ of characteristic $0$ and $p=2$, is $1$. As mentioned above, it was previously known that the abscissa of $\zeta_{\SL_2(\cO)}(s)$ is $1$ whenever $\chara\cO=0$. Our computations of the abscissa of $\twistGLtwo(s)$ in this case, together with Proposition~\ref{prop:abscSLn-twistGLn}, give a new proof of this fact.
We also show that the abscissa of $\twistGLtwo(s)$ is $1$ when $\chara\cO=2$, that is, when $\cO=\F_q[[t]]$ with $q$ even. This does not follow from any previously known results and our computation is substantially harder than in the cases where $\chara\cO\neq 2$.

\subsubsection*{The zeta function of $\SL_2(\F_q[[t]])$, $q$ even}
In Section~\ref{sec:SL2-chapter}, we assume that $\chara\cO=2$, that is, $\cO=\F_q[[t]]$ with $q$ even. We give a Clifford theory construction of the representations of $\SL_2(\F_q[[t]]/(t^r))$ for $r$ even, which is completely explicit apart from the order of certain finite groups $V(\beta,\theta)$ (see Definition~\ref{def:u_r(beta,theta)}) and certain integers $c\in\{1,2,3\}$ (see Lemma~\ref{lem:kernel-type-3}). We use this construction to approximate the zeta function $\zeta_{\SL_2(\cO)}(s)$, and show that its abscissa lies between $1$ and $5/2$ (Theorem~\ref{thm:SL2-Ochar2}). The lower bound $1$ follows from a general result of Larsen and Lubotzky \cite[Proposition~6.6]{Larsen-Lubotzky}, but we give an independent proof of this.

Section~\ref{sec:proof_of_lemma} is devoted to a proof of Lemma~\ref{lem:kernel-type-3}, which is crucial for our results about $\tilde{\zeta}_{\GL_2(\F_q[[t]])}(s)$ and $\zeta_{\SL_2(\F_q[[t]])}(s)$ when $q$ is even. The lemma gives the number of solutions, up to a factor $c\in\{1,2,3\}$, in $\F_q[[t]]/(t^i)$, $i\geq 1$, to the equation
\[
x^2+\tau xy+\Delta y^2 = 1,
\]
where $\tau,\Delta\in \F_q[[t]]/(t^i)$ are such that the image of the matrix
$\left[\begin{smallmatrix}0 & 1\\ \Delta & \tau \end{smallmatrix}\right]$ mod $(t)$ is a scalar plus a regular nilpotent matrix.
The number of solutions depends in a delicate way on a new invariant, which we call the odd depth, of the twist orbit (i.e., orbit modulo scalars) of the matrix $\left[\begin{smallmatrix}0 & 1\\ \Delta & \tau \end{smallmatrix}\right]$ (see Definition~\ref{def:Odd-depth}).

\begin{rem*}
After the present paper had been accepted for publication, Hassain M and Pooja Singla \cite{M-Singla} announced results about the representations of $\SL_2(\cO)$, $p=2$, which in particular imply that the abscissa of  $\zeta_{\SL_2(\F_q[[t]])}(s)$ is $1$.	
\end{rem*}

\subsection*{Notation}
We let $\N$ stand for the set of natural numbers, not including $0$.

In Sections~\ref{sec:twist-GL2} and \ref{sec:SL2-chapter}, we will use the Vinogradov notation $f(r)\ll g(r)$ for two functions $f(r),g(r)$ of $r$ (or of $l=\ceil{r/2}$). Note that $f(r)\ll g(r)$ is equivalent to $f(r)=O(g(r))$. We will also write $f(r)\asymp g(r)$ when $f(r)\ll g(r)$ and $g(r)\ll f(r)$.

\begin{acknowledgement*} This research was supported by EPSRC grant EP/K024779/1. We are grateful to Hassain M and Pooja Singla for pointing out a mistake in Section~\ref{sec:SL2-chapter} in a previous version of this paper.
\end{acknowledgement*} 

\section{Representation zeta functions of simple Lie groups}\label{sec:Lie groups}

In this section, we will determine the abscissa of convergence of a class of Mellin zeta functions which contains the Witten zeta functions. Let $P$ be a polynomial in $r$ variables over $\R$ with positive coefficients. As mentioned in the introduction, the series 
\begin{equation}\label{equation:mellin_zeta}
\zeta(P;s)=\sum_{x_1,\dots,x_r=1}^\infty P(x_1,\dots,x_r)^{-s}
\end{equation}
has an abscissa of convergence and extends meromorphically to a Mellin zeta function. Both the series \eqref{equation:mellin_zeta} and the continued function will be denoted by $\zeta(P;s)$. We will speak of the abscissa of convergence either of the series \eqref{equation:mellin_zeta} or its meromorphic continuation (where, in the latter case, we mean the maximum of the real parts of its poles), and since one determines the other, there is no ambiguity. 

In the following, we will write $\bar{x}=(x_1,\dots,x_r)$ and use the convention that for positive integers $n$ and $N$, a sum $\sum_{\bar{x}=n}^{N}$ means $\sum_{x_1 = n}^N\sum_{x_2 = n}^N\dots\sum_{x_r = n}^N $.

Let $\kappa$ stand for the degree of $P$. Since the coefficients of $P$ are positive real numbers, we can estimate the partial sums of~\eqref{equation:mellin_zeta}, when $s$ is real and positive, as follows:
\[
\sum_{\bar{x}=1}^N P(\bar{x})^{-s}
\ge N^r P(N,\dots,N)^{-s}.
\]
\label{page:argument_for_optimal_abscissa}It is then clear that the series \eqref{equation:mellin_zeta} diverges whenever $s < r/\kappa$ (for real $s$), so the abscissa of convergence is always at least $r/\kappa$.

We will restrict our attention to the case where $P$ is a product of linear factors. In Subsection~\ref{sec:convergence_mellin_zeta}, we compute the exact abscissa of convergence for these types of series. In particular, polynomials of this form arise from the Weyl formula for the dimensions of irreducible representations of compact Lie groups. In that context, the Mellin zeta 
function is the representation zeta function of the group, also known as the 
Witten zeta function. Larsen and Lubotzky showed in
\cite[Theorem~5.1]{Larsen-Lubotzky} that the abscissa of convergence of the 
Witten zeta function corresponding to a simple, simply-connected, complex Lie group is 
$r/\kappa$, where $r$ is the rank of the root system of the group and $\kappa$ 
the number of positive roots. This result follows from our more general setting, as will be explained in Subsection~\ref{sec:connection_with_lie_groups}.

Mellin zeta functions of a certain type were also studied by Kurt 
Mahler in \cite{MahlerMellin}, among others. Mahler considered polynomials  $P(\bar{x})$ satisfying the following hypothesis: $P$ does not vanish for $x_1\ge 0,\dots,x_r\ge 0$, and its top degree homogeneous part vanishes only at the origin.
This class of polynomials is also described in \cite{FriedmanPereira}. For these polynomials, he could prove that the abscissa of convergence has the minimal possible value $r/\kappa$. It was claimed in \cite[Theorem~20]{BdlH} that the abscissa of convergence of Witten zeta functions follows directly from Mahler's result and 
the Weyl dimension formula. However, we note that 
this is not the case, since the polynomials arising from the Weyl dimension 
formula are usually not included in the class of polynomials considered by 
Mahler. In fact, it follows from our results that one can not prove the 
Larsen--Lubotzky result without proving a 
certain property of the root systems, which we record as Lemma~\ref{lemma:2.5}. 

Having explained in detail how the abscissa of convergence of a Witten zeta function is affected by the structure of the root system, it is possible to consider the partial contributions of certain classes of representations to the abscissa. In Subsection~\ref{sec:regular_and_irregular}, we define a representation to be \emph{regular} if its dominant weight corresponds to a semi-simple element via the Killing form, and irregular otherwise. We then show in Proposition~\ref{prop:irregular_reps_of_Lie_groups} that only the regular representations contribute to the abscissa of convergence (and hence to the representation growth) of a simple compact Lie group.

\subsection{The abscissae of convergence}
\label{sec:convergence_mellin_zeta}

We will determine the abscissae of convergence of Mellin zeta functions associated to a certain class of polynomials, which we now define.
Let $P=P_1\cdots P_\kappa$, where each $P_i\in\R[\bar{x}]$ is a linear polynomial of the form
\[
P_i=a_{i1}x_1+a_{i2}x_2+\cdots+a_{ir}x_r.
\]
We assume all the coefficients $a_{ij}$ are non-negative and at least one of them is positive for each $i$.
The degree of $P$ is $\kappa$, the number of linear factors.

For any subset of variables $\{x_{j_1},\dots,x_{j_l}\}$, we define the 
\emph{subpolynomial} $P[j_1,\dots,j_l]$ as the product of those linear factors 
of $P$ which contain only the variables $x_{j_1},\dots,x_{j_l}$ with a non-zero 
coefficient. For any subpolynomial $Q$ of $P$, we write $r(Q)$ for $l$, the 
number of variables appearing in $Q$,
and $\kappa(Q)$ for the number of linear factors in $Q$, which is also the degree of $Q$.

For example, if $P=x_1 x_2(x_1+3x_2)(2x_1+x_3)$, we have $r=3$ and $\kappa=4$. The polynomial $P$ has subpolynomials $P[1,2]=x_1 x_2(x_1+3x_2)$ and $P[1]=x_1$, among others. We have $r(P[1,2])=2$, $\kappa(P[1,2])=3$ and $r(P[1])=1$, $\kappa(P[1])=1$.
 
Let $S_P$ be the set of all proper subpolynomials of $P$, including the `empty' subpolynomial $1$ with no variables and degree 0. Define
\[
\alpha=\max_{Q\in S_P}\frac{r-r(Q)}{\kappa-\kappa(Q)}.
\]
Note that because we are taking the empty subpolynomial into account, we have $\alpha\ge r/\kappa$. We proceed to prove that the abscissa of convergence of $\zeta(P;s)$ equals $\alpha$.

\begin{thm}\label{theorem:abscissa_divergence}
Suppose $s\in\R$, $s<\alpha$. Then the series $\zeta(P;s)$ diverges.
\end{thm}

\begin{proof}

Since $s<\alpha$, we can choose a proper subpolynomial $Q$ of $P$ satisfying
\[
\frac{r-r(Q)}{\kappa-\kappa(Q)}>s.
\]
By relabeling the variables, we may assume that $Q=P[1,\dots,l]$, with $l=r(Q)$.
Consider the polynomial $P'=P(1,\dots,1,x_{l+1},\dots,x_r)$. By this substitution, all linear factors contained in the subpolynomial $Q$ become constants. The number of these factors is $\kappa(Q)$, and all the remaining factors contain at least one of the variables $x_{l+1},\dots,x_r$ with a non-negative coefficient. We can now estimate
\[
\sum_{\bar{x}=1}^N P(\bar{x})^{-s} \ge
\sum_{x_{l+1},\dots,x_r=1}^N P'(x_{l+1},\dots,x_r)^{-s}.
\]
The sum on the right hand side is taken over $r-r(Q)$ variables, and the degree 
of the polynomial $P'$ is $\kappa-\kappa(Q)$. As all the coefficients in $P'$ 
are still non-negative, the abscissa of convergence of $\zeta(P';s)$ is known 
to be at least $(r-r(Q))/(\kappa-\kappa(Q))$ (as remarked in the beginning of Section~\ref{sec:Lie groups}). Since $s$ is smaller than this 
ratio, we know that the right hand side of the above inequality diverges as 
$N\to\infty$. This proves the claim.
\end{proof}
The above theorem shows that $\alpha$ is a lower bound for the abscissa of convergence of $\zeta(P;s)$. Next, we will prove that $\alpha$ is also an upper bound, and hence that the abscissa equals $\alpha$.  We will use the following integral test for convergence of our multi-variable series: 

\begin{lemma}\label{lem:integral-test}
Suppose that $f(\bar{x})=f(x_1,\dots,x_r)$ is a polynomial with non-negative real coefficients. Let $s\in\R$, $s>0$. Then the multivariate sum $\sum_{\bar{x}=1}^\infty f(\bar{x})^{-s}$ converges if the corresponding integral $\int_{[1,\infty)^r}f(\bar{x})^{-s}\df\bar{x}$ is finite.
\end{lemma}

\begin{proof}
Since all terms in the infinite series are non-negative, one can show convergence by finding an upper bound for the set of partial sums of the series. To this end, let $N\in\N=\{1,2,\dots\}$. We aim to describe an upper bound for the partial sum $\sum_{\bar{x}=1}^{N} f(\bar{x})^{-s}$ in terms of the integral $\int_{[1,\infty)^r}f(\bar{x})^{-s}\df\bar{x}$.

Let
\[
W_N = \{\bar{x}\in \{1,\dots, N\}^r \mid \text{$x_i=1$ for some $i$}\}
\]
denote the set of integral points on the boundary of the domain of integration $[1,\infty)^r$ whose coordinates are bounded by $N$. 
The proof proceeds as the one for the usual integral test for series of one variable, except that one has to consider the sum over $W_N$ separately. 

We have a map $W_N\rightarrow \{2,\dots,N\}^r$
given by sending $\bar{x}$ to $\bar{x}^{\sharp}=(x_1^{\sharp},\dots,x_r^{\sharp})$, where 
\[
\begin{cases}
x_{i}^{\sharp}=x_{i} & \text{if $x_{i}>1$},\\
x_{i}^{\sharp}=2 & \text{if $x_{i}=1$}.
\end{cases}
\]
For every monomial term $ax_1^{n_1}\cdots x_r^{n_r}$ in $f$ and each $i$, we have
\[
ax_1^{n_1}\cdots 2^{n_i} \cdots x_r^{n_r}
\le 2^{\kappa}\cdot ax_1^{n_1}\cdots 1^{n_i}\cdots x_r^{n_r},
\]
where $\kappa$ is the total degree of $f$ and $x_1,\dots,x_r\in [1,\infty)$. Therefore, for any $\bar{x}\in W_N$, we have
\[
f(\bar{x}^{\sharp}) \le 2^{\kappa r}f(\bar{x}).
\]
Thus, we obtain the estimate
\[
\sum_{\bar{x}=1}^{N} f(\bar{x})^{-s}
=\sum_{\bar{x}\in W_N} f(\bar{x})^{-s} + \sum_{\bar{x}=2}^{N} f(\bar{x})^{-s}
\le 2^{\kappa r s}\sum_{\bar{x}\in W_N} f(\bar{x}^{\sharp})^{-s} + \sum_{\bar{x}=2}^{N} f(\bar{x})^{-s}.
\]
As $\bar{x}^{\sharp}\in \{2,\dots,N\}^r$, for any $\bar{x}\in W_N$, we obtain
\[
\sum_{\bar{x}=1}^{N} f(\bar{x})^{-s}
\le (2^{\kappa r s}+1)\sum_{\bar{x}=2}^{N} f(\bar{x})^{-s}.
\]

Next, assume that the integral $\int_{[1,\infty)^r} f(\bar{x})^{-s} \df\bar{x}$ is finite.
For $\bar{k}=(k_1,\dots,k_r)\in \N^r$, let 
\[
H_{\bar{k}} =[k_1,k_1+1]\times\dots\times[k_r, k_r+1]
\]
denote the unit cube at $\bar{k}$. Splitting the domain of integration into unit cubes yields
\[
\int_{[1,\infty)^r} f(\bar{x})^{-s} \df\bar{x}
=\sum_{\bar{k}\in \N^r}\int_{H_{\bar{k}}}f(\bar{x})^{-s} \df\bar{x}.
\]
As the function $\bar{x}\mapsto f(\bar{x})^{-s}$ is decreasing in every coordinate, we have
\[
\int_{H_{\bar{k}}}f(\bar{x})^{-s} \df\bar{x}
\ge f(k_1+1,\dots,k_r+1)^{-s} \quad \text{for every $\bar{k}\in \N^r$},
\]
so that
\[
\int_{[1,\infty)^r} f(\bar{x})^{-s} \df\bar{x}
\ge \sum_{\bar{k}\in \N^r}f(k_1+1,\dots,k_r+1)^{-s}
= \sum_{\bar{k}=2}^\infty f(k_1,\dots,k_r)^{-s}.
\]
Finally, putting the above estimates together, we arrive at
\[
\sum_{\bar{x}=1}^{N} f(\bar{x})^{-s}
\le (2^{\kappa r s}+1)\int_{[1,\infty)^r} f(\bar{x})^{-s} \df\bar{x}.
\]
Hence, the partial sum $\sum_{\bar{x}=1}^{N} f(\bar{x})^{-s}$ is bounded from above by a constant not depending on $N$. This proves the claim.
\end{proof}

\begin{thm}\label{theorem:abscissa_convergence}
Suppose $s\in\R$, $s>\alpha$. Then the series $\zeta(P;s)$ 
converges.
\end{thm}

\begin{proof}
As all the coefficients of $P$ are non-negative, the aim is to show that the partial sums
\[
\sum_{\bar{x}=1}^N P(\bar{x})^{-s}
\]
are bounded from above by a constant not depending on $N$. In particular, since all summands are positive, we need not worry about the order of summation.

We start by dividing the index space $\{1,\dots,N\}^r$ into regions according to the relative magnitudes of the coordinates. For a permutation $\sigma\in\sym(r)$ of 
$1,\dots,r$, let
\[
L_\sigma=\{\bar{x}\in\{1,\dots,N\}^r \mid N\ge x_{\sigma(1)}\ge 
x_{\sigma(2)}\ge\cdots\ge x_{\sigma(r)}\ge 1\}.
\]
The sets $L_\sigma$ may overlap at their boundaries, but together they cover the set $\{1,\dots,N\}^r$. Thus, we have
\begin{equation*}
\sum_{\bar{x}=1}^N P(\bar{x})^{-s}
\le\sum_{\sigma\in\sym(r)}\;\sum_{\bar{x}\in L_\sigma}P(\bar{x})^{-s}.
\end{equation*}
We shall fix an arbitrary permutation $\sigma\in\sym(r)$ and bound the 
corresponding inner sum over $L_\sigma$ appearing on the right hand side of the 
previous inequality by a constant. As the number of different permutations is 
finite, this will be enough to show the convergence of the whole sum. 

Furthermore, we replace $P$ in the inner sum by a single monomial $P_\sigma$, chosen with respect
to $\sigma$ in a way explained below. As the coefficients of $P$ are non-negative, and $s>0$, we then have
\[
\sum_{\bar{x}\in L_\sigma}P(\bar{x})^{-s}\le\sum_{\bar{x}\in L_\sigma}P_\sigma(\bar{x})^{-s}.
\]
Finally, we use Lemma~\ref{lem:integral-test} to reduce the convergence of the final sum on the right to that of the integral 
\begin{equation}
\label{Psigma-integral}
\int_{\bar{x}\in L_\sigma^\R} P_\sigma(\bar{x})^{-s}\df\bar{x},
\end{equation}
where $L_\sigma^\R=\{\bar{x}\in[1,N]^r \mid x_{\sigma(1)}\ge\cdots\ge x_{\sigma(r)}\}$. 
We proceed to explain how to choose $P_\sigma$, and then to find a constant bound for the integral~\eqref{Psigma-integral}.

Consider $\sigma\in\sym(r)$ fixed. To choose a monomial $P_\sigma$ appearing in 
$P$, we must choose a single term from each linear factor $P_i$ of $P$. We 
start by choosing $x_{\sigma(1)}$ from as many factors as it appears in (with a non-zero coefficient). Let $e_1$ be the number of these factors, so that the 
monomial we are building contains $x_{\sigma(1)}^{e_1}$. Let then 
$Q_1=P[\sigma(2),\dots,\sigma(r)]$. For the degree, we have $\kappa(Q_1)=\kappa-e_1$. As before, pick $x_{\sigma(2)}$ from as many 
factors of $Q_1$ as it appears in, and let $e_2$ be the number of these factors. 
The monomial we are building will then contain 
$x_{\sigma(1)}^{e_1}x_{\sigma(2)}^{e_2}$. Define 
$Q_2=P[\sigma(3),\dots,\sigma(r)]$, so that $\kappa(Q_2)=\kappa-e_1-e_2$, and continue in the same manner. Finally, we obtain the monomial
\[
P_\sigma(\bar{x})=x_{\sigma(1)}^{e_1}x_{\sigma(2)}^{e_2}\cdots x_{\sigma(r)}^{e_r}.
\]
By construction, this monomial appears in a term of $P$, possibly multiplied by 
a positive constant which we may safely ignore.
(Note that the monomial can also be obtained as the first monomial in 
the lexicographic ordering based on the ordering 
$x_{\sigma(1)},\dots,x_{\sigma(r)}$ of the variables.)

Let us make some notes about the exponents $e_i$. First of all, any of them may 
be~0. Secondly, for all $i\in\{1,\dots,r\}$, we know that $e_i$ is the number of 
 factors containing $x_{\sigma(i)}$ in the subpolynomial $Q_i$. Since at each 
step we pick the variable $x_{\sigma(i)}$ from all those linear factors it 
appears in, we always have $e_i=\kappa(Q_{i-1})-\kappa(Q_i)$. It follows that 
$\kappa(Q_j)=\sum_{i=j+1}^r e_i$ for all $j$.

To simplify notation, we make a change of variables and rewrite $x_{\sigma(i)}\mapsto x_i$ for all~$i$. We also write $x_0=N$, so that
\[
\int_{\bar{x}\in L_\sigma}P_\sigma(\bar{x})^{-s}\df\bar{x}
=\int_1^{x_0}x_1^{-e_1 s}\left(\int_1^{x_1}x_2^{-e_2 s}\cdots
\left(\int_1^{x_{r-1}}x_r^{-e_r s}\df x_r\right)
\cdots\df x_2\right)\df x_1.
\]
For each $i\in\{1,\dots,r\}$, let $I_i(x_{r-i})$ denote the result of the $i$ innermost iterations in the previous integral:
\[
I_i(x_{r-i})=\int_1^{x_{r-i}}x_{r-i+1}^{-e_{r-i+1}s}\cdots
\left(\int_1^{x_{r-1}}x_r^{-e_r s}\df x_r\right)
\cdots\df x_{r-i+1}.
\]
For the extreme case, set $I_0(x_r)=0$.

Note that integrating a power expression produces another power expression (unless the exponent happens to be $-1$, which we will avoid). Our strategy is to estimate the iterated integral above by keeping track of the constants appearing in the resulting power expressions at each step. We simplify the work by finding upper bounds on the integrands at each step, and the bound on the final iteration will then provide a bound for the whole integral.

For all $i\in\{0,\dots,r\}$, we define a non-zero rational number $C_i$ and non-negative integers $a_i$ and $b_i$ recursively as follows. Let $C_0=1$ and $a_0=b_0=0$. Assume then that $C_i$, $a_i$ and $b_i$ have been chosen for some $i\in\{0,\dots,r-1\}$, and define
\[
C_{i+1}=\frac{|C_i|}{a_i+1-(b_i+e_{r-i})s}.
\]
If necessary, we can change $s$ here to a smaller value to make all the finitely 
many denominators non-zero. As long as we have $s>\alpha$ for the new value of 
$s$, it will make no difference to the statement of the theorem.
Now, if $C_{i+1}<0$, we define $a_{i+1}=b_{i+1}=0$, and if $C_{i+1}>0$, we use
\[
a_{i+1}=a_i+1
\quad \text{and} \quad
b_{i+1}=b_i+e_{r-i}.
\]

Next, we show by induction that for all $i\in\{0,\dots,r\}$, we have
\[
I_i(x_{r-i})<|C_i|x_{r-i}^{a_i-b_is} \quad \text{for $x_{r-1}\ge 1$}.
\]
Note that if $C_i<0$, this means that $I_i(x_{r-i})<|C_i|$. Clearly, the condition holds for $i=0$. Assume then that the condition holds for some $i\in\{0,\dots,r-1\}$. We get
\begin{align*}
I_{i+1}(x_{r-i-1}) & =\int_1^{x_{r-i-1}}x_{r-i}^{-e_{r-i}s}I_i(x_{r-i})\df x_{r-i} \\
& \le|C_i|\int_1^{x_{r-i-1}}x_{r-i}^{a_i-(b_i+e_{r-i})s}\df x_{r-i} \\
& =\frac{|C_i|}{a_i+1-(b_i+e_{r-i})s}\bigl(x_{r-i-1}^{a_i+1-(b_i+e_{r-i})s}-1\bigr) \\
& =C_{i+1}\bigl(x_{r-i-1}^{a_i+1-(b_i+e_{r-i})s}-1\bigr).
\end{align*}
It follows that if $C_{i+1}>0$, we have
\[
I_{i+1}(x_{r-i-1})<C_{i+1}x_{r-i-1}^{a_i+1-(b_i+e_{r-i})s}
=|C_{i+1}|x_{r-i-1}^{a_{i+1}-b_{i+1}s}.
\]
On the other hand, if $C_{i+1}<0$, we have $I_{i+1}(x_{r-i-1})<|C_{i+1}|$. We see that the desired condition holds in both cases, and the induction is complete.

It now suffices to prove that $C_r$ is negative, for then we have, by the above, that
\[
\int_{\bar{x}\in L_\sigma}P_\sigma(\bar{x})^{-s}\df\bar{x}=I_r(x_0)\le|C_r|.
\]
Let $k$ be the last index before~$r$ for which $C_k$ was negative. If such $k$ does not exist, set $k=0$. Using the recursive construction of the constants $a_i$, $b_i$ and $C_i$, we see that $a_k=b_k=0$, so that
\[
a_r=\sum_{i=k}^{r-1}1=r-k, \qquad
b_r=\sum_{i=k}^{r-1}e_{r-i}=\sum_{i=1}^{r-k}e_i
\]
and
\[
C_r=\frac{|C_{r-1}|}{a_{r-1}+1-(b_{r-1}+e_1)s}
=\frac{|C_{r-1}|}{a_r-b_r s}
=\frac{|C_{r-1}|}{r-k-\bigl(\sum_{i=1}^{r-k}e_i\bigr)s}.
\]
However, considering the subpolynomial $Q_{r-k}=P_\sigma[r-k+1,\dots,r]$, we have $r(Q_{r-k})=k$ and $\kappa(Q_{r-k})=\sum_{i=r-k+1}^r e_i$. (In the case $k=0$, the corresponding $Q_r$ is the empty subpolynomial.) This implies
\[
C_r=\frac{|C_{r-1}|}{r-r(Q_{r-k})-(\kappa-\kappa(Q_{r-k}))s}
.
\]
Now, since we have assumed that $s>\alpha$, and we have $\alpha\ge\frac{r-r(Q)}{\kappa-\kappa(Q)}$ for any subpolynomial $Q$, the denominator in the last expression is negative. This completes the proof.
\end{proof}

Based on the two previous results, we can formulate a condition for the abscissa to 
be smallest possible.
\begin{cor}\label{corollary:2.5}
The abscissa of convergence of $\zeta(P;s)$ equals $r/\kappa$ if and only if the inequality $r(Q)/\kappa(Q)\ge r/\kappa$ holds for every (non-empty) subpolynomial $Q$ of $P$.
\end{cor}

\begin{proof}
Since the abscissa of convergence is known to be at least $r/\kappa$ (this holds in general, see argument on page~\pageref{page:argument_for_optimal_abscissa}), we only need to 
find out when it is strictly greater. By 
Theorems~\ref{theorem:abscissa_divergence} and 
\ref{theorem:abscissa_convergence}, 
we know that this happens if and only if $\alpha>r/\kappa$, that is, there is some non-empty subpolynomial $Q$ with 
$(r-r(Q))/(\kappa-\kappa(Q))>r/\kappa$. The last inequality is equivalent to having 
$r(Q)/\kappa(Q)<r/\kappa$. The claim follows.
\end{proof}

\begin{rem}
The integral test for convergence of multiple series is well-known but not easy to locate in the literature. In fact, we have only found two sources containing this (in the two-variable case), but both of them misstate the result. Indeed, \cite[Section~32\,(2)]{Bromwich-infinite-series} and \cite[Proposition~7.57]{Multivariable-book} state that given a non-negative decreasing function $f(x,y):[1,\infty)\times[1,\infty)\rightarrow\R$, the double series $\sum_{(m,n)=(1,1)}^\infty f(m,n)$ converges if and only if  $\int_{[1,\infty)\times[1,\infty)}f(x,y)\df (x,y)$ converges (these sources do not make the lower bound on $m$ and $n$ explicit, but their conventions make it understood). As is easily seen, the function $f$ defined by $f(1,y)=1$, for $y\in[1,\infty)$ and $f(x,y)=0$ for $x\neq 1$ is a counterexample. The correct general statement of the integral test should have the sum starting from $(2,2)$. The fact that Lemma~\ref{lem:integral-test} nevertheless holds with the sum starting from $(1,\dots,1)$ has to do 
with a special property of polynomial functions exploited in the proof.
\end{rem}

\begin{rem}
	Taking the coefficient of \emph{every} variable in every linear factor of $P$ to be positive (non-zero) and assuming that each factor of $P$ has a non-zero constant term, one obtains the Shintani zeta functions, whose abscissae are known to always be $r/\kappa$ (see \cite{Shintani-zetas} or \cite[(9.7)]{Neukirch}). The constant terms do not affect the abscissa, and the only proper subpolynomial of such a polynomial is the empty subpolynomial, so our results give a new proof of this fact.
	
	Note that every positive rational number $a/b$ can be obtained as the abscissa of convergence of a Shintani zeta function by taking  $P=(x_1+\cdots+x_a+1)^b$.
\end{rem}

\subsection{Connection with Lie groups}
\label{sec:connection_with_lie_groups}

Suppose $G$ is a simple simply-connected Lie group over $\C$ with root system 
$\Phi$. Let $r$ and $\kappa$ denote the numbers of fundamental (or simple) and positive roots of $\Phi$, respectively. We explain how a Witten zeta 
function can be written as a Mellin zeta function corresponding to a polynomial with linear factors and non-negative real coefficients.

It is well known that the irreducible representations of $G$ are parametrised by dominant weights (also called dominant integral weights). Let $\dim(\mu)$ denote the dimension of an irreducible representation corresponding to a dominant weight $\mu$. This dimension is given by the Weyl dimension formula
\begin{equation}\label{equation:weyl_formula}
\dim(\mu)=\prod_{\alpha\in\Phi^+}
\frac{\inner{\alpha,\mu+\rho}}{\inner{\alpha,\rho}}.
\end{equation}
Here $\inner{\cdot,\cdot}$ is the Killing form, $\Phi^+$ a choice of 
positive roots and $\rho$ the sum of fundamental weights (which is also equal to 
half the sum of the positive roots).

To see what the dimension formula looks like in practice, fix a dominant weight $\mu$. Let $\alpha$ be a positive root, and write $\alpha$ and $\mu$ as linear combinations of fundamental weights and fundamental roots, respectively:
\[
\alpha=\sum_{j=1}^r b_j\alpha_j \quad \text{and} \quad \mu=\sum_{j=1}^r m_j \omega_j.
\]
Assume also that the length of each $\alpha_j$ is $z_j$ (when the length of the shortest root is normalised to 1). Then the inner products in the numerator and denominator of~\eqref{equation:weyl_formula} equal
\[
\sum_{j=1}^r z_j b_j(m_j+1) \quad \text{and} \quad \sum_{j=1}^r z_j b_j.
\]
Hence, the denominator in the dimension formula is a constant with respect to $\mu$, and the numerator becomes a product of $\kappa$ linear polynomials in the variables $m_j$, with non-negative coefficients depending on $\alpha$.

Up to a constant, namely, the denominator in~\eqref{equation:weyl_formula}, the series defining the zeta function $\zeta_G(s)$ is thus 
\begin{equation*}
\sum_{m_1,\dots,m_r=0}^\infty \; \prod_{\alpha\in\Phi^+}
\Bigl(\sum_{j=1}^r z_j b_j(m_j+1)\Bigr)^{-s}.
\end{equation*}
As we are only interested in the convergence properties of the series defining $\zeta_G(s)$, we focus our attention on this series.
Writing $x_j$ for $m_j+1$, and $a_{ij}$ for $z_j b_j$ corresponding to the $i$-th positive root $\alpha$ in some fixed ordering, we are left with the following series:
\begin{equation}\label{equation:our_zeta}
\sum_{x_1,\dots,x_r=1}^\infty(a_{11}x_1+\cdots+a_{1r}x_r)^{-s}\cdots
(a_{\kappa 1}x_1+\cdots+a_{\kappa r}x_r)^{-s}.
\end{equation}
This is a series of the form $\zeta(P;s)$ corresponding to the polynomial
\begin{equation}\label{equation:witten_polynomial}
P(x_1,\dots,x_r)=(a_{11}x_1+\cdots+a_{1r}x_r)\cdots (a_{\kappa 
1}x_1+\cdots+a_{\kappa r}x_r),
\end{equation}
and as $P$ is a product of linear factors with non-negative coefficients, we can apply the results from the previous subsection.

In the series $\zeta(P;s)$, the number of variables is the 
rank $r$ of the root system, and the degree is the number of positive roots 
$\kappa$. Consider any subpolynomial $Q=P[j_1,\dots,j_l]$. A \emph{Levi 
subsystem} of $\Phi$ is a root subsystem spanned by a subset of the fundamental 
roots of $\Phi$. Letting $\Psi$ denote the Levi subsystem spanned by 
$\alpha_{j_1},\dots,\alpha_{j_l}$ and using the Weyl formula for this subsystem, 
we get another series with the polynomial $P$ replaced by 
$Q$. Therefore, the subpolynomials of $P$ and Levi subsystems of 
$\Phi$ are in 
one-to-one correspondence. Note also that $r(Q)$ is the rank of the Levi 
subsystem and $\kappa(Q)$ is the number of its positive roots, so we may write 
$r(\Psi)$ and $\kappa(\Psi)$ instead of $r(Q)$ and $\kappa(Q)$ to reflect this 
connection.

By Corollary~\ref{corollary:2.5}, to recover 
the Larsen--Lubotzky theorem \cite[Theorem~5.1]{Larsen-Lubotzky}, we need to 
show that the ratio $r(\Psi)/\kappa(\Psi)$ is greater than 
$r(\Phi)/\kappa(\Phi)$ for any subsystem $\Psi$ of $\Phi$. This fact can easily 
be proved by inspecting all irreducible root systems, as is done in 
\cite{Larsen-Lubotzky} and \cite{Liebeck-Shalev05}. We give a proof of this 
which does not use the classification of irreducible root systems, in the hope that this will give better understanding of the structural 
properties of roots systems responsible for the value 
$r/\kappa$. In the proof of the following lemma, the idea to use 
the relation between the Coxeter number and the maximal height is due to Stefan 
Patrikis. We thank him for allowing us to include that argument here.

\begin{lemma}\label{lemma:2.5}
For any  proper Levi subsystem $\Psi$  of $\Phi$, we have 
$r/\kappa<r(\Psi)/\kappa(\Psi)$.
\end{lemma}

\begin{proof}
Note that since $G$ is assumed to be simple, we know that $\Phi$ is 
irreducible. If $\Psi$ is a reducible root system and $\Psi=\Psi_1\oplus\Psi_2$, 
we have
\[
\frac{r(\Psi)}{\kappa(\Psi)}=\frac{r(\Psi_1)+r(\Psi_2)}{\kappa(\Psi_1)+\kappa(\Psi_2)}.
\]
It is easy to check that if $a/b\le c/d$ holds for some positive integers $a$, $b$, $c$ and $d$, then
\[
\frac{a+c}{b+d}\le\frac{c}{d}.
\]
This means that the ratio for a reducible Levi subsystem cannot be greater than the ratio for one of its irreducible components, so we may restrict our attention to the irreducible subsystems.

Assume now that $\Psi$ is a proper irreducible Levi subsystem of $\Phi$. Let us first observe that the Dynkin diagram of an irreducible root system cannot contain a cycle. Indeed, assuming that the nodes $\alpha_1,\dots,\alpha_l$ form a cycle (with no other edges between these nodes), consider the corresponding unit vectors $u_i=\frac{\alpha_i}{|\alpha_i|}$, and the vector sum $u=\sum_i u_i$. It follows from basic properties of root systems that the smallest possible angle between two fundamental roots not orthogonal to each other is $2\pi/3$. Therefore, we have $\inner{u_i,u_j}\le -1/2$ for any $\alpha_i$, $\alpha_j$ connected by an edge in the Dynkin diagram. Therefore, we get
\[
\inner{u,u}=\sum_{i=1}^{l}\inner{u_i,u_i}+2\sum_{i=1}^{l-1}\inner{u_i,u_{i+1}}+2\inner{u_l,u_1}
\le l-l=0.
\]
It follows that $u$ is the zero vector, which is impossible, since the fundamental roots $\alpha_1,\dots,\alpha_l$ are linearly independent.

Knowing that the Dynkin diagram of an irreducible root system is a tree, one can show that for such a root system, the ratio of the number of roots over the rank is equal to the Coxeter number of the system. This is proved for example in \cite[Theorem~10.5.3]{carter} and \cite{SteinbergReflectionGroups}.
In \cite[Theorem~1.4]{SteinbergReflectionGroups}, the author proves also that the Coxeter 
number equals $\eta+1$, where $\eta$ is the height of the highest root of the 
system, without using the classification. It then suffices to show that 
$\eta(\Psi)<\eta(\Phi)$, 
where $\eta(\Psi)$ and $\eta(\Phi)$ are the heights of the highest roots in the 
root systems $\Psi$ and $\Phi$, respectively.

Let $\alpha_1,\dots,\alpha_n$ be all the fundamental roots of $\Phi$, and write 
$\beta$ for the highest root in $\Psi$. Relabeling the fundamental roots if 
necessary, we may assume that $\beta=\sum_{i=1}^k b_i\alpha_i$, where each $b_i$ 
is positive and $k<n$. Since the Dynkin 
diagram of $\Phi$ is connected, one of the fundamental roots 
$\alpha_{k+1},\dots,\alpha_n$ must be connected to one of 
$\alpha_1,\dots,\alpha_k$. Assume without loss of generality, that $\alpha_1$ is 
connected to $\alpha_n$.

Consider the simple reflection $\sigma_n$ corresponding to $\alpha_n$. We have
\[
\sigma_n(\beta)=\beta-\frac{2\inner{\alpha_n,\beta}}{\inner{\alpha_n,\alpha_n}}\alpha_n.
\]
Note that $\inner{\alpha_i,\alpha_n}\le 0$ for all $i<n$ (the angles between fundamental roots are obtuse), and $\inner{\alpha_1,\alpha_n}<0$ ($\alpha_1$ is connected to $\alpha_n$). It follows that
\[
\inner{\alpha_n,\beta}=\sum_{i=1}^k b_i\inner{\alpha_i,\alpha_n}<0,
\]
so that $\sigma_n(\beta)=\beta+c\alpha_n$ for some $c>0$. This shows
that the height of $\sigma_n(\beta)$ is strictly greater than the 
height of $\beta$, which implies that $\eta(\Psi)<\eta(\Phi)$.
\end{proof}

The theorem of Larsen and Lubotzky \cite[Theorem~5.1]{Larsen-Lubotzky} that the abscissa of convergence of $\zeta_{G}(s)$ is $r/\kappa$ now follows from 
Corollary~\ref{corollary:2.5} and Lemma~\ref{lemma:2.5}.

\subsection{Regular and irregular representations}
\label{sec:regular_and_irregular}

We continue to let $G$ denote a simple simply-connected complex Lie group. In this subsection, we briefly consider the effect on the abscissa of convergence of $\zeta_G$ of different types of representations.

Let $\ca$ stand for a Cartan subalgebra of the Lie algebra of $G$. An element 
$s\in\ca$ is called \emph{regular} if $\alpha(s)\neq 0$ for all positive 
roots $\alpha$. Suppose now that $\mu\in\ca^*$ is a dominant weight 
corresponding to a regular element $s\in\ca$ via the Killing form. That is, 
assume that $\alpha_i(s)=\inner{\alpha_i,\mu}$ for all fundamental roots 
$\alpha_i$. We call such a dominant weight \emph{regular}. Note that $\mu$ is regular if and only if $\inner{\alpha_i,\mu}>0$ for all fundamental roots $\alpha$. A non-regular dominant weight will be called \emph{irregular}.

Call a representation $\rho\in\Irr(G)$ \emph{regular} 
if it corresponds to a regular dominant weight. Define the zeta function of 
regular representations 
of $G$ by
\[
\zetareg_G(s)=\sum_{\substack{\rho\in\Irr(G)\\ \text{regular}}}\dim(\rho)^{-s}.
\]
A dominant weight is irregular if and only if it is 
orthogonal to a fundamental root. Let $S_j$ stand for the set of dominant weights that are 
orthogonal to the fundamental root $\alpha_j$. For a dominant weight $\mu$, let 
$\rho(\mu)$ denote the corresponding irreducible representation of $G$. We may 
then define the representation zeta function corresponding to the weights in $S_j$ by
\[
\zetairr{j}_G(s)=\sum_{\mu\in S_j}\dim(\rho(\mu))^{-s}.
\]
With the above notation, we may write
\[
\zeta_G=\zetareg_G+\sum_j\zetairr{j}_G.
\]
We show now that the irregular weights do not contribute to the abscissa of $\zeta_G$. The simplest system $A_1$ has no irregular weights, so in the following we shall only consider simple groups $G$ of other types.

\begin{prop}\label{prop:irregular_reps_of_Lie_groups}
The abscissa of convergence of $\sum_j\zetairr{j}_G$ is strictly smaller than 
$r/\kappa$. It follows that the representation growth rate of $G$ equals 
its regular representation growth rate.
\end{prop}

\begin{proof}
It suffices to consider $\zetairr{j}_G$ for an arbitrary $j$, and by changing the labeling of the fundamental roots if necessary, we may assume that $j=1$. We restrict the numerator of the Weyl dimension formula to those $\mu$ for which $\inner{\alpha_1,\mu}=0$. This results in the series
\begin{equation*}
\sum_{x_2,\dots,x_r=1}^\infty P(1,x_2,\dots,x_r)^{-s},
\end{equation*}
where $P$ is the polynomial defined in~\eqref{equation:witten_polynomial}. Write 
$P_1$ for $P(1,x_2,\dots,x_r)$ and $r_1=r-1$ for the number of variables in 
$P_1$. Write similarly $\kappa_1$ for the degree of $P_1$, which is the number 
of linear factors remaining in $P_1$ after the substitution $x_1=1$.
As there is exactly one factor in $P$ that contains only the variable $x_1$ 
(any fundamental root appears exactly once as a positive root in any root 
system), we have $\kappa_1=\kappa-1$.

Let $Q=P_1[x_{j_1},\dots,x_{j_l}]$ be any proper subpolynomial of $P_1$, and consider 
the subpolynomial $\bar{Q}=P[x_1,x_{j_1},\dots,x_{j_l}]$ of $P$. As above, we see that 
$r(Q)=r(\bar{Q})-1$ and $\kappa(Q)=\kappa(\bar{Q})-1$. Hence, using 
Lemma~\ref{lemma:2.5}, we get
\[
\frac{r_1-r(Q)}{\kappa_1-\kappa(Q)}
=\frac{r-r(\bar{Q})}{\kappa-\kappa(\bar{Q})}<\frac{r}{\kappa}.
\]
It now follows from Theorem~\ref{theorem:abscissa_convergence} that the abscissa 
of convergence for $\zetairr{j}_G$ is strictly smaller than $r/\kappa$, which 
proves 
the claims.
\end{proof}

\begin{rem}
There is an analogy between representation growth of Lie groups and finite groups of Lie type.
In \cite{Liebeck-Shalev05}, Martin Liebeck and Aner Shalev have considered a type 
of representation zeta function for finite simple groups $G(q)$ of a fixed Lie type, parametrised by the field 
of definition $\F_q$. Examining these finitely supported zeta functions $\zeta_{G(q)}(s)$, they show that there is a bound on convergence independent of $q$, and that this bound is found at $s=r/\kappa$.
We note that the result of Liebeck and Shalev depends on the same structural property of the irreducible root systems (Lemma~\ref{lemma:2.5}) which is, by our Corollary~\ref{corollary:2.5}, ultimately responsible for the abscissa in the case of compact Lie groups.
We hope that this might shed some light on a remark of Klopsch
\cite[Section~4]{Klopsch-rep-zeta-survey}, looking for a conceptual explanation 
of the fact that the abscissa of convergence is the same for Lie groups as for 
the finite groups of the same Lie type.
\end{rem}

\section{Twist representation growth of $\GLnO$ }\label{sec:Twist reps}
Let $\cO$ be a complete discrete valuation ring with finite residue field
$\F_q$ and maximal ideal $\mfp$. For $r\geq 1$ an integer, we will write $\cO_r$ for the finite ring $\cO/\mfp^r$. In this section and the next, we will consider the growth of equivalence classes of representations of $\GL_n(\mathcal{O})$ under one-dimensional twists. Since $\GL_n(\mathcal{O})$ has infinitely many one-dimensional representations, it does not have a representation zeta function. Nevertheless, we will prove that it does have a twist zeta function. Our main motivation for this is that it turns out that the abscissa of the twist zeta function is equal to the abscissa of the zeta function of $\SL_n(\mathcal{O})$, except in the case where the characteristic of $\cO$ divides $n$. 

We say that two representations $\rho,\sigma$ of a group $G$ are
\emph{twist equivalent} if there exists a one-dimensional representation
$\chi$ of $G$ such that $\rho\cong\sigma\otimes\chi$. The corresponding
equivalence classes of representations are called \emph{twist isoclasses}.
All the representations in a twist isoclass obviously have the same dimension
and we define the dimension of the twist isoclass to be the dimension
of any of its representations.
For a group $G$, let $\tilde{r}_{i}(G)$ denote the number of twist isoclasses (possibly infinite) of irreducible complex representations of $G$ of dimension $i$.
If $G$ is a topological group, we demand that the function
$\tilde{r}_{i}(G)$ only counts continuous representations. Moreover, in case
$\tilde{r}_{i}(G)$ is finite for all $i\geq 1$ we define
\[
\tilde{R}_{N}(G)=\sum_{i=1}^{N}\tilde{r}_{i}(G),
\]
for any integer $N\geq 1$.
Suppose that $\tilde{r}_{i}(G)$ is finite for all $i\geq 1$. Then the (twist)
representation zeta function of $G$ is defined to be
\[
\tilde{\zeta}_G(s)=\sum_{i=1}^{\infty}\tilde{r}_{i}(G)i^{-s},
\]
where $s$ is a complex variable.

A representation of $\GL_n(\cO_r)$ (or $\SL_n(\cO_r)$) is called \emph{primitive} if it does not factor through $\GL_n(\cO_{r-1})$ (or $\SL_n(\cO_{r-1})$). A representation which is not primitive will be called \emph{imprimitive}.

It is well known that $r_i(\SLnO)$ is finite for all $i\geq 1$ (see \cite{Jaikin-zeta}). The analogous statement for $\GLnO$ is false since the latter group has infinitely
many one-dimensional representations. However, we have:
\begin{lem}
\label{lem:twist-rigid}For any $d\in\N$, the number of twist isoclasses
of $\GLnO$ of dimension $d$ is finite.\end{lem}
\begin{proof}
Let $\mathcal{C}$ be a twist isoclass of dimension $d$. Let $r\in\N$
be the smallest natural number such that there exists a representation
$\rho\in\mathcal{C}$ which factors through $\GL_{n}(\cO_{r})$. Since there are
only finitely many representations of $\GL_n(\F_q)$, we may without loss of
generality assume that $r\geq 2$.
We will use Clifford theory for $\GL_{n}(\cO_{r})$ (see for example \cite{Alex_smooth_reps_GL2}) to find a lower bound on $d$. Let $l=\ceil{\frac{r}{2}}$ and $l'=\floor{\frac{r}{2}}$ and let $K^l$ be the maximal abelian congruence kernel of $\GL_{n}(\cO_{r})$. 
By Clifford's theorem, the restriction of $\rho$ to $K^l$ decomposes as
$$
\rho\mid_{K^l}=e\bigoplus_{\chi\in \Omega(\rho)} {}^{g}\chi,
$$
where $e\in \N$ and $\Omega(\rho)$ is an orbit of representations of $K^l$
under the conjugation action of $\GL_{n}(\cO_{r})$. Thus we can estimate $\dim
\rho$ from below by the size of the orbit $\Omega(\rho)$. The trace form induces a $\GL_{n}(\cO_{r})$-equivariant bijection between the irreducible representations of $K^l$ and $\M_n(\cO_{l'})$, so the size of $\Omega(\rho)$ equals the size of an orbit in $\M_n(\cO_{l'})$. By assumption, no twist of $\rho$ factors through $\GL_{n}(\cO_{r})$, and this means that the orbit in $\M_n(\cO_{l'})$ corresponding to $\Omega(\rho)$ is non-scalar. 
It is easy to see that $\M_{n-1}(\cO_{l'})\oplus \cO_{l'}$ is a centraliser in $\M_n(\cO_{l'})$ of maximal order. Thus the minimal order of an orbit in $\M_n(\cO_{l'})$ is
$$
|\M_n(\cO_{l'})|/|\M_{n-1}(\cO_{l'})\oplus \cO_{l'}| = q^{(2n-2)l'} \geq q^{r-1}.
$$
We conclude that $\dim \rho \geq q^{r-1}$, and thus we have shown that $\mathcal{C}$
contains a representation which is a pull-back
of a primitive representation of $\GL_{n}(\cO_{r})$, where
$r\leq 1+\log_q(d)$. Since there are only finitely many of the latter, there are
only finitely many possibilities for $\mathcal{C}$.
\end{proof}

We will now show that (at least when $p$ does not divide $n$) there is a close relationship between the representation growth of $\SL_n(\cO)$ and the twist representation
growth of $\GL_n(\cO)$, that is, between the growth rates of the two sequences $R_i(\SL_n(\cO))$ and $\tilde{R}_{i}(\GL_n(\cO))$.
For any group $G$ and any $i\in\N$, we let $\Irr_{i}(G)$ denote
the set of isomorphism classes of irreducible complex
representations of $G$ of dimension $i$. Similarly, 
\[
\widetilde{\Irr}_{i}(G)
\]
denotes the set of twist isoclasses of continuous irreducible complex representations of $G$ of dimension $i$.

In the following two lemmas, $F$ denotes the field of fractions of $\cO$.
\begin{lem}
\label{lem:O/On-finite} The group
$\cO^{\times n}$ of $n$-th powers has
finite index in $\cO^{\times}$ if and only if $\chara F$ does not
divide $n$.\end{lem}
\begin{proof}
The group $\cO^{\times}/\cO^{\times n}$ embeds into $F^{\times}/F^{\times n}$.
By \cite[I~(5.9)]{ivan}, the group $F^{\times}/F^{\times n}$
is finite if $\chara F$ does not
divide $n$. For the ``only if'' direction, note that it indeed
follows from the proof in loc.~cit.
\end{proof}

\begin{lem}
\label{lem:Irr(O)-mu_n}Let $\mu_{n}$ be the group of $n$-th roots
of unity of $F$. Then
\[
\Irr(\cO^{\times})/\Irr(\cO^{\times})^{n}\cong\Irr(\mu_{n}).
\]
In particular, the order of $\Irr(\cO^{\times})/\Irr(\cO^{\times})^{n}$
is $|\mu_{n}|$.
\end{lem}
\begin{proof}
Applying the contravariant functor $\Hom(\,\cdot\,,\C^{\times})$
to the exact sequence
\[
1\longrightarrow\mu_{n}\longrightarrow\cO^{\times}\xrightarrow{\,[n]\,}\cO^{\times n}\longrightarrow1,
\]
where $[n]$ is the $n$th power map, we obtain the exact sequence
\[
\Irr(\cO^{\times n})\xrightarrow{\Hom([n],\C^{\times})}\Irr(\cO^{\times})\longrightarrow\Irr(\mu_{n})\longrightarrow1,
\]
where the first map $\Hom([n],\C^{\times})$ is given by $f\mapsto f\circ[n]$.
If $f\in\Irr(\cO^{\times n})$ is such that $f\circ[n]=1$, then
$f(x^{n})=1$ for all $x\in\cO^{\times}$, that is, $f=1$. Hence,
$\Hom([n],\C^{\times})$ is an injection, and its image coincides with
$\Irr(\cO^{\times})^{n}$, since $f\circ[n]=f^{n}$ in $\Irr(\cO^{\times})$.
Since $\mu_{n}$ is a finite abelian group, $|\mu_{n}|=|\Irr(\mu_{n})|$,
so we are done.
\end{proof}
\begin{prop}\label{prop:abscSLn-twistGLn}
	Suppose that the characteristic of $\cO$ does not divide $n$. Then
	the abscissa of convergence of $\twistGLn$ exists and is equal to that of $\zeta_{\SLnO}$.
\end{prop}
\begin{proof}
	Let $Z$ be the centre of $\GLnO$, that is, the subgroup of scalar
	matrices. For ease of notation, set $G=\GLnO$ and $S=\SLnO$. Composing
	the determinant $\det\colon G\rightarrow\cO^{\times}$ with the quotient
	map $\cO^{\times}\rightarrow\cO^{\times}/\cO^{\times n}$ gives rise
	to the exact sequence 
	\[
	1\longrightarrow ZS\longrightarrow G\longrightarrow\cO^{\times}/\cO^{\times n}\longrightarrow1.
	\]
	Note that the kernel is exactly $ZS$ since for any $g\in G$, such
	that $\det(g)=\lambda^{n}$ for some $\lambda\in\cO^{\times}$, we
	have $\lambda^{-1}g\in S$, and hence $g\in ZS$. By Lemma~\ref{lem:O/On-finite},
	the index $a:=[G:ZS]$ is finite. 
	
	The proof is in two steps: First we will show that $R_N(S)$ has the same rate of polynomial growth as $\tilde{R}_N(ZS)$, and then that the latter has the same rate of polynomial growth as $\tilde{R}_N(G)$, which will prove the equality of the abscissae.
	
	First, we claim that, for all $N$,
	\begin{equation}
	R_{N}(S)\leq\tilde{R}_{N}(ZS)\leq|\mu_{n}|\cdot R_{N}(S),\label{eq:R_N-estimates_H-ZH}
	\end{equation}
	where $\mu_n$ is the group of $n$-th roots of unity of $F$.
	We now prove this claim. Every one-dimensional representation of $G$
	is of the form $\chi\circ\det$ for some homomorphism $\chi\colon\cO^{\times}\rightarrow\C^{\times}$,
	and since the commutator subgroup of $ZS$ is $S$, every one-dimensional
	representation of $ZS$ is also of the form $\chi\circ\det$. Any
	$\tau\in\Irr(S)$ can be extended to a representation $\rho\tau$
	of $ZS$, where $\rho$ is an extension to $Z$ of the central character
	of $\tau$, and $(\rho\tau)(zs)=\rho(z)\tau(s)$, for $z\in Z,s\in S$.
	Thus, we clearly have 
	\[
	R_{N}(S)\leq\tilde{R}_{N}(ZS).
	\]
	Conversely, every $\pi\in\Irr(ZS)$ is an extension of an irreducible
	representation of $S$, because $\pi$ factors through some finite
	quotient $Z_{r}S_{r}$, where $Z_{r}$ is the centre of $\GL_{n}(\cO_{r})$
	and $S_{r}=\SL_{n}(\cO_{r})$, and it is well known that every $\pi\in\Irr(Z_{r}S_{r})$
	is an extension of an irreducible representation of $S_{r}$.
	
	Let $\rho_{1}$ and $\rho_{2}$ be two extensions to $Z$ of the central character of $\tau\in\Irr(S)$. Identify $\rho_{1}$ and $\rho_2$ with their corresponding homomorphisms $\cO^{\times}\rightarrow\C^{\times}$.
	For $\lambda I\in Z$, $\lambda\in\cO^{\times}$, we have $(\rho_{1}\otimes\chi\circ\det)(\lambda I)=\rho_{1}(\lambda)\chi(\lambda)^{n}$, for all homomorphisms $\chi:\cO^{\times}\rightarrow\C^{\times}$,
	so if $\rho_{1}$ and $\rho_{2}$ have the same image in $\Irr(Z)/\Irr(Z)^{n}$,
	then $\rho_{1}\tau$ is in the same twist isoclass as $\rho_{2}\tau$.
	Thus, by Lemma~\ref{lem:Irr(O)-mu_n}, there exist no more than
	\[
	|\Irr(Z)/\Irr(Z)^{n}|=|\Irr(\cO^{\times})/\Irr(\cO^{\times})^{n}|=|\mu_{n}|
	\]
	twist isoclasses of extensions of $\tau$ to the group $ZS$, that
	is,
	\[
	\tilde{R}_{N}(ZS)\leq|\mu_{n}|\cdot R_{N}(S).
	\]
	
	The goal in the remaining part of the proof is to show that 
	\begin{equation}
	\tilde{R}_{N}(ZS)\leq a\tilde{R}_{aN}(G)\quad\text{and}\quad\tilde{R}_{N}(G)\leq a\tilde{R}_{N}(ZS).\label{eq:R_N-estimates_ZH-G}
	\end{equation}
	The proof of this proceeds in a way similar to \cite[Lemma~2.2]{Lubotzky-Martin},
	but with twists taken into account. For any representation $\pi$
	of $ZS$ or $G$, let $[\pi]$ denote its twist isoclass. For each
	$[\tau]\in\bigcup_{m=1}^{N}\widetilde{\Irr}_{m}(ZS)$, choose an irreducible
	component $\psi(\tau)$ of $\Ind_{ZS}^{G}\tau$. The formula $\Ind_{ZS}^{G}(\tau\otimes\chi\circ\det|_{ZS})\cong(\Ind_{ZS}^{G}\tau)\otimes\chi\circ\det$
	implies that this induces a well defined function
	\begin{align*}
		\psi\colon\bigcup_{m=1}^{N}\widetilde{\Irr}_{m}(ZS) & \longrightarrow\bigcup_{m=1}^{Na}\widetilde{\Irr}_{m}(G),\\{}
		[\tau] & \longmapsto[\psi(\tau)].
	\end{align*}
	Let now $[\tau]\in\bigcup_{m=1}^{N}\widetilde{\Irr}_{m}(ZS)$ be an
	element such that $\dim[\tau]=\dim\tau$ is minimal among the dimensions
	of the elements of the fibre $\psi^{-1}([\psi(\tau)])$. Frobenius
	reciprocity implies that for any $[\tau']\in\psi^{-1}([\psi(\tau)])$,
	the representation $\tau'$ is an irreducible component of $\psi(\tau)|_{ZS}$.
	Hence $\dim\psi(\tau)\geq|\psi^{-1}([\psi(\tau)])|\dim\tau$. But
	$\dim\psi(\tau)\leq a\dim\tau$, which implies that 
	\[
	|\psi^{-1}([\psi(\tau)])|\leq a.
	\]
	whence $\tilde{R}_{N}(ZS)\leq a\tilde{R}_{aN}(G)$.
	
	Next, define a function 
	\[
	\phi\colon\bigcup_{m=1}^{N}\widetilde{\Irr}_{m}(G)\longrightarrow\bigcup_{m=1}^{N}\widetilde{\Irr}_{m}(ZS)
	\]
	by choosing, for each $[\sigma]\in\bigcup_{m=1}^{N}\widetilde{\Irr}_{m}(G)$,
	an irreducible component of $\sigma|_{ZS}$ (thanks to the
	formula $(\sigma\otimes\chi\circ\det)|_{ZS}=\sigma|_{ZS}\otimes (\chi\circ\det|_{ZS})$,
	this induces a well defined function on the sets of twist isoclasses). Let $[\sigma]\in\bigcup_{m=1}^{N}\widetilde{\Irr}_{m}(G)$ be such
	that $\dim\sigma=\dim[\sigma]$ is minimal among the dimensions of the elements
	of the fibre $\phi^{-1}(\phi([\sigma]))$. By Frobenius reciprocity,
	every element in $\phi^{-1}(\phi([\sigma]))$ contains an irreducible
	component of $\Ind_{ZS}^{G}\phi(\sigma)$. Hence $|\phi^{-1}(\phi([\sigma]))|\dim\sigma\leq a\dim\phi([\sigma])$. But $\dim\phi([\sigma])\leq\dim\sigma$, so 
	\[
	|\phi^{-1}(\phi([\sigma]))|\leq a,
	\]
	which proves that $\tilde{R}_{N}(G)\leq a\tilde{R}_{N}(ZS)$.
	
	Now, the abscissa of convergence of $\zeta_{\SLnO}$ exists (see \cite{Lubotzky-Martin}), so by \eqref{eq:R_N-estimates_H-ZH} and the second inequality in \eqref{eq:R_N-estimates_ZH-G}, the abscissa of convergence of $\twistGLn$ exists. Moreover, since the abscissae of convergence are the 
rates of polynomial growth of $R_N(S)$ and $\tilde{R}_{N}(G)$, respectively, 
 \eqref{eq:R_N-estimates_H-ZH}
	and \eqref{eq:R_N-estimates_ZH-G}, imply that they are equal.
\end{proof}
We do not know whether the above proposition holds when $\chara\cO$ divides $n$.

\begin{cor}\label{cor:abs-lower_bound_p_good}
	Suppose that the characteristic of $\cO$ does not divide $n$. Then the abscissa of convergence of $\twistGLn$ is at least $2/n$.
\end{cor}
\begin{proof}
	By Proposition~\ref{prop:abscSLn-twistGLn}, the abscissa of $\twistGLn$ is equal to the abscissa of $\zeta_{\SLnO}$, and by \cite[Proposition~6.6]{Larsen-Lubotzky}, the latter is at least $r/\kappa=2/n$.
\end{proof}

\section{Counting the twist isoclasses of $\GL_{2}(\cO)$}\label{sec:twist-GL2}

We continue to let $\cO$ be a complete discrete valuation ring with finite residue field
$\F_q$ and maximal ideal $\mfp$. Let $p$ be the characteristic of $\F_{q}$.

In this section we set up our general approach to counting twist isoclasses
of representations of $\GL_{2}(\cO)$ using a Clifford theoretic description of the representations. Our goal is to compute the twist zeta function of $\GL_{2}(\cO)$ exactly when $p$ is odd, and to approximate it well enough to compute its abscissa of convergence when $p$ is even. 

For $r\geq1$, let $G_{r}=\GL_{2}(\cO_{r})$.
We will describe the twist isoclasses of representations of each group
$G_{r}$. We start with the case $r=1$, which is rather different from the case $r\geq 2$.

\subsection{The twist isoclasses of $\GL_2(\F_q)$}
Counting the number of twist isoclasses in this case is made possible thanks to the following
result, which is an analogue of the formula
$\Ind_{H}^{G}(\theta\chi|_{H})=(\Ind_{H}^{G}\theta)\chi$
for Deligne--Lusztig induction in the case, where $\chi$ is $1$-dimensional.
\begin{lem}
\label{lem:DL-induction-twist}Let $G$ be a connected reductive group
defined over a finite field $\F_{q}$ with corresponding Frobenius
endomorphism $F$. Let $T$ be an $F$-stable maximal torus in $G$, let
$\theta\colon T^{F}\rightarrow\C^{\times}$ be an irreducible character,
and let $R_{T}^{G}(\theta)$ denote the corresponding Deligne--Lusztig
character of $G^{F}$. Then, for any $1$-dimensional character $\chi$
of $G^{F}$, we have
\[
R_{T}^{G}(\theta\chi|_{T^{F}})=R_{T}^{G}(\theta)\chi.
\]
\end{lem}
\begin{proof}
Let $g\in G^{F}$ have Jordan decomposition $g=su$, where $s$ is
semisimple and $u$ unipotent. By the character formula for Deligne--Lusztig
characters \cite[Theorem~7.2.8]{carter}, we have
\[
R_{T}^{G}(\theta)(g)=\frac{1}{|C_{G}(s)^{\circ F}|}\sum_{\substack{x\in G^{F}\\
x^{-1}sx\in T^{F}
}
}\theta(x^{-1}sx)R_{xTx^{-1}}^{C_{G}(s)^{\circ}}(\mathbf{1})(u).
\]
Thus,
\begin{align*}
R_{T}^{G}(\theta\chi|_{T})(g) & =\frac{1}{|C_{G}(s)^{\circ F}|}\sum_{\substack{x\in
G^{F}\\
x^{-1}sx\in T^{F}
}
}\theta(x^{-1}sx)\chi(x^{-1}sx)R_{xTx^{-1}}^{C_{G}(s)^{\circ}}(\mathbf{1})(u)\\
 & =\frac{1}{|C_{G}(s)^{\circ F}|}\sum_{\substack{x\in G^{F}\\
x^{-1}sx\in T^{F}
}
}\theta(x^{-1}sx)\chi(s)R_{xTx^{-1}}^{C_{G}(s)^{\circ}}(\mathbf{1})(u)\\
 & =R_{T}^{G}(\theta)(g)\chi(s).
\end{align*}
It now remains to observe that any $1$-dimensional representation
$\chi$ of $G^{F}$ is trivial on unipotent elements, so that $\chi(s)=\chi(g)$.
Indeed, by \cite[Proposition~17.2]{humphreys} the derived group $[G,G]$
is closed and connected. Since $F([x,y])=[F(x),F(y)]$ the subgroup
$[G,G]$ is also $F$-stable. Thus by \cite[Corollary~3.13]{dignemichel}
$(G/[G,G])^{F}\cong G^{F}/[G,G]^{F}$. On the other hand, we have
$G=[G,G]Z$, where $Z$ is a central torus of $G$. Thus $G/[G,G]$
injects into $Z$, so $(G/[G,G])^{F}$, and hence $G^{F}/[G,G]^{F}$,
has no non-trivial unipotent elements. Thus $\chi$ is trivial on
any unipotent element of $G^{F}$.\end{proof}
\begin{lem}
\label{lem:twist-zetaGL2(Fq)}The twist zeta function of $\GL_{2}(\F_{q})$
is
\[
\tilde{\zeta}_{\GL_{2}(\F_{q})}(s)=\begin{cases}
1+q^{-s}+\frac{1}{2}(q-1)(q+1)^{-s}+\frac{1}{2}(q+1)(q-1)^{-s} & \text{if
}p\neq2,\\
1+q^{-s}+\frac{1}{2}(q-2)(q+1)^{-s}+\frac{1}{2}q(q-1)^{-s} & \text{if }p=2.
\end{cases}
\]
\end{lem}
\begin{proof}
Let $\overline{\F}_{q}$ be an algebraic closure of $\F_q$ and denote $G=\GL_{2}(\overline{\F}_{q})$ with its standard $\F_{q}$-rational structure given by the Frobenius endomorphism $F$.
It is well known (see e.g.~\cite[Section~28]{JamesLiebeck}) that there is only one twist isoclass of each of $1$- and $q$-dimensional representations of $G^F=\GL_2(\F_{q})$, respectively; that is,
\[
\tilde{r}_1(G^F)=\tilde{r}_q(G^F)=1.
\]
 It is also well known  (see e.g.\ \cite[15.9]{dignemichel})
that all the remaining irreducible representations are obtained as the
Deligne--Lusztig representations $\pm R_{T}^{G}(\theta)$, where $T$ is one of the two
$\F_{q}$-rational maximal tori of $G$ and $\theta\in\Irr(T^{F})$ is in general position, that is, $^{w}\theta\neq\theta$, where $w$ is the non-trivial element in the Weyl group $W$ with respect to $T$ (note that we have $W=W^F$ in the present situation). 
More precisely, let $T$ be an $\F_q$-rational maximal torus of $G$ and define
\[
X_T=\{\theta W\in\Irr(T^F)/W\mid{}^{w}\theta\neq\theta\}
\]
to be the set of orbits of characters of $T^{F}$ in general position,
modulo the action of the group $W$. Then there is a bijection
between $X_T$ and the irreducible representations of the form $\pm R_{T}^{G}(\theta)$. Moreover, when $T$ is split, $R_{T}^{G}(\theta)$ has dimension $q+1$ and when $T$ is non-split the dimension of $-R_{T}^{G}(\theta)$ is $q-1$. 

We now show the following:
\begin{claim}
Restriction defines an injective homomorphism $\Irr_1(G^F)\rightarrow\Irr(T^F)$ from one-dimensional representations of $G^F$ to irreducible representations of $T^F$. The image of this homomorphism consists of representations in $\Irr(T^F)$ which are not in general position. 
\end{claim}
We first show injectivity. If $\chi$ is a one-dimensional
representation of $G^{F}$, then $\chi$ factors
through the determinant $\det\colon G^{F}\rightarrow\F_{q}^{\times}$.
The restriction of the determinant to $T^{F}$ is surjective,
as is easily seen directly for the split torus and follows from the
surjectivity of the norm map $\F_{q^{2}}^{\times}\rightarrow\F_{q}^{\times}$
for the non-split torus. Thus $\chi|_{T^F}$ is trivial
if and only if $\chi$ is trivial.

Next, if
$\theta=\chi|_{T^{F}}$ for $\chi\in\Irr_1(G^F)$,
then for any $t\in T^{F}$ we have
\[
^{w}\theta(t)=\theta(wtw^{-1})=\chi(wtw^{-1})=\chi(t)=\theta(t).
\]
This proves the claim.

Let $\Gamma$ be the image of the homomorphism in the above claim, that is, $\Gamma$ is the group of representations of the form $\chi|_{T^F}$, for $\chi\in\Irr_1(G^F)$. By the above claim, $\Gamma$ has $q-1$ elements and acts on $X$ by multiplication (this gives a well defined action since every element in $\Gamma$ is fixed by the action of $W$, by the claim).

Lemma~\ref{lem:DL-induction-twist} together with the above claim implies that there is a bijection between twist isoclasses of irreducible representations of the form $\pm R_{T}^{G}(\theta)$ and orbits $X_T/\Gamma$. To count the number of twist isoclasses, we will now compute the number of orbits $X_T/\Gamma$.

Let $\theta\in\Irr(T^F)$ and $\gamma\in\Gamma$. Then $\theta\gamma=\theta$
if and only if $\gamma=\mathbf{1}$ (the trivial character). On the
other hand, assume that $\theta\gamma={}^{w}\theta$. Then
\[
\gamma^{-1}=({}^{w}\theta\theta^{-1})^{-1}=\theta({}^{w}\theta^{-1})={}^{w}({}^
{w}\theta\theta^{-1})={}^{w}\gamma=\gamma,
\]
so $\gamma^{2}=\mathbf{1}$.  Thus an element $\theta W\in X_T$
is fixed by $\gamma\in\Gamma$ only if $\gamma=\mathbf{\pm1}$, where $-\mathbf{1}$ denotes the non-trivial character of $T^F$ whose square is $\mathbf{1}$. Note that when $p=2$, we have $\mathbf{1}=-\mathbf{1}$. 

Assume now that $p\neq2$ and $\gamma=-\mathbf{1}$. We determine the number
of fixed points of $-\mathbf{1}$ in $X_T$ in this case. The map $\theta\mapsto {}^w\theta\cdot\theta^{-1}$ is a group homomorphism and its kernel $\{\theta\in\Irr(T^F)\mid {}^w\theta=\theta \}$  has order 
\[
|T^F|-2|X_T|.
\]
Since the set $\{\theta\in \Irr(T^F)\mid {}^w\theta=(-\mathbf{1})\theta\}$ is a coset of the kernel, it has the same size as the kernel. It follows that the number of fixed points $X_T^{\gamma}$, where $\gamma=-\mathbf{1}$, that is, the number of $W$-orbits of characters in general position fixed by $\gamma=-\mathbf{1}$,  is equal to $\frac{1}{2}|T^F|-|X_T|$.
The Frobenius--Burnside formula now implies that for any $T$, the number of orbits
in $X_T$ under the action of $\Gamma$ is
\begin{equation}
|X_T/\Gamma|=\begin{cases}
\frac{|T^F|}{2(q-1)} & \text{if }p\neq2,\\[5pt]
\frac{|X_T|}{q-1} & \text{if }p=2.
\end{cases}\label{eq:num-twists}
\end{equation}
It is well known (see e.g. \cite[15.9]{dignemichel}) that, if $T$ is the split torus, then $|X_T|=\frac{(q-1)(q-2)}{2}$,
and if $T$ is the non-split torus, then $|X_T|=\frac{q(q-1)}{2}$.
 Thus, by \eqref{eq:num-twists}, we have
\[
\tilde{r}_{q+1}(G^{F})=\begin{cases}
(q-1)/2 & \text{if }p\neq2,\\
(q-2)/2 & \text{if }p=2,
\end{cases}\qquad\tilde{r}_{q-1}(G^{F})=\begin{cases}
(q+1)/2 & \text{if }p\neq2,\\
q/2 & \text{if }p=2.
\end{cases}
\]
This gives the twist zeta function of $G^{F}$, as asserted.
\end{proof}

\subsection{The representations of $G_r$ with $r\geq2$}

Assume from now on that $r\geq2$. To describe the representations
of $G_{r}=\GL_{2}(\cO_{r})$ we use the construction of regular representations in \cite{SS/2017}. Knowing the regular representations is enough because for $G_r$ any irreducible representation is either regular or factors through $G_{r-1}$, up to twisting.

Let $l=\ceil{\frac{r}{2}}$ and $l'=\floor{\frac{r}{2}}$, so that $r=l+l'$. To any $\pi\in\Irr(G_{r})$, we associate its $G_{l'}$-conjugacy orbit in
$\M_{2}(\cO_{l'})$; see for example \cite{Alex_smooth_reps_GL2}.
First note that if the orbit of $\pi$ is scalar
mod $\mfp$, it means that the twist isoclass of $\pi$ contains an imprimitive representation. We therefore have two cases: either the twist isoclass contains an imprimitive representation,
or its orbit is non-scalar mod $\mfp$, in which case it is primitive. Without loss of generality,
we focus on the latter case. The orbits in this case are of three
different types, represented by the following matrices in $\M_{2}(\cO_{l'})$:
\begin{enumerate} \itemsep.5em
\item $\begin{bmatrix}a & 0\\
0 & d
\end{bmatrix}$, where $a-d\notin\mfp,$
\item $\begin{bmatrix}0 & 1\\
-\Delta & \tau
\end{bmatrix}$, where the characteristic polynomial $x^{2}+\tau x+\Delta$ is
irreducible
mod $\mfp$.
\item $\begin{bmatrix}0 & 1\\
-\Delta & \tau
\end{bmatrix}$, where $x^{2}+\tau x+\Delta\equiv(x-a)^{2}\bmod\mfp$, for some
$a\in\cO_{l'}$.
\end{enumerate}
We will say that a matrix of one of the above forms is of type 1,2 or 3, respectively.
For $\beta\in\M_{2}(\cO_{l'})$, we denote its \emph{orbit} (i.e, the $G_{l'}$-conjugacy
class of $\beta$) by $[\beta]$. 

For $1\leq i\leq r$, let $K^i$ be the kernel of the map $G_r\rightarrow G_i$. Each matrix in $\M_{2}(\cO_{l'})$ defines a unique one-dimensional character $\psi_{\beta}$ of $K^l$; see \cite{Alex_smooth_reps_GL2}.
In the following, if $G$ is any finite group, we write $\Irr(G)$ for the
set of irreducible characters of $G$. Let $H\subseteq G_{r}$ be
a subgroup containing $K^{l}$. For any $\beta\in\M_{2}(\cO_{l'})$, we denote
\[
\Irr(H\mid\beta)=\{\pi\in\Irr(H)\mid\langle\pi|_{K^l},\psi_{\beta}\rangle\neq 0\}.
\]
We will usually fix a $\beta\in\M_{2}(\cO_{l'})$ and a lift $\hat{\beta}\in\M_{2}(\cO_{r})$. In this situation, we define the centralisers
\[
\centbi=C_{G_i}(\beta_i),\qquad \centbr=C_{G_r}(\hat{\beta}),
\]
where $1\leq i \leq r-1$ and $\beta_i$ denotes the image of $\hat{\beta}$ in $\M_2(\cO_i)$. Similarly, we set
\[
C^i = \centbr \cap K^i.
\]
We summarise the construction of regular characters of $\GL_{2}(\cO_r)$
for $r\geq2$ in the following lemma. This is a slight reformulation of  \cite[Theorem~4.10]{SS/2017}.

\begin{lem}
\label{lem:Constr-repsGL2}Let $\beta\in\M_{2}(\cO_{l'})$ be of one
of the three types above, and take any lift $\hat{\beta}\in\M_{2}(\cO_{r})$
of $\beta$. For any $\pi\in\Irr(G_r\mid\beta)$ there exists an extension $\theta$ of $\psi_{\beta}$ to $\centbr^1 K^{l}$,
a unique irreducible character $\eta_{\theta}$ of $\centbr^1 K^{l'}$ containing $\theta$,
and an extension $\hat{\eta}_{\theta}$ of $\eta_{\theta}$ to $\centbr K^{l'}$
such that
\[
\pi=\Ind_{\centbr K^{l'}}^{G_{r}}\hat{\eta}_{\theta}.
\]
Moreover, this establishes a bijection (depending on the choice of
$\theta$ and $\hat{\eta}_{\theta}$):
\begin{gather*}
\Irr(\centbr^{1} K^{l}/K^{l})\times\Irr(\centbr K^{l'}/\centbr^{1} K^{l'}) 
\longleftrightarrow\Irr(G_{r}\mid\beta)\\
(\omega,\lambda) \longmapsto\Ind_{\centbr
K^{l'}}^{G_{r}}\hat{\eta}_{(\theta\omega)}\lambda.
\end{gather*}
\end{lem}
Using the above lemma, it is relatively easy to compute the dimensions and multiplicities of the irreducible representations of $G_r$:
\begin{lem}
\label{lem:reps-dim-mult-GL2}
Let 
\[
d_r(i)=
\begin{cases}
(q+1)q^{r-1} & \text{if $i=1$},\\
(q-1)q^{r-1} & \text{if $i=2$},\\
(q^{2}-1)q^{r-2} & \text{if $i=3$}.
\end{cases}
\]
Then $\pi\in\Irr(G_r\mid\beta)$ has dimension $d_r(i)$ if and only if $\beta$ is of type $i\in\{1,2,3\}$. Moreover, we have
\[
r_{d_r(i)}(G_r)=
\begin{cases}
\frac{1}{2}(q-1)^3 q^{2r-3} & \text{if $i=1$},\\
\frac{1}{2}(q-1)^2 (q+1) q^{2r-3} & \text{if $i=2$},\\
(q-1)q^{2r-2} & \text{if $i=3$}.
\end{cases}
\]
\end{lem}

\begin{proof}
It follows from the proof of Lemma~\ref{lem:Constr-repsGL2} in \cite{SS/2017} that
\begin{equation}\label{eq:dim-eta}
\dim \hat{\eta}_{\theta} = \left|\frac{\centbr^1 K^{l'}}{\centbr^1 K^{l}}\right|^{1/2}
= \left|\frac{K^{l'}}{\centbr^{l'} K^{l}}\right|^{1/2}
= \left|\frac{K^{l'}/K^l}{\centbr^{l'} K^{l}/K^l}\right|^{1/2}
= \left|\frac{q^{4(l-l')}}{q^{2(l-l')}}\right|^{1/2}
= q^{l-l'}.
\end{equation}
Thus, Lemma~\ref{lem:Constr-repsGL2} implies that
\begin{align*}
\dim \pi &= [G_r:\centbr K^{l'}]\cdot q^{l-l'} =[G_{l'}:\centbl]\cdot q^{l-l'}
 = \frac{q(q-1)^2(q+1)q^{4(l'-1)}}{|C_1|\cdot q^{2(l'-1)}}\cdot q^{l-l'} \\
& = \frac{q(q-1)^2(q+1)q^{r-2}}{|C_1|}.
\end{align*}
Since
\[
|C_1|=
\begin{cases}
(q-1)^2 & \text{if $\beta$ is of type $1$},\\
q^2-1 & \text{if $\beta$ is of type $2$},\\
(q-1)q & \text{if $\beta$ is of type $3$},
\end{cases}
\]
the first assertion follows. 

For the assertion about multiplicities, a straightforward computation shows that the number of orbits of each type is given by
\begin{equation}
\label{eq: number-of-orbits}
\#\{[\beta]\mid\text{$\beta$ of type $i$}\}=
\begin{cases}
\frac{1}{2}(q-1)q^{2l'-1} & \text{if $i=1$ or $i=2$},\\
q^{2l'-1} & \text{if $i=3$}.
\end{cases}
\end{equation}
Thus, when $\beta$ is of type 1, Lemma~\ref{lem:Constr-repsGL2} implies that
\begin{align*}
r_{d_r(1)}(G_r) &= \frac{1}{2}(q-1)q^{2l'-1}\cdot |\centbr^{1} K^{l}/K^{l}|\cdot |\centbr K^{l'}/\centbr^{1} K^{l'}|\\
&= \frac{1}{2}(q-1)q^{2l'-1}\cdot q^{2(l-1)}\cdot|C_{l'}/C_{l'}^{1}|\\
&= \frac{1}{2}(q-1)\cdot q^{2r-3}\cdot |C_1|\\
&= \frac{1}{2}(q-1)^3 q^{2r-3}.
\end{align*}
Here, for the third equality, we have used the fact that $\rho_{l'}$ maps $C_{l'}$ surjectively onto $C_1$.
Similarly, when $\beta$ is of type 2 and 3, respectively, we get
\begin{align*}
r_{d_r(2)}(G_r) &=  \frac{1}{2}(q-1)q^{2l'-1}\cdot q^{2(l-1)}\cdot |C_1| = \frac{1}{2}(q-1)^2 (q+1) q^{2r-3},\\
r_{d_r(3)}(G_r) &= q^{2l'-1}\cdot q^{2(l-1)}\cdot q(q-1)=(q-1)q^{2r-2}.\qedhere
\end{align*}
\end{proof}

\subsection{Generalities on twist isoclasses of $G_r$}\label{sec:Generalities-twist}
We will now consider twist isoclasses of representations of $G_r$.
The additive group $\cO_{l'}$ acts on the set of orbits of $\beta\in\M_{2}(\cO_{l'})$ via
\[
(x,[\beta])\longmapsto[xI+\beta],
\]
where $x\in\cO_{l'}$ and $I$ is the identity matrix. We denote the
orbit of $[\beta]$ under this action by $[[\beta]]$, and will refer
to this as the \emph{twist orbit} of $\beta$. It is clear that this
action preserves each of the three types of matrices above. 
A regular class $[\beta]$ (i.e.\ one which is non-scalar mod $\mfp$) is fixed by an element $x\in\cO_{l'}$ if and only if $x$ fixes the trace and determinant of $\beta$, that is, if the
following equations hold:
\begin{equation}\label{eq:fixedpoints-equations}
\begin{cases}
2x=0\\
x(x+\tau)=0.
\end{cases}
\end{equation}

Let $e$ denote the ramification index of $2$ in $\cO$, that is,
we have $2\cO=\mfp^{e}$. In particular, if $\chara\cO=2$, we set $e=\infty$.
If $\chara\cO=0$ and $l'\geq 1$, we have $\chara\cO_{l'}=2^{m}$,
where $m$ is the smallest integer such that $em\geq l'$, that is,
$m=\lceil l'/e\rceil$. Note in particular that $\chara\cO_{l'}=2$
whenever $l'\leq e$, so that if $\chara\cO_{l'}=2^{m}$ with $m\geq2$,
we necessarily have $l'>e$.

From now on, let 
\[
B_i=\#\{[[\beta]]\mid \text{$\beta$ of type $i$}\},
\]
so $B_i$ denotes the number of twist orbits of type $i$ for each $i\in\{1,2,3\}$. We will need to compute the numbers $B_i$ and consider different cases depending on whether the level $l'$ is below or above the ramification index $e$.

If $p$ is odd and $x\in \cO_{l'}$ fixes the orbit $[\beta]$, the first equation in \eqref{eq:fixedpoints-equations} implies that $x=0$. Thus, $\#[[\beta]]=|\cO_{l'}|=q^{l'}$ for all $\beta$, so when $p\neq 2$, we have 
\begin{equation}\label{eq:numb-of-twist-orbits}
B_i=\frac{\#\{[\beta]\mid\text{$\beta$ of type $i$}\}}{q^{l'}}
=\begin{cases}
\frac{1}{2}(q-1)q^{l'-1} & \text{for $i\in\{1,2\}$}\\
q^{l'-1} & \text{for $i=3$}
\end{cases}
\end{equation}
(see \eqref{eq: number-of-orbits} in the proof of Lemma~\ref{lem:reps-dim-mult-GL2} for the number of orbits of type $i$). 

When $p=2$, it is more difficult to compute the numbers $B_i$, especially when $i=3$.
\begin{lem}\label{lem:orbit_numbers}
	Suppose that $p=2$, and recall that $e=\infty$ if $\chara\cO=2$. Then
	\begin{enumerate}[label=(\alph*)]
		\item If $l'\le e$, we have
		\[
		B_i=
		\begin{cases}
		(q-1)q^{l'-1} & \text{for $i\in\{1,2\}$}, \\
		((l'-1)(q-1)+1)q^{l'-1} & \text{for $i=3$}.
		\end{cases}
		\]
		\item If $l'>e$, we have
		\[
		B_i=
		\begin{cases}
		\frac{1}{2}(q-1)q^{l'-1} & \text{for $i\in\{1,2\}$},\\
		(e(q-1)+1)q^{l'-1} & \text{for $i=3$}.
		\end{cases}
		\]
	\end{enumerate}
\end{lem}

\begin{proof}
	First we deduce some facts about $\tau=\tr(\beta)\in\cO_{l'}$. In type~1, as the residue characteristic is 2, the condition $a\not\equiv d\bmod\mfp$ implies that $a+d\not\equiv 0\bmod\mfp$. Therefore $\tau$ is a unit. In type~2, the characteristic polynomial $x^2+\tau x+\Delta$ is irreducible modulo $\mfp$. Because the residue characteristic is 2, this implies that modulo $\mfp$, we must have $\tau\not\equiv 0$. Therefore $\tau$ is again a unit. In type~3, it is clear that $\tau\equiv 0\bmod\mfp$, so $\tau$ is not a unit.
	
	(a) Assume that $l'\le e$, so that $\chara\cO_{l'}=2$. We intend to use the Frobenius--Burnside orbit counting formula, so we need to find the number of orbits fixed by any given scalar. Suppose therefore that $x\in\cO_{l'}$ fixes $[\beta]$ for some $\beta\in\M_2(\cO_{l'})$ of type 1 or 2. By the equations in~\eqref{eq:fixedpoints-equations}, we have $2x=0$ and $x(x+\tau)=0$. Here, the first equation $2x=0$ is trivially satisfied. If $x$ is a unit, it follows from the second equation that $x=\tau$. On the other hand, if $x\in \mfp$, then, as $\tau$ is a unit in types 1 and 2, we see that $x+\tau$ is a unit, and it follows that $x=0$.
	
	Now, the scalar $x=0$ fixes all orbits. (The number of orbits was noted for each type in \eqref{eq: number-of-orbits} in the proof of Lemma~\ref{lem:reps-dim-mult-GL2}.) If $x\neq 0$, we saw above that $\tau=x$, so the trace of $\beta$ is determined. On the other hand, if $a_1,a_2\in\cO_{l'}$ are the roots of the characteristic polynomial of $\beta$, and $x=\tau=a_1+a_2$, then $\det\beta=a_1(x-a_1)$, so there are $q^{l'}/2$ choices for the determinant of $\beta$, and this is then also the number of orbits fixed by $x$. Hence, the Frobenius--Burnside formula gives
	\begin{align*}
		B_1=B_2 & =\frac{1}{|\cO_{l'}|}\left(\frac{1}{2}(q-1)q^{2l'-1}
		+\sum_{x\in\cO_{l'}^{\times}}\frac{q^{l'}}{2}\right) \\
		& =\frac{1}{2}(q-1)q^{l'-1}
		+(q-1)q^{l'-1}\cdot\frac{1}{2}=(q-1)q^{l'-1}.
	\end{align*}
	
	For type~3, assume that $x\in\cO_{l'}$ fixes $[\beta]$ of type $3$. Write $v$ for the valuation of $x$ and $w\ge 1$ for the valuation of $\tau$. The equation $2x=0$ is trivially satisfied, and $x(x+\tau)=0$
	is equivalent to $x+\tau\in\mfp^{l'-v}$.
	
	Suppose first that $v<l'-v$. Then $x+\tau\in\mfp^{l'-v}$ can only hold if we have $w=v$, which also entails $v\ge 1$. Write $x=u\pi^w$ and $\tau=u'\pi^w$ for some units $u,u'$ uniquely determined modulo $\mfp^{l'-w}$. Then
	$\pi^w(u+u')\in\mfp^{l'-w}$ is equivalent to $u\equiv u'\bmod\mfp^{l'-2w}$.
	This gives $q^{l'-w-(l'-2w)}=q^w$ possible choices
	for $u'$, and hence for $\tau$.
	
	Suppose then that $v\geq l'-v$, but $x\neq 0$. In particular, we have $v\ge\ceil{l'/2}\ge 1$. Then $x^2=0$ in $\cO_{l'}$, and the equation $x\tau=0$ holds
	if and only if $\tau\in\mfp^{l'-v}$. We find that we get $q^v$ possibilities
	for $\tau$ also in this case.
	
	The scalar $x=0$ fixes every orbit. If $x\neq 0$, we found above that $v\ge 1$ and the number of possibilities for the trace of $\beta$ is $q^v$, so the number of fixed orbits is $q^{v+l'}$, as the determinant can be chosen freely from $\cO_{l'}$. Finally, note that there are $(q-1)q^{l'-v-1}$ elements $x\neq 0$ with valuation $v$. With this information, the Frobenius--Burnside formula yields
	\begin{align*}
		B_3 & =\frac{1}{|\cO_{l'}|}\left(1\cdot q^{2l'-1}
		+\sum_{v=1}^{l'-1}(q-1)q^{l'-v-1}\cdot q^{v+l'}\right) \\
		& =\frac{1}{q^{l'}}\left(q^{2l'-1}+(q-1)q^{2l'-1}\sum_{v=1}^{l'-1}1\right)
		=q^{l'-1}(1+(l'-1)(q-1)).
	\end{align*}
	%

	(b) Assume that $l'>e$, so that, in particular, $\chara\cO_{l'}=2^m$ for some $m\ge 2$. Suppose that $x\in\cO_{l'}$ fixes $[\beta]$ of type~1 or 2. As $\chara\cO_{l'}\neq 2$, the equation $2x=0$ in~\eqref{eq:fixedpoints-equations} implies that $x$ is not a unit. As $\tau$ is a unit in types 1 and 2, it follows that $x+\tau$ is a unit, so the equation $x(x+\tau)=0$ yields $x=0$. Thus the stabiliser of $[\beta]$ is trivial, and every orbit has the same size. Hence, for $i\in \{1,2\}$, we get
	\[
	B_i=\frac{\#\{[\beta]\mid\text{$\beta$ of type }i\}}{|\cO_{l'}|}
	=\frac{1}{2}(q-1)q^{l'-1}.
	\]
	For type~3, we imitate the corresponding case in part (a). Assume that $x\in\cO_{l'}$ fixes $[\beta]$ of type $3$. Write $v$ for the valuation of $x$ and $w\ge 1$ for the valuation of~$\tau$. The equation $x(x+\tau)=0$ is equivalent to $x+\tau\in\mfp^{l'-v}$, as before, so the only additional restriction compared to the other case is that $2x=0$ is now equivalent to $v\geq l'-e$. Noting that $l'-e\ge 1$, the argument goes through exactly as in the other case, with the only difference that the final summation starts from $v=l'-e$ instead of $v=1$. This gives the result.
\end{proof}

For $i\in\{1,2\}$, let $\beta_i\in\M_{2}(\cO_{l'})$, and let $\rho_{i}\in\Irr(G_{r}\mid\beta_{i})$. If $\rho_{1}$ and $\rho_{2}$ lie in the same twist isoclass, then $[[\beta_{1}]]=[[\beta_{2}]]$. We can therefore speak of the twist orbit associated to a twist isoclass of representations of $G_{r}$.
For any $\beta\in\M_{2}(\cO_{l'})$ we write
\[
\twirr(G_{r}\mid\beta)
\]
for the set of twist isoclasses of irreducible representations of
$G_{r}$ whose twist orbit is $[[\beta]]$, and
\[
\Irr(G_{r}\mid[[\beta]])
\]
for the set of irreducible representations of $G_r$ whose orbit is an element of $[[\beta]]$.
The group $\chigroup:=\Irr(\cO_{r}^{\times})$ acts on $\Irr(G_{r}\mid[[\beta]])$
via
\[
(\chi,\rho) \longmapsto \rho\otimes(\chi\circ\det),
\]
where $\chi\in\chigroup$, and the orbits are exactly the elements of the set
$\twirr(G_{r}\mid\beta)$.

For each twist orbit $[[\beta]]$, we will compute or estimate the number of elements in $\twirr(G_r\mid\beta)$ by considering the stabilisers of representations under the above action.
It follows directly from the orbit-stabiliser theorem that 
\begin{equation}\label{eq:ost-inequality}
\min_\rho|\Stab_{\chigroup}(\rho)|
\le\frac{\#\twirr(G_r\mid\beta)\cdot|\chigroup|}{\#\irr(G_r\mid[[\beta]])}
\le\max_\rho|\Stab_{\chigroup}(\rho)|,
\end{equation}
where the minimum and maximum are taken over all $\rho$ in $\irr(G_r\mid[[\beta]])$.

In the following, we shall use this inequality to estimate the size of $\twirr(G_r\mid\beta)$ for different orbits $[[\beta]]$. We start by computing the size of the set on which the action takes place.

\begin{lem}\label{lem:no_of_irreps}Suppose that $\beta\in\M_{2}(\cO_{l'})$ is regular. Then
\[
\#\irr(G_r\mid[[\beta]])=q^{-2r}|\centbr K^{l'}|\cdot\#[[\beta]].
\]
\end{lem}

\begin{proof}
We have $\#\irr(G_{r}\mid[[\beta]])=\#\irr(\centbr K^{l'}\mid\beta)\cdot\#[[\beta]]$, so we need to determine the size of the set $\Irr(\centbr K^{l'}\mid\beta)$.
This is equivalent to finding the number of distinct irreducible constituents
of $\Ind_{K^{l}}^{\centbr K^{l'}}\psi_{\beta}$.

Let $\lambda\in\Irr(\centbr
K^{l'}\mid\beta)$.
Then $\langle\Ind_{K^{l}}^{\centbr
K^{l'}}\psi_{\beta},\lambda\rangle=\langle\psi_{\beta},\lambda|_{K^{l}}\rangle=\dim\lambda$,
where the last equality follows from Clifford's theorem and the fact
that $\centbr K^{l'}\subseteq\Stab_{G_{r}}(\psi_{\beta})$. By Lemma~\ref{lem:reps-dim-mult-GL2},
we know that every $\lambda\in\Irr(\centbr K^{l'}\mid\beta)$ has the same dimension. Thus
\[
[\centbr K^{l'}:K^{l}]=\dim\Ind_{K^{l}}^{\centbr
K^{l'}}\psi_{\beta}=\#\Irr(\centbr K^{l'}\mid\beta)\cdot(\dim\lambda)^{2}.
\]
By \eqref{eq:dim-eta} in the proof of Lemma~\ref{lem:reps-dim-mult-GL2}, we have  $\dim\lambda=q^{l-l'}$, so
\[
\#\Irr(\centbr K^{l'}\mid\beta)=|\centbr K^{l'}|\cdot q^{-4l'}q^{-2(l-l')}
=q^{-2r}|\centbr K^{l'}|. \qedhere
\]
\end{proof}

Next, we establish certain conditions that will later help us estimate the order of $\Stab_{\chigroup}(\rho)$.
For an element $g\in G_{r}$ we write $\bar{g}\in G_{l'}$ for
the image of $g$ in $G_{l'}$.

\begin{lem}
\label{lem: twist-criterion}Let $\rho=\Ind_{\centbr K^{l'}}^{G_r}\hat{\eta}_\theta$
for some $\hat{\eta}_\theta\in\Irr(\centbr K^{l'}\mid\beta)$,
as in Lemma~\ref{lem:Constr-repsGL2}. Then
$\rho(\chi\circ\det)=\rho$ holds
for some $\chi\in\Irr(\cO_{r}^{\times})$, if and only if there exists an $a\in\cO_{l'}$, such that the following two properties hold for some $g\in G_{r}$:
\begin{enumerate}
\item $aI+\beta=\bar{g}^{-1}\beta\bar{g}$,
\item $\hat{\eta}_{\theta}(\chi\circ\det|_{\centbr
K^{l'}})={}^{g}\hat{\eta}_{\theta}$.
\end{enumerate}
\end{lem}

\begin{proof}
Assume first that $aI+\beta=\bar{g}^{-1}\beta\bar{g}$, for some $a\in\cO_{l'}$ and $g\in G_{r}$. Then
$\centbl=\bar{g}\centbl\bar{g}^{-1}$,
so $g$ normalises the group $\centbr K^{l'}$.
Now, if $\hat{\eta}_{\theta}(\chi\circ\det|_{\centbr
K^{l'}})={}^{g}\hat{\eta}_{\theta}$,
then
\begin{align*}
\rho(\chi\circ\det) & =\Ind_{\centbr K^{l'}}^{G_r}(\hat{\eta}_{\theta}(\chi\circ\det|_{\centbr
K^{l'}}))=\Ind_{\centbr K^{l'}}^{G_r}{}^g\hat{\eta}_{\theta}\\
 & =\Ind_{^{g^{-1}}\centbr K^{l'}}^{G_r}\hat{\eta}_{\theta}=\Ind_{\centbr K^{l'}}^{G_r}\hat{\eta}_\theta=\rho.
\end{align*}
Assume conversely that $\rho(\chi\circ\det)=\rho$. We have
$\chi\circ\det|_{K^{l}}=\psi_{aI}$
for some $a\in\cO_{l'}$, so $\rho|_{K_l}$ contains
$\psi_\beta\psi_{aI}=\psi_{\beta+aI}$.
Thus $\psi_{\beta+aI}={}^{g}\psi_\beta$, for some $g\in G_r$,
and so $aI+\beta=\bar{g}^{-1}\beta\bar{g}$. Hence, as above, $g$
normalises $\centbr K^{l'}$, so we have

\[
\Ind_{\centbr K^{l'}}^{G_r}(\hat{\eta}_\theta(\chi\circ\det|_{\centbr K^{l'}}))=\rho(\chi\circ\det)
=\rho=\Ind_{\centbr K^{l'}}^{G_r}{}^g\hat{\eta}_\theta.
\]
We have $(\hat{\eta}_\theta(\chi\circ\det|_{\centbr K^{l'}}))|_{K_l}
=\psi_{aI+\beta}=\psi_{\bar{g}^{-1}\beta\bar{g}}={}^g\hat{\eta}_\theta|_{K_l}$,
and a basic result from Clifford theory says that induction provides
a bijection
\[
\Irr(\centbr K^{l'}\mid\beta)\longiso\Irr(G_r\mid\beta).
\]
Thus $\hat{\eta}_\theta(\chi\circ\det|_{\centbr K^{l'}})={}^g\hat{\eta}_\theta$.
\end{proof}

\begin{lem}\label{lem:stabiliser_order_p_two}
	For any $\rho\in\Irr(G_r\mid[[\beta]])$, we have
	\[
	|\Stab_{\chigroup}(\rho)|
	\le|\Stab_{\cO_{l'}}[\beta]|\cdot\frac{|\cO_r^\times|}{|\det(\centbr^{1}K^{l})|}.
	\]
\end{lem}

\begin{proof} 
	Suppose that $\rho(\chi\circ\det)=\rho$.
	Write $\rho=\Ind_{\centbr K^{l'}}^{G_r}\hat{\eta}_\theta$ for some $\hat{\eta}_\theta\in\Irr(\centbr K^{l'}\mid\beta)$, as in Lemma~\ref{lem:Constr-repsGL2}. By Lemma~\ref{lem: twist-criterion}, we have 
	$\bar{g}^{-1}\beta\bar{g}=aI+\beta$ and $\hat{\eta}_{\theta}(\chi\circ\det|_{\centbr K^{l'}})={}^{g}\hat{\eta}_{\theta}$ for some $a\in\cO_{l'}$ and $g\in G_r$.
	Let $Z$ denote the subalgebra of $\M_2(\cO_{l'})$ consisting of scalar matrices. The group $G_{l'}$ acts on the quotient $\M_2(\cO_{l'})/Z$ by conjugation. Then the condition $\bar{g}^{-1}\beta\bar{g}=aI+\beta$ is equivalent to $g\in\Stab_{G_{l'}}(\beta+Z)$. We will estimate the number of possible distinct characters ${}^{g}\hat{\eta}_{\theta}$.
	We have an isomorphism
	\begin{align*}
		\Stab_{G_{l'}}(\beta+Z)/\centbl & \longiso \Stab_{\cO_{l'}}[\beta]\\
		\bar{g}\centbl &\longmapsto  \bar{g}^{-1}\beta\bar{g}-\beta.
	\end{align*}
	(It is straightforward to verify that it is a well defined homomorphism, injective and surjective.) 
	If $\bar{g}\in\centbl$, then $g\in\centbr K^{l'}$, so ${}^{g}\hat{\eta}_{\theta}= \hat{\eta}_{\theta}$; thus there are at most
	$$|\Stab_{G_{l'}}(\beta+Z)/\centbl| = |\Stab_{\cO_{l'}}[\beta]|$$
	distinct characters of the form ${}^{g}\hat{\eta}_{\theta}$, where $g\in\Stab_{G_{l'}}(\beta+Z)$.
	
	Now, restricting both sides of the equality 
	$\hat{\eta}_{\theta}(\chi\circ\det|_{\centbr K^{l'}})={}^{g}\hat{\eta}_{\theta}$
	to $\centbr^{1}K^{l}$ we get
	\[
	m\theta(\chi\circ\det|_{\centbr^{1}K^{l}})=m'\theta'
	\]
	for some positive integers $m$ and $m'$, and some linear character $\theta'$ contained in ${}^{g}\hat{\eta}_{\theta}$. Since $\hat{\eta}_{\theta}$ and ${}^{g}\hat{\eta}_{\theta}$ have the same dimension, we have $m=m'$, and hence,
	\[
	\chi\circ\det|_{\centbr^{1}K^{l}}=\theta'\theta^{-1}.
	\]
	Since we have shown that there are at most $|\Stab_{\cO_{l'}}[\beta]|$ distinct characters ${}^{g}\hat{\eta}_{\theta}$, there are also at most $|\Stab_{\cO_{l'}}[\beta]|$ distinct characters $\theta'$, and thus at most $|\Stab_{\cO_{l'}}[\beta]|$ possibilities for 
	$\chi\circ\det|_{\centbr^{1}K^{l}}$. Each of the latter has exactly
	\[
	\frac{|G_r/\SL_2(\cO_r)|}{|\centbr^{1}K^{l}/\centbr^{1}K^{l}\cap\SL_2(\cO_r)|}
	=\frac{|\cO_r|}{|\det(\centbr^{1}K^{l})|}.
	\]
	extensions to $G_r$, whence the lemma.
\end{proof}

\begin{lem}
	\label{lem:C_i-order}
	Let $k$ be such that $1\leq k \leq r$. Then
	\[
	|C_k|=\begin{cases}
	(q-1)^{2}q^{2(k-1)} & \text{for }\beta\text{ of type }1,\\
	(q^{2}-1)q^{2(k-1)} & \text{for }\beta\text{ of type }2,\\
	q(q-1)q^{2(k-1)} & \text{for }\beta\text{ of type }3.
	\end{cases}
	\]
\end{lem}
\begin{proof}
	When $\beta$ is regular, all the reduction maps $C_i\rightarrow C_j$, for $i\geq j\geq 1$ are surjective, by a theorem of Hill. On the other hand, it is well known that every kernel $C_i^{i-1}$ has order $q^2$. Thus $|C_k|=|C_1|\cdot q^{2(k-1)}$, and the result follows from the well known orders of $C_1$.
\end{proof}

\subsection{The twist zeta function of $\GL_2(\cO)$ when $p\protect\neq2$}\label{sub:Odd-p-case}

Assume in this subsection that the characteristic $p$ of the residue field of $\cO$ is odd (apart from in Lemma~\ref{lem:scentbr-1-2-3podd}), and that $\beta\in\M_2(\cO_{l'})$ is fixed and regular, that is, of type 1, 2 or 3. We will explicitly compute the twist zeta function $\twistGLtwo(s)$. To this end, we start by computing the orders of stabilisers under the action of $\chigroup$ on representations of $G_r$.

In what follows, we will use $\mathbf{1}$ to denote the trivial character of a group.

\begin{lem}\label{lem:stabiliser_order_p_odd}
 For any $\rho\in\Irr(G_{r}\mid[[\beta]])$, we have
\[
|\Stab_{\chigroup}(\rho)|=\frac{|\cO_r^\times|}{|\det(\centbr K^{l'})|}.
\]
\end{lem}

\begin{proof}
We first show that $\chi\in\Stab_{\chigroup}(\rho)$ if and only if $\chi\circ\det|_{\centbr K^{l'}}=\mathbf{1}$. Write $\rho=\Ind_{\centbr K^{l'}}^{G_{r}}\hat{\eta}_{\theta}$ for a suitable $\hat{\eta}_{\theta}$, as in Lemma~\ref{lem:Constr-repsGL2}. If $\chi\circ\det|_{\centbr K^{l'}}=\mathbf{1}$, then
\[
\rho(\chi\circ\det)=\Ind_{\centbr K^{l'}}^{G_{r}}(\hat{\eta}_{\theta}(\chi\circ\det|_{\centbr K^{l'}}))=\rho.
\]

Conversely, suppose that $\rho(\chi\circ\det)=\rho$. The first condition in Lemma~\ref{lem: twist-criterion} tells us that $\bar{g}^{-1}\beta\bar{g}=aI+\beta$ for some $a\in\cO_{l'}$ and $g\in G_r$.
Thus $\tr(\beta)=2a+\tr(\beta)$, and since $p$ is odd, $a=0$ and $\bar{g}\in\centbl$. It follows that $g\in\centbr K^{l'}$, and therefore
$^{g}\hat{\eta}_{\theta}=\hat{\eta}_{\theta}$.
By the second condition in Lemma~\ref{lem: twist-criterion}, we
get $\hat{\eta}_{\theta}(\chi\circ\det|_{\centbr
K^{l'}})=\hat{\eta}_{\theta}$, and by restriction, we get
$$\eta_{\theta}(\chi\circ\det|_{\centbr^{1}K^{l'}})=\eta_{\theta}.$$
Since $\eta_{\theta}$ is the unique character of $\centbr^{1}K^{l'}$ lying
above $\theta$, we obtain $m\theta(\chi\circ\det|_{\centbr^{1}K^{l}})=m\theta$,
for some $m\in\N$, and since $\theta$ is one-dimensional, this implies
that
\[
\chi\circ\det|_{\centbr^{1}K^{l}}=\mathbf{1}.
\]
We now show that in the current situation this implies that
$\chi\circ\det|_{\centbr^{1}K^{l'}}=\mathbf{1}$
by showing that $\det(\centbr^{1}K^{l})\supseteq\det(\centbr^{1}K^{l'})$. 
To prove the latter inclusion, we first prove that 
\begin{equation}
\det(C^{l'}K^{l})=1+\mfp^{l'}.\label{eq:detCK}
\end{equation}
For any $x\in\M_{2}(\cO_{r})$ we see by direct calculation that $\det(1+\pi^{l'}x)\in(1+\pi^{l'}\tr(x))(1+\mfp^{l})$.
Thus, letting $I\in\M_{2}(\cO_{r})$ denote the identity matrix, we
have
\begin{align*}
\det(C^{l'}K^{l})=\det(C^{l'})(1+\mfp^{l}) & \supseteq\{\det(1+\pi^{l'}\lambda I)\mid\lambda\in\cO_{r}\}(1+\mfp^{l})\\
& \supseteq\{1+\pi^{l'}2\lambda\mid\lambda\in\cO_{r}\}(1+\mfp^{l})\\
& \supseteq(1+\mfp^{l'})(1+\mfp^{l})=1+\mfp^{l'}.
\end{align*}
where for the last inclusion we have used that $2\in\cO^{\times}$
(since $p\neq2$). Since $\det(C^{l'}K^{l})\subseteq1+\mfp^{l'}$,
this proves (\ref{eq:detCK}).
We can now conclude that 
$$\det(\centbr^{1}K^{l})\supseteq\det(\centbr^{1}\centbr^{l'}K^{l})\supseteq \det(\centbr^{1}K^{l'}),$$
and so
\[
\chi\circ\det|_{\centbr^{1}K^{l'}}=\chi\circ\det|_{\centbr^{1}K^{l}}=\mathbf{1}.
\]
We have shown that
$\Ind_{CK^{l'}}^{G_{r}}\hat{\eta}_{\theta}(\chi\circ\det|_{\centbr
K^{l'}})=\Ind_{CK^{l'}}^{G_{r}}\hat{\eta}_{\theta}$,
with $\chi\circ\det|_{\centbr K^{l'}}\in\Irr(\centbr K^{l'}/\centbr^{1}K^{l'})$, so the uniqueness part of Lemma~\ref{lem:Constr-repsGL2} implies that
\[
\chi\circ\det|_{\centbr K^{l'}}=\mathbf{1}.
\]

Now, identify the set of one-dimensional characters of $G_{r}$ with $\Irr(G_{r}/\SL_{2}(\cO_{r}))$. Similarly,
the set of characters of the form $\chi\circ\det|_{\centbr K^{l'}}$
can be identified with $\Irr(\centbr K^{l'}/(\centbr
K^{l'}\cap\SL_{2}(\cO_{r})))$.
The number of characters $\chi\in\Irr(\cO_{r}^{\times})$ such that
$\chi\circ\det|_{\centbr K^{l'}}=\mathbf{1}$ is therefore
\begin{equation}\label{eq:First-twist-count}
\frac{|G_{r}/\SL_{2}(\cO_{r})|}{|\centbr K^{l'}/(\centbr
K^{l'}\cap\SL_{2}(\cO_{r}))|}=\frac{|\cO_r^\times|}{|\det(\centbr K^{l'})|}. \qedhere
\end{equation}
\end{proof}
From now on, we will write $\scentbl$ for the kernel of the determinant map $\det\colon\centbl\rightarrow \cO_{l'}^{\times}$, that is, $\scentbl=\centbl\cap \SL_2(\cO_{l'})$.
\begin{lem}\label{lem:CK^lcapS_r-onto-SC_l}
 We have
\[
|\det(\centbr K^{l'})|=q^{-3l}\frac{|\centbr K^{l'}|}{|\scentbl|}.
\]
\end{lem}

\begin{proof}
We have
\[
|\det(\centbr K^{l'})|=\frac{|\centbr K^{l'}|}{|\centbr K^{l'}\cap\SL_{2}(\cO_{r}))|}.
\]
In order to rewrite this expression, we show that reduction modulo
$\mfp^{l'}$, denoted~$\rho_{l'}$, induces an isomorphism
\[
\frac{\centbr
K^{l'}\cap\SL_{2}(\cO_{r})}{K^{l'}\cap\SL_{2}(\cO_{r})}\cong\scentbl.
\]
This follows if we can show that the reduction map induces a surjection
$\centbr K^{l'}\cap\SL_{2}(\cO_{r})\rightarrow\scentbl$. Let $t\in\scentbl$.
Since $\rho_{l'}\colon\centbr\rightarrow\centbl$ is surjective (because $\beta$ is assumed to be regular), there
exists a lift $\hat{t}\in\centbr$ of $t$. Since $\det(t)=1$, we
have $\rho_{l'}(\det(\hat{t}))=\det(t)=1$, so $\det(\hat{t})\in1+\mfp^{l'}$.
Thus $K^{l'}$ contains an element $k$ such that $\det(k)=\det(\hat{t})^{-1}$,
and so $\hat{t}k\in\centbr K^{l'}\cap\SL_{2}(\cO_{r})$ is an element
that maps to $t$.

Now, we have
\[
\frac{|\centbr K^{l'}|}{|\centbr K^{l'}\cap\SL_{2}(\cO_{r})|}
=\frac{|\centbr K^{l'}|}{|\scentbl|\cdot |K^{l'}\cap\SL_{2}(\cO_{r})|}.
\]
Finally, since the Lie algebra $\mathfrak{sl}_2(\F_q)$ has dimension three, $$|K^{l'}\cap\SL_{2}(\cO_{r})|=q^{3(r-l')}=q^{3l},$$
whence the claim follows.
\end{proof}

We now combine the results from above to obtain the number of twist isoclasses above a given $\beta$.

\begin{lem}\label{lem:twist_class_count_p_odd}
 We have
\[
\#\twirr(G_r\mid\beta)=q^{l-l'}|\scentbl|.
\]
\end{lem}

\begin{proof}
From the two previous lemmas, we see that
\[
|\Stab_{\chigroup}(\rho)|=q^{3l}\frac{|\cO_r^\times||\scentbl|}{|\centbr K^{l'}|}
\]
for all $\rho\in\irr(G_r\mid[[\beta]])$. We combine this with Lemma~\ref{lem:no_of_irreps} and substitute into~\eqref{eq:ost-inequality}. Note that $\Stab_{\chigroup}(\rho)$ has the same order for all $\rho\in\irr(G_r\mid[[\beta]])$, so~\eqref{eq:ost-inequality} becomes an equality, yielding
\begin{align*}
\#\twirr(G_r\mid\beta)
& =\frac{q^{-2r}|\centbr K^{l'}|\#[[\beta]]\cdot q^{3l}|\cO_r^\times||\scentbl|/|\centbr K^{l'}|}
{|\chigroup|} \\
& =q^{l-2l'}|\scentbl|\#[[\beta]].
\end{align*}
Since $p$ is odd, the first equation in~\eqref{eq:fixedpoints-equations} implies that $\#[[\beta]]=|\cO_{l'}|=q^{l'}$, whence the claim follows.
\end{proof}

Before the final result, we still need the following information.

\begin{lem}\label{lem:scentbr-1-2-3podd}
For $\cO$ of any residue characteristic, we have
\[
|\scentbl|=
\begin{cases}
(q-1)q^{l'-1} & \text{if $\beta$ is of type $1$}, \\
(q+1)q^{l'-1} & \text{if $\beta$ is of type $2$}.
\end{cases}
\]
For $\cO$ of odd residue characteristic and $\beta$ of type 3, we have
\[
|\scentbl|=2q^{l'}.
\]
\end{lem}
\begin{proof}
We first prove the assertion for $\beta$ of types $1$ and $2$,
in any characteristic. For type~1, $\scentbl$ is the diagonal
subgroup of $G_{l'}$, so the determinant is surjective. For type~$2$,
$\scentbl$ is equal to $\cO[\hat{\beta}]^{\times}/(1+\mfp^{l'})$,
where $\hat{\beta}\in\M_{2}(\cO)$ is a lift of $\beta_{l'}$. Since
$\beta$ has irreducible characteristic polynomial modulo $\mfp$, the
ring $\cO[\beta]$ is unramified over $\cO$. Now, the determinant
map coincides with the norm map $\cO[\beta]^{\times}\rightarrow\cO^{\times}$,
and it is well-known that the latter is surjective for unramified
extensions. Thus, for $\beta$ of type~1, we have
\[
|\scentbl|=\frac{|\centbl|}{|\cO_{l'}^{\times}|}=\frac{(q-1)^2 q^{2(l'-1)}}{
(q-1)q^{l'-1}}=(q-1)q^{l'-1},
\]
and for $\beta$ of type~2, we have
\[
|\scentbl|=\frac{|\centbl|}{|\cO_{l'}^{\times}|}=\frac{(q^2-1)q^{2(l'-1)}}{
(q-1)q^{l'-1}}=(q+1)q^{l'-1}.
\]

Assume now that $\beta$ is of type~3 and that $p\neq 2$. By adding
a suitable scalar to $\beta$ we may assume that $\Delta$ and $\tau$ are in $\mfp$.
For $\bigl[\begin{smallmatrix}x & y\\
-\Delta y & x+\tau y
\end{smallmatrix}\bigr]\in\centb{l'}$, we have
\[
\det\begin{bmatrix}x & y\\
-\Delta y & x+\tau y
\end{bmatrix}=x^{2}+\tau xy+\Delta y^{2}.
\]
Let $f(x,y)=x^{2}+\tau xy+\Delta y^{2}-1$, so that the kernel of the
map $\det\colon\centbl\rightarrow\cO_{l'}^{\times}$ is $\{(x,y)\in(\cO_{l'})^{2}\mid
f(x,y)=0\}$.
The gradient $\nabla f$ then satisfies $\nabla f\equiv\bigl[\begin{smallmatrix}2x\\
0\end{smallmatrix}\bigr]$ mod $\mfp$, and since $x$ is a unit for any $(x,y)\in(\cO_{l'})^{2}$
such that $f(x,y)=0$, we have $\nabla f\not\equiv 0$ mod $\mfp$. By Hensel's lemma,
we obtain the existence of $q^{l'-1}$ lifts to $(\cO_{l'})^{2}$ of
each of the $2q$ solutions modulo $\mfp$. Thus the total number of
solutions in $(\cO_{l'})^{2}$ is $2q^{l'}$ and so $|\scentbl|=2q^{l'}$.
\end{proof}

Finally, we are in a position to compute the twist zeta function in odd characteristic.

\begin{thm}\label{thm:twist_zeta_function_p_odd}
Assume that the residue characteristic of $\cO$ is odd. Then the
twist zeta function of $\GL_{2}(\cO)$ is
\begin{align*}
\twistGLtwo(s) &
=1+\frac{1}{q^s}+\frac{q-1}{2(q+1)^s}+\frac{q+1}{2(q-1)^s} \\[0.5ex]
 &
\phantom{{={}}}+\left(\frac{(q-1)^{2}}{2(q^{2}+q)^{s}}+\frac{q^{2}-1}{2(q^{2}
-q)^{s}}+\frac{2q}{(q^{2}-1)^{s}}\right)\left(\frac{1}{1-q^{1-s}}\right).
\end{align*}
In particular, the abscissa of convergence of $\tilde{\zeta}_{\GL_{2}(\cO)}$ is $1$.
\end{thm}

\begin{proof}
Since we have
\[
\twistGLtwo(s)=\lim_{R\to\infty}\tilde{\zeta}_{\GL_{2}(\cO_R)}(s),
\]
we will compute $\tilde{\zeta}_{\GL_{2}(\cO_R)}(s)$ and take the limit. Assume  that $r\ge 2$, and recall that by Lemma~\ref{lem:reps-dim-mult-GL2} the dimensions of irreducible primitive representations of type~$i$, and hence of twist isoclasses, are given by
\[
d_{r}(i)=
\begin{cases}
(q+1)q^{r-1} & \text{if $i=1$}\\
(q-1)q^{r-1} & \text{if $i=2$}\\
(q^{2}-1)q^{r-2} & \text{if $i=3$}.
\end{cases}
\]
For the multiplicities, we have
\[
\tilde{r}_{d_{r}(i)}(G_r)=\sum_{\beta}\#\widetilde{\Irr}(G_{r}\mid\beta),
\]
where $\beta$ runs through a set of representatives of the twist orbits of type $i$.
Lemmas~\ref{lem:twist_class_count_p_odd} and \ref{lem:scentbr-1-2-3podd} imply that
\[
\#\widetilde{\Irr}(G_{r}\mid\beta)=
\begin{cases}
(q-1)q^{l-1} & \text{if $\beta$ is of type 1} \\
(q+1)q^{l-1} & \text{if $\beta$ is of type 2} \\
2q^l & \text{if $\beta$ is of type 3},
\end{cases}
\]
so in particular, the number $\#\widetilde{\Irr}(G_{r}\mid\beta)$ only depends on the type of $\beta$. Hence, recalling that $B_i$ denotes the number of twist orbits of type $i$, we can write
\[
\tilde{r}_{d_{r}(i)}(G_r)=\#\widetilde{\Irr}(G_{r}\mid\beta)\cdot B_i
\]
for all $i\in\{1,2,3\}$.
Combining the numbers $B_i$ from \eqref{eq:numb-of-twist-orbits} with the cardinalities of $\widetilde{\Irr}(G_{r}\mid\beta)$ obtained above gives
\[
\tilde{r}_{d_{r}(i)}(G_r)=
\begin{cases}
\frac{1}{2}(q-1)^{2}q^{r-2} & \text{if $i=1$}\\
\frac{1}{2}(q^2-1)q^{r-2} & \text{if $i=2$}\\
2q^{r-1} & \text{if $i=3$}.
\end{cases}
\]
Using the above values together with Lemma~\ref{lem:twist-zetaGL2(Fq)}, we
compute the twist zeta function of $\GL_{2}(\cO_{R})$, $R\geq 1$,
to be
\begin{align*}
\tilde{\zeta}_{\GL_{2}(\cO_{R})}(s)
& =\tilde{\zeta}_{\GL_{2}(\F_{q})}(s)
+\sum_{i=1,2,3}\,\sum_{r=2}^R \frac{\tilde{r}_{d_r(i)}(G_r)}{d_r(i)^s}
\\
& =\tilde{\zeta}_{\GL_{2}(\F_{q})}(s)
+\frac{(q-1)^{2}}{2(q+1)^{s}}q^{s-2}\sum_{r=2}^{R}q^{(1-s)r}
+\frac{q^{2}-1}{2(q-1)^{s}}q^{s-2}\sum_{r=2}^{R}q^{(1-s)r} \\
& \qquad +\frac{2}{(q^{2}-1)^{s}}q^{2s-1}\sum_{r=2}^{R}q^{(1-s)r} \\
& =1+\frac{1}{q^s}+\frac{q-1}{2(q+1)^s}+\frac{q+1}{2(q-1)^s} \\
& \qquad +\left(\frac{(q-1)^{2}}{2(q^{2}+q)^{s}}+\frac{q^{2}-1}{2(q^{2}
-q)^{s}}+\frac{2q}{(q^{2}-1)^{s}}\right)\left(\frac{1-q^{(1-s)(R-1)}}{1-q^{1-s}}\right).
\end{align*}
It is evident from the final factor that the twist zeta function has a pole at 
$s=1$ for all $R$. When $s$ has real part greater than 1, we 
obtain the twist zeta function for $\GL_2(\cO)$ by letting $R\to\infty$.
\end{proof}
\begin{rem}
 The formula in the preceding theorem implies that 
$\twistGLtwo(s)$ has a zero at $s=-1$, when $\cO$ has odd residue 
characteristic.
\end{rem}

\subsection{Estimating the twist zeta function when $p=2$ and $\chara\cO=0$}

In this subsection we assume that the residue characteristic $p$ is two and the characteristic of $\cO$ is zero. Compared to the previous subsection, we shall content ourselves with only estimating the twist zeta function.

Recall that $e$ denotes the ramification index of $\cO$. Fix a prime element $\pi\in\cO$ such that $2=\pi^e$.

\begin{lem}\label{lem:O-0_p-2_Cdet(CK)}
For any $\beta\in\M_2(\cO_{l'})$ and any lift $\hat{\beta}\in\M_2(\cO_{r})$, we have
\[
\frac{|\cO_r^\times|}{|\det(\centbr^{1}K^{l})|}<3(q-1)q^{2e}.
\]
\end{lem}

\begin{proof}
We estimate the order of $\det(\centbr^{1})$. Note that $\centbr$ contains the scalar matrices of $G_r$, so $\det(\centbr^{1})$ contains every element of $\cO_r^\times$ of the form
\[
(1+x\pi)^2=1+2x\pi+x^2\pi^2 \quad \text{with $x\in\cO_r$}.
\]
We proceed by bounding the number of squares. The kernel of the map $f\colon 1+\mfp\rightarrow 1+\mfp$, $1+x\pi\mapsto (1+x\pi)^2$ is defined by the equation $2x\pi+x^2\pi^2=0$. Recalling that $2=\pi^e$, this leads to
\[
x\pi^2(x+\pi^{e-1})=0.
\]

Let $v$ denote the valuation of $x$ and assume that $x\neq 0$. We consider different cases. First, if $v<e-1$, the valuation of $x+\pi^{e-1}$ is $v$, and the equation above implies that $2+2v\ge r$. This is possible only if $r<2e$. When this condition is satisfied, the number of solutions to the equation is the number of $x$ with $\ceil{r/2}-1\le v<e-1$. This number is
\[
q^{r-e-1}(q^{e-\ceil{r/2}}-1)<q^{\floor{r/2}}<q^e.
\]

Assume then that $v>e-1$. Then the valuation of $x+\pi^{e-1}$ is $e-1$, and we get $v+2+e-1\ge r$. The number of $x$ with $v\ge r-e-1$ is $q^{e+1}-1$.

Finally, assume that $v=e-1$. Write $x=u\pi^{e-1}$ for some unit $u$, so that the above equation becomes $u\pi^{2e}(u+1)=0$. As $u$ is a unit, this equation implies the valuation of $u+1$ is at least  $r-2e$. There are therefore $q^{2e}-1$ possible solutions.

Adding together the possible numbers of solutions, together with $x=0$, gives an upper bound for the kernel order:
\[
|\Ker f|=|\{(1+x\pi)^2=1 \mid x\in\cO_r\}|<3q^{2e}.
\]
Now, we have
\[
\frac{|\cO_r^\times|}{|\det(\centbr^{1}K^{l})|}
\le\frac{|\cO_r^\times|}{|\im f|}
=\frac{|\cO_r^\times||\Ker f|}{|1+\mfp|}
<3(q-1)q^{2e}. \qedhere
\]
\end{proof}

\begin{lem}\label{lem:twist_class_count_char_zero}
For any regular $\beta\in\M_2(\cO_{l'})$ we have
\[
\#\twirr(G_r\mid\beta)<3q^{l+2e+1}.
\]
\end{lem}

\begin{proof}
Considering the second inequality in~\eqref{eq:ost-inequality}, and using Lemma~\ref{lem:no_of_irreps} and Lemma~\ref{lem:stabiliser_order_p_two} together with Lemma~\ref{lem:O-0_p-2_Cdet(CK)}, we find that
\[
\#\twirr(G_r\mid\beta)\le\frac{q^{-2r}|\centbr K^{l'}|\cdot\#[[\beta]]
\cdot|\Stab_{\cO_{l'}}[\beta]|\cdot 3(q-1)q^{2e}}{|\chigroup|}.
\]
By the orbit-stabiliser theorem, $|\cO_{l'}|=\#[[\beta]]\cdot |\Stab_{\cO_{l'}}[\beta]|$, so this simplifies to
\[
3q^{-3r+2e+1}|\centbr K^{l'}|\cdot|\cO_{l'}|
=3q^{-2r-l+2e+1}|\centbr K^{l'}|.
\]
For the order of $\centbr K^{l'}$, we consider the quotient $\centbr K^{l'}/K^{l'}$, which is isomorphic to $\centbl$ (since we have assumed that $\beta$ is regular). By Lemma~\ref{lem:C_i-order}, the order of $\centbl$ can be bounded from above by $q^{2l'}$, for $\beta$ of any type, so we get
\[
|\centbr K^{l'}|=|K^{l'}|\cdot |\centbl|
<q^{4l}\cdot q^{2l'}=q^{2r+2l}.
\]
The result follows.
%
\end{proof}

We can now proceed to the main result of this subsection.

\begin{thm}\label{thm:twist_zeta_function_char_zero}
Assume that the ring $\cO$ has characteristic zero and residue characteristic two. For  $i\in\{1,2,3\}$, we have 
\[
\tilde{r}_{d_r(i)}(G_r) \ll q^{r}.
\]
Moreover, the abscissa of convergence of the twist zeta function $\twistGLtwo$ is $1$.
\end{thm}

\begin{proof}
The dimension of a twist isoclass only depends on the type of the corresponding twist orbit, and the dimensions were given as $d_r(i)$ in Lemma~\ref{lem:reps-dim-mult-GL2}. Using Lemma~\ref{lem:twist_class_count_char_zero}, we can estimate the number of twist isoclasses corresponding to a given dimension as follows:
\[
\tilde{r}_{d_{r}(i)}(G_r)
=\sum_{\beta\text{ of type $i$}}\#\widetilde{\Irr}(G_{r}\mid\beta)
\ll q^{l}B_i,
\]
where $\beta$ runs through a set of representatives of the twist orbits of type $i$.
By Lemma~\ref{lem:orbit_numbers}, we have $B_i\ll q^{l'}$, which leads to the bound in theorem.

It follows from the obtained upper bounds that for any real positive~$s$ there is a positive constant $A\in\R$ such that for any $R\geq 2$, we have 
\[
\twistGLR(s)
=\twistGLFq(s)
+\sum_{i=1,2,3}\,\sum_{r=2}^R \frac{\tilde{r}_{d_r(i)}(G_r)}{d_r(i)^s}
\leq A\sum_{r=2}^R q^{(1-s)r}.
\]
Clearly, if $s>1$, the sum on the right hand side converges when $R\to\infty$, so the abscissa of convergence of $\twistGLtwo=\lim_{R\to\infty}\twistGLR $ is at most $1$.
On the other hand, Corollary~\ref{cor:abs-lower_bound_p_good} implies that the abscissa of $\twistGLtwo$ is at least $1$, so it is exactly~$1$.
%
\end{proof}

\subsection{Estimating the twist zeta function when $p=2$ and $\chara\cO=2$}
\label{sec:p-2-charO-2}

In this subsection we assume that the characteristic of $\cO$ is two. It follows that $\cO$ is isomorphic to a ring of power series over the residue field, and $\cO_i=\F_{q}[[t]]/(t^i)$, where $t$ is an indeterminate.

In order to estimate the multiplicity of the representation dimensions in type 3, we need to subdivide the twist isoclasses of this type further, depending on two invariants which we now define. For later purposes, the invariants are defined with respect to a matrix $\alpha$ of type 3 in $\M_2(\cO_i)$ for any $i\le r$, although they will mostly be used with $\beta\in\M_2(\cO_{l'})$.

\begin{defn}\label{def:Odd-depth}
Let $\alpha\in\M_2(\cO_i)$ be a matrix of type 3 and write $\tau=\tr(\alpha)$. Write also $\Delta=\det(\alpha)=\Delta_0+\Delta_1 t+\dots+\Delta_{i-1}t^{i-1}$, with $\Delta_k\in \F_q$ for $k<i$.
\begin{enumerate}[label=(\alph*)]
\item If $\tau\neq 0$, let $w(\alpha)$ denote the valuation of $\tau$. If $\tau=0$, we define $w(\alpha)=i$.
\end{enumerate}
Write $w(\alpha)=2M+\varepsilon$, where $M=\floor{w(\alpha)/2}$ and $\varepsilon\in\{0,1\}$.
\begin{enumerate}[label=(\alph*),resume]
\item Let $\delta(\alpha)$ denote the smallest $0\leq k<M$, for which $\Delta_{2k+1}\neq 0$. If such $k$ does not exist, we define $\delta(\alpha)=M$. We call $\delta(\alpha)$ the \emph{odd depth} of $\Delta$ (and also of $\alpha$).
\end{enumerate}
\end{defn}

Note that $\delta(\alpha)\in\{0,\dots,M\}$, and if $\delta(\alpha)<M$, then $\Delta_{2\delta(\alpha)+1}\neq 0$. Note also that if $i=1$ in the above definition, then $\tau=0$ (because $\alpha$ is of type 3), and thus $\delta(\alpha)=0$.
Let us show that the parameters $w(\alpha)$ and $\delta(\alpha)$ are invariants of the twist orbit $[[\alpha]]$.
Indeed, as the characteristic of $\cO_i$ is 2, adding any scalar to $\alpha$
does not change $\tr(\alpha)$ at all. On the other hand, let $x\in\cO_i$ and consider $\Delta'=\det(x+\alpha)$. We have
\[
\Delta'=x^2+x\tau+\Delta.
\]
As the valuation of $\tau$ is at least $2M$, we see that $\Delta'_{2k+1}=\Delta_{2k+1}$ for $k\le M-1$. Thus also $\delta(\alpha)$ is an invariant of the twist orbit.

For the most part, we will consider $w(\beta)$ and $\delta(\beta)$, with a fixed $\beta\in\M_2(\cO_{l'})$ parametrising a one-dimensional character $\psi_{\beta}$ of $K^l$, as before. In these cases, for notational simplicity, we write
\[
w=w(\beta)\qquad\text{and}\qquad\delta=\delta(\beta).
\]

We will now count the number of twist orbits in $\M_2(\cO_{l'})$ of type 3 with fixed parameters $w$ and $\delta$. Given $w,\delta\in\cO_{l'}$, we write $B(w,\delta)$ for the number of twist orbits of conjugacy classes $[\alpha]$, $\alpha\in\M_2(\cO_{l'})$ of type 3, such that $w(\alpha)=w$ and $\delta(\alpha)=\delta$.

\begin{lem}\label{lem:orbit_numbers_with_w_and_delta}
Assume that $\cO$ has characteristic $2$. Let
\[
D(\delta)=
\begin{cases}
(q-1)q^{l'-\delta-1} & \text{if $\delta<M$} \\
q^{l'-M} & \text{if $\delta=M$}.
\end{cases}
\]
Then
\[
B(w,\delta)=
\begin{cases}
2(q-1)q^{-1}D(\delta) & \text{when $1\le w<l'/2$} \\
(q-1)q^{\floor{l'/2}-w-1}D(\delta) & \text{when $l'/2\le w<l'$} \\
q^{-\floor{l'/2}}D(\delta) & \text{when $w=l'$}.
\end{cases}
\]
\end{lem}

\begin{proof}
Fix the parameters $w$ and $\delta$. Note that $w\ge 1$, as $\tau$ is congruent to 0 modulo $\mfp$ in type 3.
We use the Frobenius--Burnside formula as in the proof of Lemma~\ref{lem:orbit_numbers}.
Assume that $x\in\cO_{l'}$ fixes a conjugacy class $[\alpha]$ with parameters $w$ and $\delta$, and write $v$ for the valuation of $x$.
As before, we know from the second equation in~\eqref{eq:fixedpoints-equations} that $x(x+\tau)=0$.
Note that $D(\delta)$ is equal to the number of determinants with odd depth $\delta$.

\emph{Case $1\le w<l'/2$.} Assume first that $v\ge l'/2$. Then the valuation of
$x+\tau$ is $w$, so the equation $x(x+\tau)=0$ holds if and only if $v\ge
l'-w$. The number of fixed points for any such $x$ is then simply the number of
conjugacy classes, which is the number of traces with valuation $w$ times the
number of determinants with odd depth $\delta$, and equals $(q-1)q^{l'-w-1}D(\delta)$.
On the other hand, the number of $x\in\cO_{l'}$ with valuation at least $l'-w$ is $q^w$.

Assume then that $v<l'/2$. Since both $x$ and $\tau$ have valuation less than
$l'/2$, the equation $x(x+\tau)=0$ can hold only if $v=w$. As in the proof of Lemma~\ref{lem:orbit_numbers} for type 3,
the trace can only be one of $q^w$ many possibilities. This gives $q^w D(\delta)$ many fixed points. Finally, noting that the number of elements $x$ with valuation $w$ is $(q-1)q^{l'-w-1}$, the Frobenius--Burnside formula gives
\begin{align*}
B(w,\delta) & =\frac{1}{q^{l'}}
\bigl(q^w\cdot(q-1)q^{l'-w-1}D(\delta)
+(q-1)q^{l'-w-1}\cdot q^w D(\delta)\bigr) \\
& =2(q-1)q^{-1}D(\delta).
\end{align*}

\emph{Case $l'/2 \le w < l'$.} For the equation $x(x+\tau)=0$ to
hold, we need $v\ge l'/2$ (we include the case $x=0$ as $v=l'$). In that case,
every conjugacy class with the given parameters is fixed by $x$. The number of $x$ with valuation at least $l'/2$ is $q^{\floor{l'/2}}$, so we get
\[
B(w,\delta)=\frac{1}{q^{l'}}\cdot q^{\floor{l'/2}}\cdot(q-1)q^{l'-w-1}D(\delta)
=(q-1)q^{\floor{l'/2}-w-1}D(\delta).
\]

\emph{Case $w=l'$.} Here we have $\tau=0$. Now, $x(x+\tau)=x^2=0$ holds if and
only if $v\ge l'/2$ (including $x=0$ as $v=l'$). For these $x$, all conjugacy
classes with the given parameters are fixed points, and their number is $D(\delta)$, as the value of
$\tau$ is already determined. Hence, the Frobenius--Burnside formula yields
\[
B(w,\delta)=\frac{1}{q^{l'}}\cdot q^{\floor{l'/2}}\cdot D(\delta)
=q^{-\ceil{l'/2}}D(\delta). \qedhere
\]
\end{proof}

Recall that in our notation, we have $\scentbi=C_{G_i}(\beta_i)\cap\SL_2(\cO_i)$, where $\beta_i=\rho_i(\hat{\beta})$, and $\hat{\beta}\in\M_2(\cO_r)$ is some lift of $\beta\in\M_2(\cO_{l'})$.
The proof of the following key lemma is rather long, is independent of the rest of this section, and can be found in Section~\ref{sec:proof_of_lemma}.

\begin{lemma}\label{lem:kernel-type-3}
Let $1\le i\le r$, and assume $\beta\in\M_2(\cO_{l'})$ is of type 3. Then
\[
|\scentbi| = cq^{i+\delta(\beta_i)}
\]
for some $c\in \{1,2,3\}$.
\end{lemma}

The following lemma can be used to estimate $|\scentb{l}|$ with respect to $|\scentbl|$ when $l\neq l'$. The proof works for any $(i,i-1)$, $i\geq 2$, instead of $(l,l')$, but for simplicity we only state it in the case where we will apply it.

\begin{lemma}\label{lem:kernel-type-3-higher-level}
Suppose that $r$ is odd so that $l'=l-1$ and let $\beta\in\M_2(\cO_{l'})$ be of type $3$. Then, for any lift $\beta_l\in\M_2(\cO_l)$ of $\beta$, we have $\frac{1}{3}q\leq |\scentb{l}|/|\scentbl|\leq 3q^2$.
\end{lemma}

\begin{proof}
By Lemma~\ref{lem:kernel-type-3}, we have $|\scentbl|=c_1 q^{l'+\delta(\beta)}$ and $|\scentb{l}|=c_2q^{l+\delta(\beta_l)}$, with $c_1,c_2\in\{1,2,3\}$, so
\[
\frac{1}{3}q^{1+\delta(\beta_l)-\delta(\beta)}\leq \frac{|\scentb{l}|}{|\scentbl|}\leq 3q^{1+\delta(\beta_l)-\delta(\beta)}.
\]
It remains to show that $\delta(\beta_l)-\delta(\beta)\in\{0,1\}$. Let $w$ be the valuation of $\tr(\beta)$ and $w_l$ the valuation of $\tr(\beta_l)$. (We define the valuation of $0\in\cO_i$ to be $i$.)

Assume first that $w_l=w$. 
Then $w_l\le l'$, and according to Definition~\ref{def:Odd-depth}, the cutoff parameter $M=\floor{w_l/2}$ is the same for $\beta_l$ as for $\beta$. As $M\leq\floor{l'/2}\leq l'$, the coefficients of the $\mfp$-adic expansion of $\det(\beta_l)$ are equal to those of $\det(\beta)$ up to the cutoff, so $\delta(\beta_l)=\delta(\beta)$.

Assume on the other hand that $w_l\neq w$. Then we must have $w_l=l$ and $w=l'$, so $\floor{w_l/2}=\floor{w/2}$ unless $l'$ is odd, in which case $\floor{w_l/2}=\floor{w/2}+1$. If $\floor{w_l/2}=\floor{w/2}$, we have $\delta(\beta_l)=\delta(\beta)$, as before. On the other hand, if $\delta(\beta_l)<\floor{w_l/2}$, there is a non-zero coefficient with odd index at most $w_l-1=w$, so $\delta(\beta_l)=\delta(\beta)$. However, if $l$ is odd and $\delta(\beta_l)=\floor{w_l/2}$, then $\delta(\beta)=\floor{w/2}=\delta(\beta_l)-1$. There are no other possibilities, whence the result.
%
\end{proof}


Recall that we write $\delta=\delta(\beta)$ when considering a particular $\beta\in\M_2(\cO_{l'})$.

\begin{lem}\label{lem:double_ratio_char_two}
We have
\[
\frac{|\cO_r^\times|}{|\det(\centbr^{1}K^{l})|}<
\begin{cases}
q-1 & \text{for $\beta$ of type $1$ or $2$} \\
3(q-1)q^{\delta+3} & \text{for $\beta$ of type $3$}.
\end{cases}
\]
\end{lem}

\begin{proof}
If $\beta$ is of type $1$, the group $\centbr^1$ is conjugate
to $\bigl[\begin{smallmatrix}1+\mfp & 0\\0 & 1+\mfp\end{smallmatrix}\bigr]$,
so the image of the determinant map from $\centbr^1$ is $1+\mfp$.
Assume next that $\beta$ is of type $2$. Then $\centbr$ is conjugate
to $\tilde{\cO}_{r}^{\times}$, where $\tilde{\cO}$ is the ring of
integers in the unramified extension of degree two of the field of
fractions of $\cO$. Thus $\centbr^1$ is conjugate
to $1+\mfp\tilde{\cO}_{r}$ and the determinant on $\centbr^1$
corresponds to the norm map on $1+\mfp\tilde{\cO}_{r}$. Since
$\tilde{\cO}$ is unramified over $\cO$, the image of
$1+\mfp\tilde{\cO}_{r}$ under the norm map is $1+\mfp$. Therefore, for $\beta$ of type 1 and 2, we have
\[
\frac{|\cO_r^\times|}{|\det(\centbr^{1}K^{l})|}
<\frac{|\cO_r^\times|}{|\det(\centbr^1)|}
=\left|\frac{\cO_r^\times}{1+\mfp}\right|=q-1.
\]
Assume now that $\beta$ is of type $3$. We have
\begin{align*}
\frac{|\cO_{r}^{\times}|}{|\det(C^{1}K^{l})|} & =\frac{|\cO_{r}^{\times}|}{|\det(C_{l}^{1})|\cdot|1+\mfp^{l}|}=\frac{|\cO_{l}^{\times}|\cdot|\scentb{l}\cap K^{1}|}{|C_{l}^{1}|}=\frac{(q-1)q^{l}\cdot|\scentb{l}\cap K^{1}|}{q^{2(l-1)}}\\
& <\frac{(q-1)q^{l}\cdot|\scentb{l}|}{q^{2(l-1)}}\leq3(q-1)q^{\delta+3},
\end{align*}
where the estimate on the order of $\scentb{l}$ comes from Lemmas~\ref{lem:kernel-type-3} and \ref{lem:kernel-type-3-higher-level}.
\end{proof}

\begin{lem}\label{lem:twist_class_count_char_two}
 We have
\[
\#\twirr(G_r\mid\beta)<
\begin{cases}
q^{l+1} & \text{for $\beta$ of type $1$ or $2$} \\
3q^{l+\delta+4} & \text{for $\beta$ of type $3$}.
\end{cases}
\]
\end{lem}

\begin{proof}
We imitate the proof of Lemma~\ref{lem:twist_class_count_char_zero}. The second inequality in~\eqref{eq:ost-inequality} and Lemmas~\ref{lem:no_of_irreps} and \ref{lem:stabiliser_order_p_two} give the estimate
\[
\#\twirr(G_r\mid\beta)\le\frac{q^{-2r}|\centbr K^{l'}|\cdot\#[[\beta]]
\cdot|\Stab_{\cO_{l'}}[\beta]|}{|\chigroup|}
\cdot\frac{|\cO_r^\times|}{|\det(\centbr^{1}K^{l})|},
\]
which becomes
\[
\#\twirr(G_r\mid\beta)\le q^l\frac{q}{q-1}\cdot\frac{|\cO_r^\times|}{|\det(\centbr^{1}K^{l})|}.
\]
Application of Lemma~\ref{lem:double_ratio_char_two} yields the claim. 
\end{proof}

We now prove the main result of this subsection. Recall that we write $f(r)\asymp g(r)$ when $f(r)\ll g(r)$ and $g(r)\ll f(r)$ (see the Introduction).

\begin{thm}
Assume that the characteristic of $\cO$ is two. Then the abscissa of convergence of the twist zeta function $\twistGLtwo$ is $1$.
\end{thm}

\begin{proof}
Letting $r\ge 2$, we start by considering
\[
\tilde{r}_{d_{r}(i)}(G_r)
=\sum_{\beta\text{ of type $i$}}\#\widetilde{\Irr}(G_{r}\mid\beta),
\]
where $\beta$ runs through a set of representatives of the twist orbits of type $i$.
For $i\in\{1,2\}$, we use Lemma~\ref{lem:twist_class_count_char_two} with Lemma~\ref{lem:orbit_numbers}, to get
\[
\tilde{r}_{d_{r}(i)}(G_r)
\le B_{i}\cdot q^{l+1}
=(q-1)q^{l'-1}\cdot q^{l+1}<q^{r+1}.
\]

The case $i=3$ is slightly more complicated. By Lemma~\ref{lem:twist_class_count_char_two}, we have
\begin{equation}\label{eq:r_d_r(3)-upper-bound}
\tilde{r}_{d_{r}(3)}(G_r)<\sum_{w=1}^{l'}\sum_{\delta=0}^{\floor{w/2}}B(w,\delta)\cdot 3q^{l+\delta+4},
\end{equation}
By Lemma~\ref{lem:orbit_numbers_with_w_and_delta}, we have
\[
D(\delta)\asymp q^{l-\delta},
\]
and so
\[
B(w,\delta)\asymp\begin{cases}
q^{l-\delta} & \text{when $1\leq w<l'/2$}.\\
q^{3l/2-w-\delta} & \text{when $l'/2\leq w<l'$},\\
q^{l/2-\delta} & \text{when $w=l'$}.
\end{cases}
\]
Thus, \eqref{eq:r_d_r(3)-upper-bound} implies that
\begin{align*}
\tilde{r}_{d_{r}(3)}(G_{r}) & \ll\sum_{w=1}^{l'}\sum_{\delta=0}^{\floor{w/2}}B(w,\delta)\cdot q^{l+\delta}\\
& \asymp\sum_{w=1}^{\ceil{l'/2}-1}\sum_{\delta=0}^{\floor{w/2}}q^{2l}+\sum_{w=\ceil{l'/2}}^{l'-1}\sum_{\delta=0}^{\floor{w/2}}q^{5l/2-w}+\sum_{\delta=0}^{\floor{l'/2}}q^{3l/2}\\
& \ll l^{2}q^{2l}+\sum_{w=\ceil{l'/2}}^{l}lq^{5l/2-w}+lq^{3l/2}\\
& \ll l^2 q^{2l}+l^2 q^{2l}+l q^{3l/2}\ll r^2 q^{r}.
\end{align*}
From the upper bounds for $\tilde{r}_{d_r(i)}(G_r)$ derived above and the fact that $d_r(i)\asymp q^r$ for each $i\in\{1,2,3\}$ (Lemma~\ref{lem:reps-dim-mult-GL2}), it follows that for any real positive~$s$ there exists a positive real constant $A$ such that for any integer $R\geq 2$ we have
\[
\twistGLR(s)
\leq A\sum_{r=2}^R r^2 q^{(1-s)r}.
\]
Since the right hand side converges for any $s>1$ when $R\to\infty$, the abscissa of convergence of $\twistGLtwo=\lim_{R\to\infty}\twistGLR$ is at most 1.

We now prove that $1$ is also a lower bound for the abscissa by estimating the number of twist isoclasses of type 1. By the first inequality in \eqref{eq:ost-inequality}, we have
\[
\frac{\#\Irr(G_r\mid [[\beta]])}{(q-1)q^{r-1}}\leq \#\widetilde{\Irr}(G_r\mid\beta),
\]
for any $\beta$. Thus, for any $s\in\R$ and any integer $R\geq 2$, we have
\begin{align*}
\twistGLR(s) & \geq\sum_{r=2}^{R}\left(\sum_{\beta\text{ type }1}\#\widetilde{\Irr}(G_{r}\mid\beta)\right)d_{R}(1)^{-s}\\
& \geq\sum_{r=2}^{R}\frac{1}{(q-1)q^{r-1}}\left(\sum_{\beta\text{ type }1}\#\Irr(G_{r}\mid[[\beta]])\right)d_{R}(1)^{-s}\\
& =\sum_{r=2}^{R}\frac{1}{(q-1)q^{r-1}}r_{d_{R}(1)}(G_{R})\cdot d_{R}(1)^{-s}.
\end{align*}
It follows from Lemma~\ref{lem:reps-dim-mult-GL2} that for any real positive $s$ there exists a positive real constant $B$ such that for any integer $R\geq 2$, we have
\[
\twistGLR(s)\geq B\sum_{r=2}^{R}q^{(1-s)r}.
\]
The sum on the right hand side diverges for $s=1$ when $R\to\infty$, so the abscissa of convergence of $\twistGLtwo$ is at least $1$.

We have thus shown that the abscissa of $\twistGLtwo$ is precisely $1$.
\end{proof}

\section{Estimating the zeta function of $\SL_{2}(\F_q[[t]])$, $p=2$}
\label{sec:SL2-chapter}

Given the results in the previous sections, we know that for all $\cO$, such that $\chara\cO\neq 2$ (including $\Z_2$ and its extensions), the abscissa of convergence of $\SL_2(\cO)$ is $1$. In this section, we assume that $\chara\cO=2$, that is, $\cO=\F_q[[t]]$ where $q$ is a power of $p=2$. We will explicitly describe the representations of $\SL_{2}(\F_q[[t]])$, up to the orders of certain groups $V(\beta,\theta)$, and find estimates for the abscissa of convergence of its representation zeta function. 

We continue to consider a fixed but arbitrary $r\geq 2$, and preserve all the notation from the previous sections; in particular, $G_i=\GL_2(\cO_i)$. In addition, we set $S_i=\SL_2(\cO_i)$ and write $\SK^i$ for the kernel of the reduction map $\rho_{i}\colon S_r\rightarrow S_i$.

We start with a general summary of the representation theory of $S_r$ in terms of Clifford theory and orbits. This was first described in \cite[Section~3.1]{EDLs}, where further details can be found.

In connection with the representations of $G_{r}$, we have already
seen that every irreducible character of $K^{l}$ is of the form $\psi_{\beta}$,
for some $\beta\in\M_{2}(\cO_{l'})$. Restricting characters in $\Irr(K^{l})$
to $\SK^{l}$ gives rise to a surjective homomorphism
\[
\Irr(K^{l})\longrightarrow\Irr(\SK^{l}),\qquad\psi_{\beta}\longmapsto\psi_{\beta}|_{\SK^{l}}.
\]
The kernel of this homomorphism consists of those $\psi_{\beta}$
where $\beta$ is a scalar matrix. That is, if we let $Z=\left\{ \left[\begin{smallmatrix}a & 0\\
0 & a
\end{smallmatrix}\right]\mid a\in\cO_{l'}\right\} $, we get the following commutative diagram
\[
\begin{tikzcd}[column sep=0.4cm] 
\M_2(\cO_{l'})\arrow{r}\arrow{d}{\cong} & \M_2(\cO_{l'})/Z\arrow{d}{\cong}\\
\Irr(K^{l})\arrow{r} & \Irr(\SK^l)
\end{tikzcd}
\]
Hence the elements in $\Irr(\SK^{l})$ are of the form $\psi_{\beta+Z}$,
for $\beta+Z\in\M_{2}(\cO_{l'})/Z$, and $\psi_{\beta+Z}$ is given
by the same formula as $\psi_{\beta}$ (being the restriction of $\psi_{\beta}$
to $\SK^{l}$). The conjugation action of $G_{l'}$ on $\M_{2}(\cO_{l'})$
induces an action of $G_{l'}$ (and thus of $S_{l'}$) on $\M_{2}(\cO_{l'})/Z$. 

In analogy with what we did for the groups $G_r$, we will also write $\Irr(H\mid\beta+Z)$  for $\Irr(H\mid\psi_{\beta+Z})$, where $H$ is a subgroup of $S_r$ containing $\SK^l$. By a well known result in Clifford theory (see \cite[6.11]{Isaacs}), we have a bijection
\begin{align}
\Irr(\Stab_{S_{r}}(\psi_{\beta+Z})\mid\beta+Z) & \longiso\Irr(S_{r}\mid\beta+Z)\nonumber \\
\rho & \longmapsto\Ind_{\Stab_{S_{r}}(\psi_{\beta+Z})}^{S_{r}}\rho.\label{eq: bijection-Clifford-theory}
\end{align}
  To compute the representation zeta function of $S_r$, we thus need:
\begin{enumerate}
	\item a description and enumeration of the orbits of $S_{l'}$ acting on $\M_{2}(\cO_{l'})/Z$
	\item a description of the groups $\Stab_{S_{r}}(\psi_{\beta+Z})$ and an enumeration of the elements in $\Irr(\Stab_{S_{r}}(\psi_{\beta+Z})\mid\beta+Z)$ together with their dimensions.
\end{enumerate}
We will give a complete solution to the first of these points and a partial solution to the second. This will allow us to give estimates of the representation growth of $\SL_2(\cO)$. We begin by describing the orbits.

The $G_{l'}$-orbits in $\M_{2}(\cO_{l'})/Z$ are nothing but the
twist orbits considered in earlier sections. In the present section,
we will consider $S_{l'}$-orbits in $\M_{2}(\cO_{l'})/Z$. Excluding
orbits which are zero mod $\mfp$ (these correspond to representations
of $S_{r}$ which factor through $S_{r-1}$), there are three types
of $S_{l'}$-orbits in $\M_{2}(\cO_{l'})/Z$, represented by
matrices of the form $\begin{bmatrix}0 & \lambda\\
\Delta & \tau
\end{bmatrix}$, satisfying one of the following three conditions:
\begin{enumerate}
	\item $x^{2}+\tau x+\Delta$ has two distinct roots mod $\mfp$ and $\lambda\notin\mfp$.
	\item $x^{2}+\tau x+\Delta$ is irreducible mod $\mfp$ and $\lambda\notin\mfp$.
	\item $\Delta,\tau\in\mfp$ and $\lambda\notin\mfp$. 
\end{enumerate}
We will refer to these orbits as being of type 1, 2 and 3, respectively, and these are precisely the regular orbits.
These orbits are derived from the regular twist orbits for $G_{l'}$
and the unit $\lambda$ accounts for the $S_{l'}$-splittings of orbits.
Note that the above representatives do not all represent distinct
$S_{l'}$-orbits (two different values of $\lambda$ may result in
the same orbit), but two orbits of different type are never $S_{l'}$-conjugate.

Fix $\beta\in\M_{2}(\cO_{l'})$, where $\beta=\begin{bmatrix}0 & \lambda\\
\Delta & \tau
\end{bmatrix}$ is of any of the three types above. We also fix a lift
$\hat{\beta}\in\M_{2}(\cO_{r})$ of $\beta$ and use the notation
\[
C=C_{G_{r}}(\hat{\beta}),\qquad C_{i}=C_{G_{i}}(\rho_{i}(\hat{\beta})),\qquad SC_{i}=C_{i}\cap S_{i},
\]
for $r\geq i\geq1$.

\begin{lem}\label{lem:SL-twist-orb-1}
For any $\beta\in\M_2(\cO_{l'})$, the $G_{l'}$-orbit
$[\beta]$ is the union of precisely
\[
\left|\frac{\cO_{l'}^{\times}}{\det(\centbl)}\right|
\]
$S_{l'}$-orbits.
In particular, if $\beta$ is of type $1$ or $2$, then $[\beta]$ is one $S_{l'}$-orbit, and if $\beta$ is of type $3$, then $[\beta]$ is the union of 
\[
cq^{1+\delta}
\]
$S_{l'}$-orbits, for some $c\in\{1,2,3\}$.
\end{lem}

\begin{proof}
Let $\gamma\in [\beta]$ and $g\in G_{l'}$. We claim that the elements $\gamma$ and $g\gamma g^{-1}$ are conjugate under $S_{l'}$ if and
only if
\[
\det(g)\in\det(\centbl).
\]
The forward implication is obvious, and for the reverse, note that
$\det(g)\in\det(\centbl)$ implies $g\in\centbl S_{l'}$,
which implies that $\gamma$ and $g\gamma g^{-1}$ are 
$S_{l'}$-conjugate.
Thus $\gamma$ and $g\gamma g^{-1}$ are $S_{l'}$-conjugate
if and only if the image of $\det(g)$ is trivial in
$\cO_{l'}^{\times}/\det(\centbl)$.
It follows that
\[
h\gamma h^{-1}\longmapsto\det(h)\det(\centbl)
\]
induces a bijection between the set of $S_{l'}$-orbits in
the $G_{l'}$-orbit $[\beta]$ of $\beta$ and the group 
$\cO_{l'}^{\times}/\det(\centbl)$.

Now, when $\beta$ is of type $1$ or $2$, we have $\det(\centbl)=\cO_{l'}^\times$ (type $1$ being obvious, while type $2$ follows from the surjectivity of the norm, as in previous sections). Thus, if $\beta$ is of type $1$ or $2$, the $G_{l'}$-orbit $[\beta]$ equals the $S_{l'}$-orbit of $\beta$. Moreover, we have
\[
\left|\frac{\cO_{l'}^{\times}}{\det(\centbl)}\right|=\frac{(q-1)q^{r-1
}|\scentbl|}{|\centbl|}=\frac{(q-1)q^{r-1}|\scentbl|}{(q-1)q^{
2(l'-1)+1}},
\]
so when $\beta$ is of type $3$, Lemma~\ref{lem:kernel-type-3} implies the last assertion.
\end{proof}

\begin{lem}\label{lem:SL-twist-orb-2}
Assume that $\beta$ is regular and let $x\in\cO_{l'}$. Then $xI+\beta$ is $S_{l'}$-conjugate
to $\beta$ if and only if $xI+\beta$ is $G_{l'}$-conjugate
to $\beta$.\end{lem}
\begin{proof}
Every regular $S_{l'}$-orbit in $\M_{2}(\cO_{l'})$ has
a representative of the form
\[
\beta=\begin{bmatrix}0 & \lambda\\
\Delta & \tau
\end{bmatrix},
\]
where $\lambda\in\cO_{l'}^{\times}$. We have 
\[
\begin{bmatrix}1 & 0\\
x\lambda^{-1} & 1
\end{bmatrix}(xI+\beta)\begin{bmatrix}1 & 0\\
x\lambda^{-1} & 1
\end{bmatrix}=\begin{bmatrix}0 & \lambda\\
\Delta+x(x+\tau)\lambda^{-1} & \tau
\end{bmatrix},
\]
where we have used $x+x=0$, since $\chara\cO=2$. This shows
that if $x(x+\tau)=0$, then $xI+\beta$ is $S_{l'}$-conjugate
to $\beta$. Conversely, if $xI+\beta$ is $S_{l'}$-conjugate
to $\beta$, then by comparing determinants, we see that we must have
$x(x+\tau)=0$. On the other hand, by \eqref{eq:fixedpoints-equations}, we know that
$xI+\beta$ is $G_{l'}$-conjugate to $\beta$ if and only if $x(x+\tau)=0$,
which proves the lemma.
\end{proof}
Recall that we use $B(w,\delta)$ to denote the number of twist orbits of type $3$ in $\M_2(\cO_{l'})$ whose trace has valuation $w$ and whose odd depth is $\delta$. An immediate consequence of the two preceding lemmas is the following:
\begin{lem}\label{lem:SL-twist-orb-3}
	Let $B_{\SL}(w,\delta)$ denote the number of $S_{l'}$-twist orbits of type $3$ in $\M_2(\cO_{l'})$ whose trace has valuation $w$ and whose odd depth is $\delta$. Then  $$B(w,\delta)q^{1+\delta}\leq B_{\SL}(w,\delta)\leq 3\cdot B(w,\delta)q^{1+\delta},$$ so in particular
	$$B_{\SL}(w,\delta)\asymp B(w,\delta)q^{\delta}.$$
\end{lem} 
\begin{proof} 
	Let $\beta\in\M_2(\cO_{l'})$ be of type $3$ with trace $w$ and odd depth $\delta$. By Lemma~\ref{lem:SL-twist-orb-1} the $G_{l'}$-orbit $[\beta]$ splits into $cq^{1+\delta}$ $S_{l'}$-orbits, where $1\leq c\leq 3$. By Lemma~\ref{lem:SL-twist-orb-2}, passing to twist-orbits does not produce any further $S_{l'}$-splitting. This proves the inequalities and hence the asymptotic estimate.
\end{proof}

We now turn to the second goal mentioned in the beginning of the section, that is, a description of $\Stab_{S_{r}}(\psi_{\beta+Z})$ and the corresponding representations.

From the definition of $\psi_{\beta+Z}$, it is
easy to see that
\begin{equation}
\Stab_{S_{r}}(\psi_{\beta+Z})=\rho_{l'}^{-1}(C_{S_{l'}}(\beta+Z)),\label{eq:Stab(psi-bZ)-rhoC(bZ)}.
\end{equation}
Note that $G_{r}$ normalises $\SK^{l}$, so $G_{r}$ acts by the
``co-adjoint'' action on $\Irr(\SK^{l})$, and we have
\[
\Stab_{G_{r}}(\psi_{\beta+Z})\supseteq\Stab_{G_{r}}(\psi_{\beta})=CK^{l'}
\]
as well as 
\[
\Stab_{S_{r}}(\psi_{\beta+Z})\supseteq\Stab_{S_{r}}(\psi_{\beta})=\rho_{l'}^{-1}(\scentbl).
\]

In order to get a handle on $\Stab_{S_{r}}(\psi_{\beta+Z})$, we will
now determine the structure of $C_{S_{l'}}(\beta+Z)$. Let $\hat{\tau}=\tr(\hat{\beta})$,
and for $r\geq i\geq1$, write $\tau_{i}$ for the image of $\hat{\tau}$
in $\cO_{i}$; in particular, $\tau_{r}=\hat{\tau}$ and $\tau_{l'}=\tr(\beta)=\tau$.
Define the group
\[
U(\tau_{i})=
\left\{
\begin{bmatrix}
1 & 0\\
x & 1
\end{bmatrix}
\mid x\in\cO_{i},\,x(x+\tau_{i})=0\right\}.
\]
Note that this is indeed a group since $\chara\cO=2$. Let $\sigma_{i}\colon\left[\begin{smallmatrix}1 & 0\\
\cO_{i} & 1
\end{smallmatrix}\right]\rightarrow\left[\begin{smallmatrix}1 & 0\\
\cO_{r} & 1
\end{smallmatrix}\right]$ be the injective group homomorphism induced by the additive inclusion
$\cO_{i}\hookrightarrow\cO_{r}$ (note that such an injection does
not exist if $\chara\cO\neq2$). Then $\sigma_{i}$ is a section of
the reduction map $\rho_{i}\colon\left[\begin{smallmatrix}1 & 0\\
\cO_{r} & 1
\end{smallmatrix}\right]\rightarrow\left[\begin{smallmatrix}1 & 0\\
\cO_{i} & 1
\end{smallmatrix}\right]$, and from now on we identify $U(\tau_{i})$ with its image under
$\sigma_{i}$ and simply write $U(\tau_{i})$ for $\sigma_{i}(U(\tau_{i}))$. 

\begin{lem}\label{lem:cent-bZ}
For any $i$, such that $1\le i\le r$, the group $U(\tau_{i})$
normalises $\scentbi$ and $U(\tau_{i})\cap\scentbi=\{1\}$. Thus,
we have a semidirect product
\[
C_{S_{l'}}(\beta+Z)=U(\tau)\scentbl
\]
and
\[
\Stab_{S_{r}}(\psi_{\beta+Z})=U(\tau)\Stab_{S_{r}}(\psi_{\beta}).
\]
\end{lem}

\begin{proof}
Write $\beta_{i}=\begin{bmatrix}0 & \lambda_{i}\\
\Delta_{i} & \tau_{i}
\end{bmatrix}\in\M_{2}(\cO_{i})$ for $\rho_{i}(\hat{\beta})$. Since $\lambda_{i}^{-1}\beta_i=\begin{bmatrix}0 & 1\\
\lambda_{i}^{-1}\Delta_{i} & \lambda_{i}^{-1}\tau_{i}
\end{bmatrix}$, $C_{S_{l'}}(\beta+Z)=C_{S_{l'}}(\lambda^{-1}\beta+Z)$ and $\scentbi=C_{S_{i}}(\beta_{i})=C_{S_{i}}(\lambda_{i}^{-1}\beta_{i})$,
we may without loss of generality assume that $\lambda=1$. For $\begin{bmatrix}1 & 0\\
x & 1
\end{bmatrix}\in U(\tau_{i})$, we have 
\[
\begin{bmatrix}1 & 0\\
x & 1
\end{bmatrix}\beta_{i}\begin{bmatrix}1 & 0\\
x & 1
\end{bmatrix}=xI+\beta_{i},
\]
so, since $\scentbi=\cO_{i}[\beta]\cap S_{i}$, we have $\begin{bmatrix}1 & 0\\
x & 1
\end{bmatrix}\scentbi\begin{bmatrix}1 & 0\\
x & 1
\end{bmatrix}\subseteq\scentbi$. Moreover, since $\scentbi\subseteq\{aI+b\beta_{i}\mid a,b\in\cO_{i}\}$
and $\beta_{i}\notin U(\tau_{i})$, we have $U(\tau_{i})\cap\scentbi=\{1\}$.

Now, let $g\in C_{S_{l'}}(\beta+Z)$. Then $g\beta g^{-1}=xI+\beta$
for some $x\in\cO_{l'}$, so $x(x+\tau)=0$ (by taking determinants),
and thus $\begin{bmatrix}1 & 0\\
x & 1
\end{bmatrix}\in U(\tau)$. We also have $\begin{bmatrix}1 & 0\\
x & 1
\end{bmatrix}\beta\begin{bmatrix}1 & 0\\
x & 1
\end{bmatrix}=xI+\beta$, and so $g\in\begin{bmatrix}1 & 0\\
x & 1
\end{bmatrix}\scentbl$. Conversely, if $g\in\begin{bmatrix}1 & 0\\
x & 1
\end{bmatrix}\scentbl$ for some $x\in\cO_{l'}$ such that $x(x+\tau)=0$, then $g\in C_{S_{l'}}(\beta+Z)$. 

Since $\Stab_{S_{r}}(\psi_{\beta})=\rho_{l'}^{-1}(\scentbl)$, the
group $U(\tau)\Stab_{S_{r}}(\psi_{\beta})$ contains $\SK^{l'}$ and
maps surjectively onto $U(\tau)\scentbl$; hence $U(\tau)\Stab_{S_{r}}(\psi_{\beta})=\rho_{l'}^{-1}(U(\tau)\scentbl)$.
The expression for $\Stab_{S_{r}}(\psi_{\beta+Z})$ now follows from
\eqref{eq:Stab(psi-bZ)-rhoC(bZ)}.
\end{proof}

We will now determine the structure of the group $U(\tau_{i})$. In
the following, we will only need $U(\tau_{l})$ and $U(\tau)=U(\tau_{l'})$,
but it is not harder to prove the general case.

\begin{lem}\label{lem:U(tau)}
For any $i$, such that $1\le i\le r$, let $\beta_{i}=\rho_{i}(\hat{\beta})$
and $\tau_{i}=\tr(\beta_{i})$, as above. If $\beta$ is of type $1$
or $2$, we have 
\[
U(\tau_{i})=\left\{ 1,\left[\begin{smallmatrix}1 & 0\\
\tau_{i} & 1
\end{smallmatrix}\right]\right\} .
\]
If $\beta$ is of type $3$, we have 
\[
U(\tau_{i})=\begin{cases}
U^{i-w}\cup\left[\begin{smallmatrix}1 & 0\\
\tau_{i} & 1
\end{smallmatrix}\right]U^{i-w} & \text{if }w<\ceil{i/2},\\
U^{\ceil{i/2}} & \text{if }w\geq\ceil{i/2},
\end{cases}
\]
where $U^{j}=U_{i}^{j}=\begin{bmatrix}1 & 0\\
\mfp^{j} & 1
\end{bmatrix}\subseteq\M_{2}(\cO_{i})$, for any $j\geq1$, and $w=v(\tau_{i})$.
In particular, when $\beta$ is of type $3$, we have
\[
|U(\tau_{i})|=\begin{cases}
2q^{w} & \text{if }w<\ceil{i/2},\\
q^{\floor{i/2}} & \text{if }w\geq\ceil{i/2}.
\end{cases}
\]
\end{lem}

\begin{proof}
Assume that $\beta$ is of type $1$ or $2$. Then $\tau\notin\mfp$,
because if $\tau\in\mfp$, then $\beta_{1}$ has precisely one eigenvalue
in $\F_{q}$, and this is impossible for $\beta$ of type $1$ or
$2$. Thus $\tau$ is a unit, hence $\tau_{i}$ is a unit, so if $x\in U(\tau_{i})$
so that $x(x+\tau_{i})=0$, then either $x=0$ (if $x\in\mfp$), or
$x=\tau_i$ (if $x\not\in\mfp$).

Now suppose that $\beta$ is of type $3$; then $\tau$, hence $\tau_{i}$,
is not a unit. Assume that $w<\ceil{i/2}$. Then $i-w>i-\ceil{i/2}=\floor{i/2}$,
so $i-w\geq\ceil{i/2}$. Thus $U^{i-w}\subseteq U(\tau_{i})$, and
thus $U^{i-w}\cup\left[\begin{smallmatrix}1 & 0\\
\tau_{i} & 1
\end{smallmatrix}\right]U^{i-w}\subseteq U(\tau_{i})$. Conversely, let $\begin{bmatrix}1 & 0\\
x & 1
\end{bmatrix}\in U(\tau_{i})$, with $x\in\cO_{i}$, and let $a=v(x)$. Then $x(x+\tau_{i})=0$,
so $v(x+\tau_{i})\geq i-a$ and hence $x\equiv\tau_{i}\mod\mfp^{i-a}$.
If $a<w$, then $a=v(x+\tau_{i})\geq i-a$, so $w>a\geq\ceil{i/2}$.
This is a contradiction, so $a\geq w$. If $a>w$, then $v(x+\tau_{i})=w\geq i-a$,
so $a\geq i-w$, that is, $\begin{bmatrix}1 & 0\\
x & 1
\end{bmatrix}\in U^{i-w}$. If $a=w$, then $x\equiv\tau_{i}\mod\mfp^{i-w}$, which is equivalent
to $\begin{bmatrix}1 & 0\\
x & 1
\end{bmatrix}\in\begin{bmatrix}1 & 0\\
\tau_{i} & 1
\end{bmatrix}U^{i-w}$. Thus, when $w<\ceil{i/2}$, we have $U(\tau_{i})=U^{i-w}\cup\left[\begin{smallmatrix}1 & 0\\
\tau_{i} & 1
\end{smallmatrix}\right]U^{i-w}$.

Assume next that $w\geq\ceil{i/2}$. Then $U^{\ceil{i/2}}\subseteq U(\tau_{i})$.
Conversely, let $\begin{bmatrix}1 & 0\\
x & 1
\end{bmatrix}\in U(\tau_{i})$, with $x\in\cO_{i}$, and let $a=v(x)$, as before. If $a<w$, then
$a=v(x+\tau_{i})\geq i-a$, that is, $a\geq\ceil{i/2}$, so $\begin{bmatrix}1 & 0\\
x & 1
\end{bmatrix}\in U^{\ceil{i/2}}$. If $a>w$, then $a>v(x+\tau_{i})=w\geq i-a$, so again $a\geq\ceil{i/2}$,
and $\begin{bmatrix}1 & 0\\
x & 1
\end{bmatrix}\in U^{\ceil{i/2}}$. Finally, if $a=w$, we have $\begin{bmatrix}1 & 0\\
x & 1
\end{bmatrix}\in U^{w}\subseteq U^{\ceil{i/2}}$. Thus, when $w\geq\ceil{i/2}$, we have $U(\tau_{i})=U^{\ceil{i/2}}$.
\end{proof}

\begin{lem}\label{lem:U(tau)-ratio-1-2-q}
For $\beta$ of any type, we have 
\[
\frac{|U(\tau_{l})|}{|U(\tau)|}\in\{1,2,q\}.
\]
\end{lem}

\begin{proof}
When $r$ is even, we have $l=l'$, hence $U(\tau_{l})=U(\tau)$,
so there is nothing to prove. Assume now that $r$ is odd, so that
$l'=l-1$. We use the formula for $|U(\tau_{i})|$ from Lemma~\ref{lem:U(tau)}
in the cases where $i$ is $l$ and $l'$. If $\beta$ is of type 1 or 2, the assertion is clear, so assume now that $\beta$ is of type 3. 

Assume first that $\tau=0$. Then $v(\tau)=l'$ and $l'\geq \ceil{l'/2}$ for all $l'\geq 1$, so $|U(\tau)|=q^{\floor{l'/2}}$, by Lemma~\ref{lem:U(tau)}. Meanwhile, $v(\tau_l)\in\{l,l'\}$, and since $l'=l-1\geq 1$, we have $l>l'\geq \ceil{l/2}$, so $|U(\tau_l)|=q^{\floor{l/2}}$. Thus $|U(\tau_l)|/|U(\tau)|\in\{1,q\}$ when $\tau=0$.

Assume now that $\tau\neq 0$; then $w=v(\tau_{l})=v(\tau)$. If $w<\ceil{l'/2}$, then also $w<\ceil{l/2}$,
and in this case $|U(\tau_{l})|=2q^{w}=|U(\tau)|$. If $w\geq\ceil{l'/2}$
and $w\geq\ceil{l/2}$, then $|U(\tau_{l})|=q^{\floor{l/2}}$ and
$|U(\tau)|=q^{\floor{(l-1)/2}}$, so $\frac{|U(\tau_{l})|}{|U(\tau)|}\in\{1,q\}$.

Assume finally that $\ceil{l'/2}\le w<\ceil{l/2}$. Then, if
 $l$ were even, we would have $\frac{l}{2}=\ceil{\frac{l'}{2}}\leq w<\frac{l}{2}$,
which is impossible. Thus $l$ is odd, so that
\[
\frac{l-1}{2}\leq w<\frac{l+1}{2},
\]
whence it follows that $w=(l-1)/2$. Therefore, we have
\[
\frac{|U(\tau_{l})|}{|U(\tau)|}=\frac{2q^{w}}{q^{\floor{(l-1)/2}}}=\frac{2q^{(l-1)/2}}{q^{(l-1)/2}}=2.\qedhere
\]
\end{proof}

We now give an approximate description of the representations in $\Irr(S_{r}\mid \beta+Z)$,
for $\beta$ of type 1, 2 or 3. Since $\centbr$ is abelian, we know
that $\psi_{\beta}\in\Irr(K^{l})$ has an extension to $\centbr K^{l}$.
It follows by restriction of this extension that $\psi_{\beta+Z}\in\Irr(\SK^{l})$
has an extension to $\centbr K^{l}\cap S_{r}$. Now, as in the proof
of Lemma~\ref{lem:CK^lcapS_r-onto-SC_l}, it is easy to see that $\rho_l\colon CK^{l}\cap S_{r}\rightarrow SC_{l}$ is surjective: for $t\in SC_l$, any lift $\hat{t}\in C$ (which exists since $C\rightarrow C_l$ is surjective) satisfies $\det(\hat{t})\in 1+\mfp^l$, so there exists a $k\in K^l$ such that $\hat{t}k\in CK^l\cap S_r$, and $\rho_l(\hat{t}k)=t$; thus
\begin{equation}
CK^{l}\cap S_{r}=\rho_{l}^{-1}(SC_{l}).\label{eq:CK^lcapS-rho(SC_l)}
\end{equation}

By Lemma~\ref{lem:cent-bZ}, $U(\tau_{l})$ normalises $SC_{l}$,
so $U(\tau_{l})$ (considered as a subgroup of $S_{r}$) normalises
$\rho_{l}^{-1}(SC_{l})$. For an extension $\theta\in\Irr(\rho_{l}^{-1}(SC_{l})\mid \beta+Z)$
of $\psi_{\beta+Z}$, let 
\[
V(\beta,\theta)=\Stab_{U(\tau_{l})}(\theta).
\]
Note that we regard $V(\beta,\theta)$ as a subgroup of $S_{r}$.
Then, since $V(\beta,\theta)$ is abelian, the character $\theta$
extends to $V(\beta,\theta)\rho_{l}^{-1}(SC_{l})$, and by standard
Clifford theory \cite[6.11]{Isaacs}, since 
\[
\Stab_{U(\tau_l)\rho_{l}^{-1}(SC_{l})}(\theta)=V(\beta,\theta)\rho_{l}^{-1}(SC_{l}),
\]
any extension of $\theta$ to $V(\beta,\theta)\rho_{l}^{-1}(SC_{l})$
induces irreducibly to $U(\tau_l)\rho_{l}^{-1}(SC_{l})$.

Consider the following diagrams of groups and representations of the corresponding groups. The lines between the groups indicate containment of groups, and the lines between representations indicate that the restriction of a representation above contains a representation below as an irreducible constituent.
\[
\begin{tikzcd}[column sep=0.4cm] 
S_r\arrow[dash]{d}\\
U(\tau)\rho_{l'}^{-1}(\scentbl)\arrow[dash]{d}\\
U(\tau_l)\rho_{l}^{-1}(SC_{l})\arrow[dash]{d}\\
V(\beta,\theta)\rho_{l}^{-1}(SC_{l})\arrow[dash]{d}\\
\rho_{l}^{-1}(SC_{l})\arrow[dash]{d}\\
\SK^l
\end{tikzcd}
\qquad\qquad
\begin{tikzcd}[column sep=0.4cm] 
\vphantom{S_r}\rho=\Ind\eta\arrow[dash]{d}\\
\vphantom{U(\tau_l)\rho_{l}^{-1}(SC_{l})}\eta\in\Ind\kappa\arrow[dash]{d}\\
\vphantom{U(\tau_l)\rho_{l}^{-1}(SC_{l})}\kappa\arrow[dash]{d}\\
\vphantom{U(\tau_l)\rho_{l}^{-1}(SC_{l})}\hat{\theta}\arrow[dash]{d}\\
\vphantom{U(\tau_l)\rho_{l}^{-1}(SC_{l})}\theta\arrow[dash]{d}\\
\vphantom{\SK^l}\psi_{\beta+Z}.
\end{tikzcd}
\]
The rightmost diagram illustrates how an arbitrary $\rho\in\Irr(S_{r}\mid\theta)$
is obtained: Since $\Stab_{S_r}(\psi_{\beta+Z})=U(\tau)\rho_{l'}^{-1}(\scentbl)$, there exists an $\eta\in\Irr(U(\tau)\rho_{l'}^{-1}(\scentbl)\mid\theta)$ such that $\rho=\Ind_{U(\tau)\rho_{l'}^{-1}(\scentbl)}^{S_{r}}\eta$. Moreover, by the above paragraph, there exists an extension $\hat{\theta}$ of $\theta$ such that $\eta$ is an irreducible constituent of $$\Ind_{V(\beta,\theta)\rho_{l}^{-1}(SC_{l})}^{U(\tau)\rho_{l'}^{-1}(\scentbl)}\hat{\theta}=\Ind_{U(\tau_l)\rho_{l}^{-1}(SC_{l})}^{U(\tau)\rho_{l'}^{-1}(\scentbl)}\kappa,
$$
where $\kappa:=\Ind_{V(\beta,\theta)\rho_{l}^{-1}(SC_{l})}^{U(\tau_l)\rho_{l}^{-1}(SC_{l})} \hat{\theta}$ is irreducible.

The group $U(\tau_l)$ acts on $\Irr(\rho_{l}^{-1}(SC_{l}))$ by conjugation, and for each orbit we can choose a representative $\theta$. Then all the extensions $\hat{\theta}$ of $\theta$ induce to distinct representations of $U(\tau_l)\rho_{l}^{-1}(SC_{l})$. Note that if we choose another representative $\theta'$, then we end up with the same set of representations of $U(\tau_l)\rho_{l}^{-1}(SC_{l})$.

\begin{defn}\label{def:u_r(beta,theta)}
Let $u_{r}(\beta,\theta)\in\R$ be such that
\[
q^{u_{r}(\beta,\theta)}=|V(\beta,\theta)|.
\]	
\end{defn}
Note that if $q$ is a prime, $u_r(\beta,\theta)$ is  an integer, since $V(\beta,\theta)$ is a $p$-group. We do not know whether $u_r(\beta,\theta)$ is always an integer, and in fact we have not been able to determine the function $u_r(\beta,\theta)$. Nevertheless, we can express the asymptotic number of representations of $S_r$ and their dimensions in terms of $u_r(\beta,\theta)$, and this will be sufficient to establish non-trivial bounds on the abscissa of convergence of $\SL_2(\cO)$.

\begin{lem}\label{lem:dim-num_theta-asymp}
Let $\theta\in\Irr(\Stab_{S_{r}}(\psi_{\beta})\mid \beta+Z)$.
Then, for any $\rho\in\Irr(S_{r}\mid\theta)$, we have
\[
\dim\rho\asymp
\begin{cases}
q^{r} & \text{for $\beta$ of type $1$ or $2$},\\
q^{r-u_{r}(\beta,\theta)-\delta} & \text{for $\beta$ of type $3$}.
\end{cases}
\]
Moreover,
\[
\#\Irr(S_{r}\mid\theta)\asymp
\begin{cases}
1 & \text{for $\beta$ of type $1$ or $2$},\\
q^{u_{r}(\beta,\theta)} & \text{for $\beta$ of type $3$}.
\end{cases}
\]
In all cases, the implicit constants can be taken to be independent of $\beta$ and $\theta$.
\end{lem}

\begin{proof}
By the description of representations just before Definition~\ref{def:u_r(beta,theta)}, we have
\[
\dim\rho\leq\left|\frac{S_{r}}{V(\beta,\theta)\rho_{l}^{-1}(SC_{l})}\right|=\frac{|S_{r}/\SK^{l}|}{q^{u_{r}(\beta,\theta)}|\rho_{l}^{-1}(SC_{l})/\SK^{l}|}=\frac{|S_{l}|}{q^{u_{r}(\beta,\theta)}|SC_{l}|}.
\]
Now, $|S_{l}|\asymp q^{3l}$ and by Lemmas \ref{lem:scentbr-1-2-3podd}, \ref{lem:kernel-type-3} and \ref{lem:kernel-type-3-higher-level},
\[
|SC_{l}|\asymp\begin{cases}
q^{l} & \text{for $\beta$ of type $1$ or $2$},\\
q^{l+\delta} & \text{for $\beta$ of type $3$},
\end{cases}
\]
for some implicit constants independent of $\beta$ and $\theta$. Moreover, by Lemma~\ref{lem:U(tau)}, when $\beta$ is of type $1$ or $2$, we have $|U(\tau_{l})|=2$, hence $0\leq u_{r}(\beta,\theta)\leq 1$, so in this case $\dim\rho\ll q^{r}$. When $\beta$ is of type $3$, the above estimates imply that 
\[
\dim\rho\ll q^{r-u_{r}(\beta,\theta)-\delta}.
\]

On the other hand, the dimension of $\Ind_{V(\beta,\theta)\rho_{l}^{-1}(SC_{l})}^{U(\tau_{l})\rho_{l}^{-1}(SC_{l})}\hat{\theta}$
is a lower bound for $\dim\eta$ (by Frobenius reciprocity), and since
$U(\tau)\rho_{l'}^{-1}(\scentbl)=\Stab_{S_{r}}(\psi_{\beta+Z})$,
standard Clifford theory \cite[6.11]{Isaacs} implies that $\Ind_{U(\tau)\rho_{l'}^{-1}(\scentbl)}^{S_{r}}\eta$
is irreducible, so
\begin{align*}
\dim\rho & \geq\left|\frac{U(\tau_{l})\rho_{l}^{-1}(SC_{l})}{V(\beta,\theta)\rho_{l}^{-1}(SC_{l})}\right|\cdot\left|\frac{S_{r}}{U(\tau)\rho_{l'}^{-1}(\scentbl)}\right|\\
& =\frac{|U(\tau_{l})|\cdot|SC_{l}|}{|V(\beta,\theta)|\cdot|SC_{l}|}\cdot\frac{|S_{l'}|}{|U(\tau)|\cdot|\scentbl|}\\
& =\frac{|U(\tau_{l})|}{|U(\tau)|}\cdot\frac{|S_{l'}|}{q^{u_{r}(\beta,\theta)}|\scentbl|}\\
& \geq\frac{|S_{l'}|}{q^{u_{r}(\beta,\theta)}|\scentbl|}\text{\qquad by Lemma~\ref{lem:U(tau)-ratio-1-2-q}}\\
& \asymp \begin{cases}
q^{r} & \text{for $\beta$ of type $1$ or $2$},\\
q^{r-u_{r}(\beta,\theta)-\delta} & \text{for $\beta$ of type $3$}.
\end{cases}
\end{align*}
We have thus proved the assertion about $\dim\rho$.

We now prove the assertion about $\#\Irr(S_{r}\mid\theta)$. The number
of extensions of $\theta$ to $V(\beta,\theta)\rho_{l}^{-1}(SC_{l})$
is $\left|\frac{V(\beta,\theta)\rho_{l}^{-1}(SC_{l})}{\rho_{l}^{-1}(SC_{l})}\right|$,
and each such extension induces irreducibly to $U(\tau_{l})\rho_{l}^{-1}(SC_{l})$.
Each representation of $U(\tau_{l})\rho_{l}^{-1}(SC_{l})$ thus obtained
has at most $\left|\frac{U(\tau)\rho_{l'}^{-1}(\scentbl)}{U(\tau_{l})\rho_{l}^{-1}(SC_{l})}\right|$
irreducible representations of $U(\tau)\rho_{l'}^{-1}(\scentbl)$
lying over it, and each such representation of $U(\tau)\rho_{l'}^{-1}(\scentbl)$
induces irreducibly to $S_{r}$. Thus 
\begin{align*}
\#\Irr(S_{r}\mid\theta) & \leq\left|\frac{V(\beta,\theta)\rho_{l}^{-1}(SC_{l})}{\rho_{l}^{-1}(SC_{l})}\right|\cdot\left|\frac{U(\tau)\rho_{l'}^{-1}(\scentbl)}{U(\tau_{l})\rho_{l}^{-1}(SC_{l})}\right|\\
& =|V(\beta,\theta)|\cdot\left|\frac{U(\tau)}{U(\tau_{l})}\right|\cdot\frac{|\scentbl|\cdot|\SK^{l'}|}{|SC_{l}|\cdot|\SK^{l}|}\\
& \ll q^{u_{r}(\beta,\theta)},
\end{align*}
where, in the last step, we have used Lemma~\ref{lem:U(tau)-ratio-1-2-q} and Lemmas~\ref{lem:kernel-type-3} and \ref{lem:kernel-type-3-higher-level}.

On the other hand, if every extension of $\theta$ to $V(\beta,\theta)\rho_{l}^{-1}(SC_{l})$
induces irreducibly to $S_{r}$, we get the lower bound
\[
\#\Irr(S_{r}\mid\theta)\geq\left|\frac{V(\beta,\theta)\rho_{l}^{-1}(SC_{l})}{\rho_{l}^{-1}(SC_{l})}\right|=q^{u_{r}(\beta,\theta)}.
\]
It remains to note that when $\beta$ is of type $1$ or $2$, Lemma~\ref{lem:U(tau)} implies that ${q^{u_{r}(\beta,\theta)}\asymp 1}$.
\end{proof}
The lower bound $1$ for the abscissa in the following theorem follows from \cite[Proposition~6.6]{Larsen-Lubotzky}. We give an independent proof of this lower bound to illustrate our method.
\begin{thm}\label{thm:SL2-Ochar2}
	Assume that $\chara\cO=2$. Then the abscissa of convergence of $\zeta_{\SL_{2}(\cO)}(s)$ lies in the interval $[1,\,5/2]$.
\end{thm}

\begin{proof}
	For the abscissa it is enough to consider $\zeta_{\SL_{2}(\cO)}(s)$ for $s\in\R$,
	and since we know that the Dirichlet series defining the zeta function diverges for $s=0$, we henceforth assume that $s\in\R$ satisfies $s>0$ (this assumption will be used later in the proof). We have
	\[
	\zeta_{\SL_{2}(\cO)}(s)=\zeta_{\SL_{2}(\F_{q})}(s)+\sum_{r=2}^{\infty}\Bigl(\zeta_{S_{r}}^{1}(s)+\zeta_{S_{r}}^{2}(s)+\zeta_{S_{r}}^{3}(s)\Bigr),
	\]
	as formal Dirichlet series, where $\zeta_{S_r}^{i}(s)$ is
	defined to be the Dirichlet series counting only primitive representations
	of $S_{r}$ of type $i$. 
	
	We first deal with the easier parts $\zeta_{S_{r}}^{1}(s)$ and $\zeta_{S_{r}}^{2}(s)$,
	so assume that $r\ge 2$ and $\beta\in\M_2(\cO_{l'})$ is of type $1$ or $2$. Then, for any $\rho\in\Irr(S_{r}\mid\beta+Z)$,
	Lemma~\ref{lem:dim-num_theta-asymp} implies that $\dim\rho\asymp q^{r}$.
	Moreover, the number of extensions of
	$\psi_{\beta+Z}$ to $\rho_{l}^{-1}(SC_{l})$ is $\frac{|\rho_{l}^{-1}(SC_{l})|}{|\SK^{l}|}$, which equals $|SC_{l}|$ by \eqref{eq:CK^lcapS-rho(SC_l)}. Hence, by Lemmas \ref{lem:dim-num_theta-asymp} and \ref{lem:scentbr-1-2-3podd}, we have
	\[
	\#\Irr(S_{r}\mid\beta+Z)\asymp|SC_{l}|\asymp q^{l}.
	\]
	By Lemma~\ref{lem:orbit_numbers}, we have $B_{i}\asymp q^{l}$ for $i\in\{1,2\}$, and by Lemma~\ref{lem:SL-twist-orb-1}, there is no $S_{l'}$-splitting
	of a $G_{l'}$-twist orbit of type $1$ or $2$, so we conclude
	that 
	\[
	\zeta_{S_{r}}^{i}(s)\asymp q^{2l}q^{-sr}=q^{r(1-s)},\qquad\text{for $i\in{1,2}$}.
	\]
	
	Assume now that $\beta$ is of type $3$. By Lemma~\ref{lem:dim-num_theta-asymp},
	every $\rho\in\Irr(S_{r}\mid\beta+Z)$ satisfies
	\[
	\dim\rho\asymp q^{r-\delta-u_{r}(\beta,\theta)}.
	\]
	Moreover, since the orbit of any $\theta\in\Irr(\rho_{l}^{-1}(SC_{l})\mid \beta+Z)$ under the action of $U(\tau_l)$ has size $|U(\tau_l)/V(\beta,\theta)|$, we have  
	\[
	\#\Irr(S_r\mid\beta+Z)=\sum_{\theta\in\Irr(\rho_{l}^{-1}(SC_{l})\mid\beta+Z)}\frac{\#\Irr(S_r\mid \theta)}{|U(\tau_l)|/|V(\beta,\theta)|}.
	\]
	Thus, by Lemmas \ref{lem:U(tau)-ratio-1-2-q} and \ref{lem:dim-num_theta-asymp},
	\[
	\#\Irr(S_{r}\mid\beta+Z)\asymp\sum_{\theta\in\Irr(\rho_{l}^{-1}(SC_{l})\mid\beta+Z)}\frac{q^{2u_{r}(\beta,\theta)}}{|U(\tau)|}.
	\]
Let $X_{l'}$ denote a complete set representatives of the $S_{l'}$-twist
orbits of elements $\beta\in\M_{2}(\cO_{l'})$ of type $3$.
The above estimates for the dimensions and multiplicities of representations in $\Irr(S_r\mid \beta+Z)$
then imply that 
\begin{align}
\zeta_{S_{r}}^{3}(s) & \asymp\sum_{\beta\in X_{l'}}\sum_{\theta\in\Irr(\rho_{l}^{-1}(SC_{l})\mid\beta+Z)} 
\frac{q^{2u_{r}(\beta,\theta)}}{|U(\tau)|}q^{-s(r-\delta(\beta)-u_{r}(\beta,\theta))}\nonumber\\
& =\sum_{\beta\in X_{l'}}\sum_{\theta\in\Irr(\rho_{l}^{-1}(SC_{l})\mid\beta+Z)}
\frac{1}{|U(\tau)|}q^{u_{r}(\beta,\theta)(2+s)-s(r-\delta(\beta))}.\label{eq:zeta3-asymp}
\end{align}

We now estimate $\zeta_{S_{r}}^{3}(s)$ from above. By definition,  $q^{u_r(\beta,\theta)}\leq |U(\tau_l)| \leq |U(\tau)|$ so by Lemma~\ref{lem:U(tau)}, we have 
\[
1\leq q^{u_r(\beta,\theta)}\leq 2q^{\ceil{l/2}}\asymp q^{l/2}.
\]
Hence, since $2+s> 0$ (by our assumption that $s>0$), \eqref{eq:zeta3-asymp} gives the upper bound
\[
\zeta_{S_{r}}^{3}(s)\ll\sum_{\beta\in X_{l'}}\sum_{\theta\in\Irr(\rho_{l}^{-1}(SC_{l})\mid\beta+Z)}q^{(l/2)(1+s)-s(r-\delta(\beta))}.
\]
Recalling the notation $B_{\SL}(w,\delta)$ from Lemma~\ref{lem:SL-twist-orb-3}, we furthermore have
\begin{align*}
\zeta_{S_{r}}^{3}(s) & \ll\sum_{\beta\in X_{l'}}\sum_{\theta\in\Irr(\rho_{l}^{-1}(SC_{l})\mid\beta+Z)}q^{l(1+s)/2-s(r-\delta(\beta))}\asymp\sum_{\beta\in X_{l'}}|SC_l|\cdot q^{l(1+s)/2-s(r-\delta(\beta))}\\
& \asymp\sum_{\beta\in X_{l'}}q^{l+\delta(\beta)}\cdot q^{l(1+s)/2-s(r-\delta(\beta))}=\sum_{w=1}^{l'}\sum_{\delta=0}^{\floor{w/2}}B_{\SL}(w,\delta)q^{l+\delta}\cdot q^{l(1+s)/2-s(r-\delta)}\\
& \asymp\sum_{w=1}^{l'}\sum_{\delta=0}^{\floor{w/2}}B(w,\delta)q^{\delta}q^{l+(\delta+l/2)(1+s)-sr},
\end{align*}
where the estimate for the order of $SC_l$ comes from Lemmas~\ref{lem:kernel-type-3} and \ref{lem:kernel-type-3-higher-level}, and in the last step, we have applied Lemma~\ref{lem:SL-twist-orb-3}.
By Lemma~\ref{lem:orbit_numbers_with_w_and_delta}, we have explicit expressions for $B(w,\delta)$, for $w$ and $\delta$ in three different ranges. Applying this, and working up to constants independent of $r$, we can change $l'$ to $l$ everywhere, and obtain 
\begin{align*}
\zeta_{S_{r}}^{3}(s) & \ll\sum_{w=1}^{l'}\sum_{\delta=0}^{\floor{w/2}}B(w,\delta)q^{\delta}q^{l+(\delta+l/2)(1+s)-sr}\\
& \asymp\sum_{w=1}^{\ceil{l'/2}-1}\sum_{\delta=0}^{\floor{w/2}}q^{l-\delta}q^{\delta+l+(\delta+l/2)(1+s)-sr}\\
& +\sum_{w=\ceil{l'/2}}^{l'-1}\sum_{\delta=0}^{\floor{w/2}}q^{3l/2-w-\delta}q^{\delta+l+(\delta+l/2)(1+s)-sr}+\sum_{\delta=0}^{\floor{l'/2}}q^{l/2-\delta}q^{\delta+l+(\delta+l/2)(1+s)-sr}\\
& \asymp\sum_{w=1}^{\ceil{l/2}}\sum_{\delta=0}^{\floor{w/2}}q^{2l+(\delta+l/2)(1+s)-sr}\\
& +\sum_{w=\ceil{l/2}}^{l}\sum_{\delta=0}^{\floor{w/2}}q^{5l/2-w+(\delta+l/2)(1+s)-sr}+\sum_{\delta=0}^{\floor{l/2}}q^{3l/2+(\delta+l/2)(1+s)-sr}\\
& \asymp q^{2l+(l/4+l/2)(1+s)-sr}+\sum_{w=\ceil{l/2}}^{l}q^{w(-1+(1+s)/2)+5l/2+(l/2)(1+s)-sr}\\
&+q^{3l/2+(l/2+l/2)(1+s)-sr}\\
& \asymp q^{r(11-5s)/8}+\sum_{w=\ceil{r/4}}^{\ceil{r/2}}q^{w(-1+(1+s)/2)+3r(2-s)/4}+q^{r(5-2s)/4}.
\end{align*}
In the last step, we have substituted $r/2$ for $l$. Assume that $s>1$, so that the coefficient $-1+(1+s)/2$ of $w$ above is positive. We then have
$$
\sum_{w=\ceil{r/4}}^{\ceil{r/2}}q^{w(-1+(1+s)/2)+3r(2-s)/4}\asymp q^{r/2(-1+(1+s)/2)+3r(2-s)/4}=q^{r(5-2s)/4},
$$
and thus, by the above upper bound,
$$
\zeta_{S_{r}}^{3}(s)\ll q^{r(5-2s)/4}.
$$
We have already shown that $\zeta_{S_{r}}^{i}(s)\ll q^{r(1-s)}$
for $i\in\{1,2\}$, so we conclude that for any $s>1$, there exists
a positive real constant $A$ such that for all $R\geq2$, we have 
\[
\zeta_{\SL_{2}(\cO_{R})}(s)\leq A\sum_{r=2}^{R}q^{r(5-2s)/4}.
\]
This upper bound converges for $s>5/2$, as $R\to\infty$. Thus, since $\zeta_{\SL_{2}(\cO)}(s)=\lim_{R\to\infty}\zeta_{\SL_{2}(\cO_{R})}(s)$,
the abscissa of convergence of $\zeta_{\SL_{2}(\cO)}(s)$ is at most
$5/2$.

We now estimate $\zeta_{S_{r}}^{3}(s)$ from below.  By Lemma~\ref{lem:U(tau)} we have $|U(\tau)|\leq q^{l/2}$ and trivially, $0\leq u_{r}(\beta,\theta)$, so (using that $2+s> 0$ by our assumption that $s>0$) equation \eqref{eq:zeta3-asymp} gives the lower bound
\[
\zeta_{S_{r}}^{3}(s)\gg \sum_{\beta\in X_{l'}}\sum_{\theta\in\Irr(\rho_{l}^{-1}(SC_{l})\mid\beta+Z)}
q^{-l/2-s(r-\delta(\beta))}.
\]
Making the analogous simplifications as for the upper bound (and using the same lemmas), this gives
\begin{align*}
\zeta_{S_{r}}^{3}(s) & \gg\sum_{\beta\in X_{l'}}|SC_l|\cdot q^{-l/2-s(r-\delta(\beta))}\asymp\sum_{\beta\in X_{l'}}q^{l/2+\delta(\beta)-s(r-\delta(\beta))}\\
& =\sum_{w=1}^{l'}\sum_{\delta=0}^{\floor{w/2}}B_{\SL}(w,\delta)q^{l/2+\delta-s(r-\delta)}\\
& \asymp\sum_{w=1}^{l'}\sum_{\delta=0}^{\floor{w/2}}B(w,\delta)q^{\delta+l/2+\delta-s(r-\delta)}\\
& \asymp\sum_{w=1}^{\ceil{l/2}}\sum_{\delta=0}^{\floor{w/2}}q^{3l/2+\delta(1+s)-sr}
 +\sum_{w=\ceil{l/2}}^{l}\sum_{\delta=0}^{\floor{w/2}}q^{2l-w+\delta(1+s)-sr}\\
& +\sum_{\delta=0}^{\floor{l/2}}q^{l+\delta(1+s)-sr}\\
& \gg q^{7r(1-s)/8}+q^{3r(1-s)/4}.
\end{align*}
	We conclude that there exists a positive constant $B\in\R$ such that
	for all $R\geq2$, we have 
	\[
	\zeta_{\SL_{2}(\cO_{R})}(s)\geq\sum_{r=2}^{R}\zeta_{S_{r}}^{3}(s)\geq B\sum_{r=2}^{R}q^{(3r/4)(1-s)}.
	\]
	The latter series diverges for $s\leq 1$, and hence the abscissa
	of convergence of $\zeta_{\SL_{2}(\cO)}(s)$ is at least $1$.
\end{proof}
\begin{rem}
Most of the implicit constants in Lemma~\ref{lem:dim-num_theta-asymp} can
be explicitly determined, and given the results in the present paper,
the only gap in our understanding of the representations of $S_{r}$, $r$ even,
is the function $u_{r}(\beta,\theta)$, together with the various constants $c\in\{1,2,3\}$ coming from Lemma~\ref{lem:kernel-type-3}. For $r$ odd, there is in addition the open problem of decomposing 
$\Ind_{V(\beta,\theta)\rho_{l}^{-1}(SC_{l})}^{U(\tau)\rho_{l'}^{-1}(\scentbl)}\hat{\theta}$  
into irreducible constituents, for the various extensions $\hat{\theta}$. 

In any case, our results show that the only remaining thing needed in order to compute the exact abscissa of convergence for $\SL_2(\F_q[[t]])$, $q$ even, is the function $u_{r}(\beta,\theta)$. Very recently, M and Singla \cite{M-Singla} have obtained strong bounds on $u_r(\beta,\theta)$ which imply that the abscissa is $1$.
\end{rem}
\begin{rem}\label{rem:SL2-polygrowth}
Until recently, there was no proof in the literature that $\SL_2(\F_q[[t]])$, $q$ even, has polynomial representation growth, as this group was excluded in \cite{Lubotzky-Martin}. This of course follows from our results in the present section, but was also proved in a general context by Jaikin (see \cite[Theorem 3.2 and Lemma 3.2.6]{Garcia-Jaikin}), around the same time as \cite{Lubotzky-Martin} appeared.
\end{rem}

\section{Proof of Lemma~\ref{lem:kernel-type-3}}\label{sec:proof_of_lemma}

We operate in $\M_2(\cO_{i})$ throughout the proof, so for notational simplicity, we write $w=w(\beta_i)$ and $\delta=\delta(\beta_i)$. Note first that conjugating $\beta_i$ by an element in $G_i$ changes $\centbi=C_{G_i}(\beta_i)$ into a conjugate group, so it does not affect the values of the determinant map. We can therefore choose $\beta_i=\bigl[\begin{smallmatrix}0 & 1 \\ \Delta & \tau\end{smallmatrix}\bigr]$, so that the centraliser has the form
\[
\centbi=
\left\{
\begin{bmatrix}
x & y \\
\Delta y & x+\tau y
\end{bmatrix} \ \middle|\ x,y\in\cO_i
\right\}^{\times}.
\]
Hence, the problem of computing $|\scentbi|$ is reduced to finding the number of solutions in $\cO_i$ to the quadratic equation $x^2+\tau xy+\Delta y^2=1$.
Secondly, adding a scalar to $\beta_i$ does not change $\centbi$, so we are free to replace $\beta_i$ by any element in its twist orbit.

We make some simplifying modifications to the equation. Writing $\tau=\eta t^w$, with $\eta$ a unit and $w\ge 1$, as well as writing $u=\Delta\eta^{-2}$, we can make a change of variables $y\mapsto\eta y$ to rewrite the quadratic equation as
\begin{equation}\label{eq:ker_determinant}
x^2+t^w xy+uy^2=1.
\end{equation}

Next, we use the fact that we may without loss of generality change $\beta_i$ by adding any scalar $\lambda\in\cO_i$ to it. The addition of $\lambda$ changes the determinant by adding $\lambda^2+\lambda\tau$ to it. It follows that $u$ gains an addition of $(\eta^{-1}\lambda)^2+\eta^{-1}\lambda t^w$. Therefore, by choosing
\[
\lambda=\eta\sqrt{u_0+1}
\]
we can make sure that $u_0=1$ after the addition.
Notice that the odd numbered coefficients of $u$ below $w$ are left unchanged, so we still have $u_{2k+1}=0$ for all $k<\delta$ and $u_{2\delta+1}\neq 0$ if $\delta<M$.

%
%

By writing $x=x_0+x_1 t+\cdots+x_{i-1}t^{i-1}$ and $y=y_0+y_1 t+\cdots+y_{i-1}t^{i-1}$, substituting these into \eqref{eq:ker_determinant} and collecting coefficients, we arrive at the following system of equations:
\begin{align}
& x_0^2+u_0 y_0^2=1 & & \label{theta:1} \\
& x_m^2+B(2m)=0 & & \text{for $2\le 2m\le w$} \label{theta:2} \\
& x_m^2+A(2m-w)+B(2m)=0 & & \text{for $w<2m<i$} \label{theta:3} \\
& B(2m-1)=0 & & \text{for $2\le 2m\le w$} \label{theta:4} \\
& A(2m-1-w)+B(2m-1)=0 & & \text{for $w<2m\le i$} \label{theta:5}
\end{align}
where $A(n)=\sum_{j+k=n}x_j y_k$ and $B(n)=\sum_{j+2k=n}u_j y_k^2$. The first three equations correspond to even powers of $t$, and the last two to odd ones.

Our strategy is to find an ``echelon form'' for the equations in order to find the order of degeneracy in the system. In other words, we shall describe an ordering $(z_1,\dots,z_{2i})$ of the variables $x_n$ and $y_n$, as well as a list of equations equivalent to \eqref{theta:1}--\eqref{theta:5}, in which some of the equations are identically zero, and the rest have the form $z_k=\varphi_k(z_1,\dots,z_{k-1})$, a new variable $z_k$ being solved from each equation.

\medskip

\newcounter{systemparts}
\renewcommand{\thesystemparts}{\roman{systemparts})}
\newcommand{\systempart}{ }

\refstepcounter{systemparts}
\emph{\thesystemparts{} Equation \eqref{theta:1}.}\label{eqsystem:1} The first equation is
\[
x_0^2+u_0 y_0^2=1.
\]
We solve for $x_0$ to obtain $x_0=\sqrt{u_0}y_0+1$.

\medskip

\refstepcounter{systemparts}
\emph{\thesystemparts{} Equations \eqref{theta:2}.} This group has the following form:
\[
\left\{
\begin{aligned}
x_1^2 & =u_0 y_1^2+u_2 y_0^2 \\
x_2^2 & =u_0 y_2^2+u_2 y_1^2+u_4 y_0^2 \\
x_3^2 & =u_0 y_3^2+u_2 y_2^2+u_4 y_1^2+u_6 y_0^2 \\
 & \sbox0{\dots}\makebox[\wd0]{\vdots} \\
x_M^2 & =u_0 y_M^2+u_2 y_{M-1}^2+\cdots+u_{2M} y_0^2.
\end{aligned}\right.
\]
There are no identically zero equations. We can solve for $x_m$ in each equation, writing $x_m=\sqrt{u_0}y_m+f_m(y_0,\dots,y_{m-1})$ for $m\in\{1,\dots,M\}$, where each $f_m$ is a linear polynomial.

\medskip

\refstepcounter{systemparts}
\emph{\thesystemparts{} Equations \eqref{theta:3} for $m<w$.}\label{eqsystem:3a} These equations read
\[
\left\{
\begin{aligned}
& x_{M+1}^2  +x_0 y_{2-\varepsilon}+x_1 y_{1-\varepsilon}+(1-\varepsilon)x_2 y_0 
& ={} & u_0 y_{M+1}^2+u_2 y_M^2+\cdots+u_{2M+2} y_0^2 \\
& x_{M+2}^2  +x_0 y_{4-\varepsilon}+\cdots+(1-\varepsilon)x_4 y_0 
& ={} &u_0 y_{M+2}^2+u_2 y_{M+1}^2+\cdots+u_{2M+4}y_0^2 \\
 & &  \sbox0{\dots}\makebox[\wd0]{\vdots} &\\
& x_{w-1}^2 +x_0 y_{w-2}+\cdots+x_{w-2}y_0 
&={}& u_0 y_{w-1}^2+u_2 y_{w-2}^2+\cdots+u_{2w-2}y_0^2.
\end{aligned}\right.
\]
There are no identically zero equations. In the previous part, we have already solved $x_m$ for $m\le M$ in terms of $y_k$ for some $k$, and we continue here from $m=M+1$. Note that in the equation corresponding to $m$, the second-degree term on the left hand side is $x_m^2$, and the first-degree term with the largest index for $x$ is $x_{2m-w}$. Since $m<w$, we know that $m>2m-w$. Hence, when solving for $x_m$ in the $m$-th equation, we can assume that all of $x_j$ in the same equation are already known except for the second-degree term. Working recursively, we conclude that $x_m=\sqrt{u_0}y_m+f_m(y_0,\dots,y_{m-1})$ for $m\in\{M+1,\dots,w-1\}$.

\medskip

\refstepcounter{systemparts}
\emph{\thesystemparts{} Equations \eqref{theta:4}.}\label{eqsystem:4} This group of equations is
\[
\left\{
\begin{aligned}
u_1 y_0^2 & =0 \\
u_1 y_1^2+u_3 y_0^2 & =0 \\
u_1 y_2^2+u_3 y_1^2+u_5 y_0^2 & =0 \\
 & \sbox0{\dots}\makebox[\wd0]{\vdots} \\
u_1 y_{M-1}^2+u_3 y_{M-3}^2+\cdots+u_{2M-1}y_0^2 & =0.
\end{aligned}\right.
\]
This is a homogeneous linear system of $M$ equations for the variables $y_0^2,\dots,y_{M-1}^2$, with coefficient matrix
\[
\begin{bmatrix}
0 & \cdots & 0 & 0 & u_1 \\
\vdots & & 0 & u_1 & u_3 \\
\vdots & & u_1 & u_3 & u_5 \\
0 & & \vdots & \vdots & \vdots \\
u_1 & u_3 & \cdots & \cdots & u_{2M-1}
\end{bmatrix}.
\]
Recall that by the definition of the odd depth $\delta$, we have $u_{2i+1}=0$ for all $i<\delta$. Therefore, the rank of the above matrix is $M-\delta$. There are $\delta$ equations that are identically zero, and from the rest, we can solve $y_0=y_1=\cdots=y_{M-\delta-1}=0$.

\medskip

\refstepcounter{systemparts}
\emph{\thesystemparts{} Equations \eqref{theta:5} for $m<w-\delta$.}\label{eqsystem:5a} These equations read
\[
\left\{
\begin{aligned}
& x_0 y_{1-\varepsilon}+(1-\varepsilon)x_1 y_0 & ={} &
 u_1 y_M^2+\cdots+u_{2M+1}y_0^2 \\
& x_0 y_{3-\varepsilon}+x_1 y_{2-\varepsilon}+
x_2 y_{1-\varepsilon}+(1-\varepsilon)x_3 y_0 & ={} &
 u_1 y_{M+1}^2+\cdots+u_{2M+3}y_0^2 \\
& & \sbox0{\dots}\makebox[\wd0]{\vdots} &  \\
& x_0 y_{w-2\delta-3}+x_1 y_{w-2\delta-4}+\cdots+x_{w-2\delta-3} y_0 & ={} &
 u_1 y_{w-\delta-2}^2+\cdots+u_{2w-2\delta-3}y_0^2.
\end{aligned}\right.
\]
They correspond to $m\in\{M+1,\dots,w-\delta-1\}$. Note that if $\delta=M$, there are no such $m$, and therefore we may assume for these equations that $\delta<M$. Then, in particular, we have $u_{2\delta+1}\neq 0$.

We shall show that there are no identically zero equations, and moreover, from the equation corresponding to $m$, we can solve $y_{m-\delta-1}=0$.
Note that these variables have not been solved for in the previous part, as $m-\delta-1\ge M-\delta$. The proof proceeds by induction on $m$.

Let $m\in\{M+1,\dots,w-\delta-1\}$, and assume that $y_k=0$ when $M-\delta\le k<m-\delta-1$. We consider the equation corresponding to $m$:
\begin{multline*}
x_0 y_{2m-1-w}+\cdots+x_{2m-1-w} y_0 \\
= u_1 y_{m-1}^2+\cdots+u_{2\delta+1}y_{m-\delta-1}^2 + u_{2\delta+3}y_{m-\delta-2}^2+\cdots+u_{2m-1}y_0^2.
\end{multline*}
Since $m\le w-\delta-1$, we have
\[
2m-1-w\le m+(w-\delta-1)-1-w=m-\delta-2.
\]
Hence, by the induction hypothesis, and recalling from the previous part that $y_k=0$ for $k<M-\delta$, the left hand side of the equation becomes identically zero. By the definition of the odd depth, we know that all the variables $u_1,\dots,u_{2\delta-1}$ for odd indices are zero. It follows that most of the right hand side vanishes, too, and we are left with
\[
0=u_{2\delta+1}y_{m-\delta-1}^2.
\]
As $u_{2\delta+1}\neq 0$, we may conclude that $y_{m-\delta-1}=0$, and the induction is complete.

\medskip

\refstepcounter{systemparts}
\emph{\thesystemparts{} Equation \eqref{theta:5} for $m=w-\delta$.}\label{eqsystem:5b} The equation reads
\[
x_0 y_{w-2\delta-1}+x_1 y_{w-2\delta-2}+\cdots+x_{w-2\delta-1} y_0
=u_1 y_{w-\delta-1}^2+\cdots+u_{2w-2\delta-1}y_0^2.
\]
It does not exist if $\delta=M$ and $\varepsilon=0$, as $2m>w$ holds for equations of type~\eqref{theta:5}. We have shown in parts~\ref{eqsystem:4} and \ref{eqsystem:5a} that $y_k=0$ for $k<w-2\delta-1$, so considering that the variables $u_1,\dots,u_{2\delta-1}$ for odd indices are zero, the equation becomes
\[
x_0 y_{w-2\delta-1} = u_{2\delta+1}y_{w-2\delta-1}^2.
\]
Assume first that $\delta<M$, so that $u_{2\delta+1}\neq 0$. As $y_0=0$ by part~\ref{eqsystem:4}, we get from part~\ref{eqsystem:1} that $x_0=1$. Hence, we get exactly two solutions: either $y_{w-2\delta-1}=0$ or $y_{w-2\delta-1}=1/u_{2\delta+1}$.

On the other hand, if $\delta=M$, we must have $\varepsilon=1$, so that $w=2M+1$. The equation then becomes
\[
x_0 y_0 = u_w y_0^2.
\]
Substituting $x_0=\sqrt{u_0}y_0+1$ from part~\ref{eqsystem:1} and collecting coefficients gives
\[
(\sqrt{u_0}+u_w)y_0^2 + y_0 = 0.
\]
Now, if $\sqrt{u_0}+u_w=0$, we have $y_0=0$. Otherwise, we get two distinct solutions: either $y_0=0$ or $y_0=1/(\sqrt{u_0}+u_w)$.

\medskip

\refstepcounter{systemparts}
\emph{\thesystemparts{} Equations \eqref{theta:5} for $w-\delta<m\le w$.} These equations have the general form
\[
x_0 y_{2m-w-1}+x_1 y_{2m-w-2}+\cdots+x_{2m-w-1} y_0
=u_1 y_{m-1}^2+u_3 y_{m-2}^2+\cdots+u_{2m-1}y_0^2.
\]
As in the previous cases, since $u_j=0$ for odd indices $j$ up to $j=2\delta-1$, the first non-zero term on the right hand side is $u_{2\delta+1}y_{m-\delta-1}^2$. As $m>w-\delta$, we see that $2m-w-1>m-\delta-1$, so we can try solving for $y_{2m-w-1}$ from the left hand side. Indeed, noting that $2m-w-1\le w-1$, we can substitute $x_j=\sqrt{u_0}y_j+f_j(y_0,\dots,y_{j-1})$ from parts \ref{eqsystem:1}--\ref{eqsystem:3a} up to $j=2m-w-1$. Collecting all terms containing $y_j$ with $j<2m-w-1$ to the right hand side, the equation becomes
\[
(x_0+\sqrt{u_0}y_0)y_{2m-w-1}=g_m(y_0,\dots,y_{2m-w-2})
\]
for some function $g_m$. Using part~\ref{eqsystem:1} gives $y_{2m-w-1}=g_m(y_0,\dots,y_{2m-w-2})$. This holds for all $m\in\{w-\delta+1,\dots,w\}$.

\medskip

\refstepcounter{systemparts}
\emph{\thesystemparts{} Equation \eqref{theta:3} for $m=w$.}\label{eqsystem:3b} This is
\[
x_w^2+x_0 y_w+x_1 y_{w-1}+\cdots+x_w y_0
=u_0 y_w^2+u_2 y_{w-1}^2+\cdots+u_{2w}y_0^2.
\]
There are different cases. If $\delta<M$, we know from part~\ref{eqsystem:4} that $y_0=0$, and we can solve $x_w=\sqrt{u_0}y_w+f_w(y_0,\dots,y_{w-1})$, as in part~\ref{eqsystem:3a}. On the other hand, if $\delta=M$, it follows from part~\ref{eqsystem:5b} that either $y_0=0$ or $y_0=1/(\sqrt{u_0}+u_w)$. In the latter case, the equation becomes a non-trivial second-degree equation for $x_w$. For each combination of values for $y_1,\dots,y_w$, the equation may therefore not have a solution, but if it does, there are two possibilities for $x_w$.
\medskip

\refstepcounter{systemparts}
\emph{\thesystemparts{} Equations \eqref{theta:3} and \eqref{theta:5} for $m>w$.} For each $m>w$, we have the equations
\[
\left\{
\begin{aligned}
& x_m^2+ x_0 y_{2m-w}+x_1 y_{2m-w-1}+\cdots+x_{2m-w}y_0 
& ={} & u_0 y_m^2+\cdots+u_{2m}y_0^2
\\
& x_0 y_{2m-w-1}+x_1 y_{2m-w-2}+\cdots+x_{2m-w-1} y_0 
& ={} & u_1 y_{m-1}^2+\cdots+u_{2m-1}y_0^2.
\end{aligned}\right .
\]
The first one corresponds to type~\eqref{theta:3} and the second to type~\eqref{theta:5}. Since $m>w$, we have $2m-w>m$ and $2m-w-1>m-1$, so we can solve each equation from the left hand side. There are two cases.

Suppose first that $x_0\neq 0$. For each $m>w$, we can solve from the equation of type~\eqref{theta:3}
\[
y_{2m-w}=h_{2m-w}(y_0,\dots,y_{2m-w-1},x_w,\dots,x_{2m-w})
\]
for some function $h_{2m-w}$, using the fact that $x_j$ were solved in terms of $y_0,\dots,y_j$ up to $j=w-1$ in parts~\ref{eqsystem:1}--\ref{eqsystem:3a} above. Similarly, from equations of type~\eqref{theta:5}, we solve
\[
y_{2m-w-1}=h_{2m-w-1}(y_0,\dots,y_{2m-w-2},x_w,\dots,x_{2m-w-1}).
\]
Note that these solved variables are never the same for any choices of $m$ because their indices differ in parity. Therefore, we have solved a new variable from each type of equation for every $m>w$.

On the other hand, if $x_0=0$, then it follows from part~\ref{eqsystem:1} that $y_0\neq 0$, and we can write $x_{2m-w}=h_{2m-w}(y_0,\dots,y_{2m-w},x_w,\dots,x_{2m-w-1})$ for type~\eqref{theta:3} and $x_{2m-w-1}=h_{2m-w-1}(y_0,\dots,y_{2m-w-1},x_w,\dots,x_{2m-w-2})$ for type~\eqref{theta:5}. All the variables are again distinct.

\medskip

\emph{Conclusion.} There are altogether $2i$ variables and $i$ equations. By part~\ref{eqsystem:4}, there are $\delta$ equations that are identically zero, and we have shown that all the other equations can be solved for one new variable. This leaves $i+\delta$ many free variables.

Furthermore, if $\delta<M$, it follows from part~\ref{eqsystem:5b} that the variable $y_{w-2\delta-1}$ can be solved in two distinct ways, but all the other solved variables have unique solutions. Suppose then that $\delta=M$. Then the variable $y_{w-2\delta-1}=y_0$ has in some cases two distinct solutions, one of them being always $y_0=0$. Now, if $y_0=0$, all other variables have unique solutions. On the other hand, if $y_0\neq 0$, part~\ref{eqsystem:3b} shows that in some cases the variable $x_w$ may have two distinct solutions.

Finally, the total number of solutions reaches its maximum when $\delta=M$. Then there are $q^{i+\delta}$ solutions where $y_0=0$, and potentially $2q^{i+\delta}$ solutions where $y_0\neq 0$. This gives altogether at most $3q^{i+\delta}$ solutions, which proves the claim.

\bibliographystyle{alex} \bibliography{alex}

\end{document}